\RequirePackage{fix-cm}
\documentclass[11pt,reqno]{amsart}                  
\usepackage{graphicx}
\usepackage{mathptmx}      
\usepackage{algorithm,algorithmicx}
\usepackage{algpseudocode}
\usepackage{bbm}
\usepackage[mathlines,pagewise,left]{lineno}
\usepackage{amssymb}
\usepackage{color}
\usepackage{amsmath}
\usepackage[export]{adjustbox}
\usepackage{multirow}
\usepackage[numbers]{natbib}
\usepackage{pdflscape}
\usepackage{units}
\usepackage[colorinlistoftodos,prependcaption,textsize=small]{todonotes}
\usepackage{agor}
\usepackage{hyperref}
\usepackage{caption}
\usepackage[font=scriptsize]{subfig}
\usepackage{bcol}

\usepackage{amsthm}
\newtheorem{proposition}{Proposition}

\newtheorem{theorem}{Theorem}

\theoremstyle{definition}
\newtheorem{definition}{Definition}
\theoremstyle{remark}
\newtheorem{remark}{Remark}
\newtheorem{example}{Example}

\allowdisplaybreaks

\begin{document}

\title[Supermodularity and quadratic optimization with indicators]{
	Supermodularity and valid inequalities for quadratic optimization with indicators 
}

\author{Alper Atamt\"urk and Andr{\'e}s G{\'o}mez
}
\thanks{ 
	\noindent A. Atamt\"urk: Department of Industrial Engineering \& Operations Research, University of California Berkeley, CA 94720. \texttt{atamturk@berkeley.edu}\\
	A. G\'{o}mez:  Department of Industrial \& Systems Engineering, Viterbi School of Engineering, University of Southern California, CA 90089. \texttt{gomezand@usc.edu}}

\maketitle

\begin{abstract}
We study the minimization of a rank-one quadratic with indicators and show that the underlying set function obtained by projecting out the continuous variables is supermodular. Although supermodular minimization is, in general, difficult, the specific set function for the rank-one quadratic can be  minimized in linear time. We show that the convex hull of the epigraph of the quadratic can be obtaining from inequalities for the underlying supermodular set function by lifting them into nonlinear inequalities in the original space of variables. Explicit forms of the convex-hull description are given, both in the original space of variables and in an extended formulation via conic quadratic-representable inequalities, along with a polynomial separation algorithm. Computational experiments indicate that the lifted supermodular inequalities in conic quadratic form are quite effective in reducing the integrality gap for quadratic optimization with indicators. \\


\end{abstract}

\begin{center} December 2020 \end{center}

\BCOLReport{20.03}

\section{Introduction}\label{sec:intro}
Consider the convex quadratic optimization problem with indicators 
\begin{equation}\label{eq:QOI}
\min\! \big\{a'x\!+\!b'y\!+\!y'Qy\! : \ y_i(1\!-\!x_i)=0, i=1,\ldots,n; \ (x,y) \in  \{0,1\}^n \! \times \R_+^n\big\}
\end{equation} 
where $a, b \in \R^n$ and $Q\in \R^{n\times n}$  is a symmetric positive semi-definite matrix. For each $i=1, \ldots, n$, the binary variable $x_i$, along with the complementary constraint $y_i (1-x_i)=0$, indicates whether $y_i$ may take positive values. Problem \eqref{eq:QOI} arises in numerous practical applications, including
 portfolio optimization \cite{B:miqp}, signal/image denoising \cite{atamturk2018signal,bach2019submodular}, 
best subset selection \cite{bertsimas2015or,cozad2015combined}, and unit commitment \cite{fgl:unit-q}.

Constructing strong convex relaxations for non-convex optimization problems is critical in devising effective solution approaches for them.
Natural convex relaxations of \eqref{eq:QOI}, where the complementary constraints $y_i(1-x_i)=0$ are linearized using the so-called ``big-$M$" constraints $y_i\leq Mx_i$, are known to be weak 
\cite[e.g.,][]{manzour2019integer}. Therefore, there is a increasing effort in the literature to better understand and describe the epigraph of quadratic functions with indicator variables. \citet{DL:ipco-qp-ind} describe lifted linear inequalities for \eqref{eq:QOI} from its continuous quadratic optimization counterpart 
over bounded variables.
\citet{BM:conv-noncov} give a characterization linear inequalities obtained by strengthening gradient inequalities of a convex objective function over a non-convex set. 

The majority of the work toward constructing strong relaxations of \eqref{eq:QOI} is based on the \emph{perspective reformulation} \cite{akturk2009strong,BLTW:mp-indicator,Frangioni2006,Gunluk2010,HBCO:on-off,Mahajan2017,Wu2017,xie2018ccp}. 
The perspective reformulation, which may be seen as a consequence of the convexifications based on disjunctive programming derived in \cite{Ceria1999}, is based on strengthening the epigraph of a univariate convex quadratic function $y_i^2\leq t$ by using its perspective $y_i^2/x_i\leq t$.  The perspective strengthening can be applied to a general convex quadratic $y'Qy$, by 
writing it as $y' (Q-D) y + y'Dy$ for a diagonal matrix $D \succ 0$ and $Q-D\succeq 0$, 
and simply reformulating each separable quadratic term $D_{ii}y_i^2$ as $D_{ii}y_i^2/x_i$~\cite{dong2015regularization,Frangioni2007,zheng2014improving}. While this approach is effective when $Q$ is strongly diagonal dominant, it is ineffective otherwise, or inapplicable when $Q$ is not full-rank as no such $D$ exists.  

To address the limitations of the perspective reformulation, a recent stream of research focuses on constructing strong relaxations of the epigraphs of simple but multi-variable quadratic functions. \citet{Jeon2017} use linear lifting to construct valid inequalities for the epigraphs of two-variable quadratic functions. \citet{frangioni2018decompositions} use extended formulations based on disjunctive programming to derive stronger relaxations of the epigraph of two-variable functions. They study heuristics and semi-definite programming (SDP) approaches to extract from $Q$ such two-variable terms. The disjunctive approach results in a substantial increase in the size of the formulations, which limits its use to small instances. \citet{atamturk2018strong} describe the convex hull of the epigraph of the two-variable quadratic function $(y_1-y_2)^2\leq t$ in the original space of variables, and \citet{atamturk2018signal} generalize this result to convex two-variable quadratic functions $a_1y_1^2-2y_1y_2+a_2y_2^2\leq t$ and show how to optimally decompose an $M$-matrix (psd with non-positive off-diagonals)
$Q$ into such two-variable terms; the results indicate that such formulations considerably improve the convex relaxations when $Q$ is an $M$-matrix, but the relaxation quality degrades when $Q$ has positive off-diagonal entries. \citet{hga:2x2} give SDP formulations for \eqref{eq:QOI} based on convex-hull descriptions of the $2x2$ case. These SDP formulations require $O(n^2)$ additional variables and constraints, which may not scale to large problems. \citet{atamturk2019rank} give the convex hull description of a rank-one function with free continuous variables, and propose an SDP formulation to tackle quadratic optimization problems with free variables arising in sparse regression. \citet{wei2020ideal,wei2020convexification} extend those results, deriving ideal formulations for rank-one functions with \emph{arbitrary} constraints on the indicator variables $x$. These formulations are shown to be effective in sparse regression problems; however as they do not account for the non-negativity constraints {on the continuous variables}, they are weak for~\eqref{eq:QOI}. 
The rank-one quadratic set studied in this paper addresses this gap and properly generalizes the perspective strengthening of a univariate quadratic to higher dimensions.

\ignore{
\todo{A natural question is how does rank-one formulations proposed here compare with Han et al. as they are both for (1)\\
AG: In theory, they coincide on the rank-one 2x2 case, and then generalize in different ways.\\
In practice, it depends. On the instances used in computations here, which are low rank, 2x2 would not do anything. In instances like the ones in the 2x2 paper, which are well-conditioned, it is unclear as it unclear how to extract rank ones. 
}
}

In the context of discrete optimization, submodularity/supermodularity plays a critical role in the design of algorithms \cite{fujishige2005submodular,grotschel1981ellipsoid,orlin2009faster} and in constructing convex relaxations to discrete problems \cite{ahmed2011maximizing,atamturk2015supermodular,AN:submodular,nemhauser1978analysis,shi2020sequence,wu2015maximizing,yu2017maximizing,yu2017polyhedral,yu2020polyhedral}. Exploiting submodularity in settings involving continuous variables as well typically require specialized arguments, e.g., see \cite{atamturk2017path,kilincc2019joint,tjandraatmadja2020convex}. A notable exception is  \citet{wolsey1989submodularity}, presenting a systematic approach for exploiting submodularity in fixed-charge network problems. As submodularity arises in combinatorial optimization, where the convex hulls of the sets under study are polyhedral, there are few papers utilizing submodularity to describe non-polyhedral convex hulls \cite{atamturk2020submodularity}, and those sets typically involve some degree of separability between continuous and discrete variables. In this paper, we show how to generalize the valid inequalities proposed in \cite{wolsey1989submodularity} to convexify non-polyhedral sets, where the continuous variables are linked with the binary variables via indicator constraints.

\subsection*{Contributions}
Here,
we study the mixed-integer epigraph of a rank-one quadratic function with indicator variables and non-negative continuous variables:
\small
\begin{align*}
X&=\bigg\{(x,y,t)\in \{0,1\}^N\times \R_+^N\times \R_+: \bigg(\sum_{i\in N^+}y_i-\sum_{i\in N^-}y_i\bigg)^2\leq t,\; y_i(1-x_i)=0, \ i\in N\bigg\}, 
\end{align*}\normalsize
where $(N^+,N^-)$ is a partition of $N := \{1, \ldots,n\}$. 
Observe that any rank-one quadratic of the form $\left(c'y\right)^2\leq t$ with $c_i\neq 0$ for all $i\in N$
can be written as in $X$ by scaling the continuous variables. If all coefficients of $c$ are of the same sign, then either $N^+=\emptyset$ or $N^-=\emptyset$, and $X$  reduces to the simpler form 
\small
$$X_+=\bigg\{(x,y,t)\in \{0,1\}^N\times \R_+^N\times \R_+: \bigg(\sum_{i\in N}y_i\bigg)^2\leq t,\; y_i(1-x_i)=0, \  i\in N\bigg\} \cdot$$
\normalsize
To the best of our knowledge, the convex hull structure of $X$ or $X_+$ has not been studied before.  Interestingly,
 optimization of a linear function over $X$ can be done in polynomial time (\S~\ref{sec:facetMin}).

\ignore{ {\color{red}Understanding $X$ is a critical building block for tackling \eqref{eq:QOI}, as any quadratic function $y'Qy$ can be (easily) written as a sum of rank-one quadratic functions (for example, by using a Cholesky decomposition of $Q$).} \todo{Such decompositions normally require a change of variables, which are inconsistent with indicator constraints. We may need to expand this, perhaps using formulations in the computational section so show how the rank-one formulation can be effectively used for the general model to better motivate the study in the paper early on. \\
}}

Our motivation for studying $X$ stems from constructing strong convex relaxations for problem \eqref{eq:QOI} by writing the convex quadratic $y'Qy$ as a sum of rank-one quadratics. Especially in large-scale applications, it is 
effective to state $Q$ as a sum of a low-rank matrix and a diagonal matrix.
Specifically, suppose that $Q=FF'+D$, where $F\in \R^{n\times r}$ and $D\in \R^{n\times n}$ is a (possibly empty) nonnegative diagonal matrix. Such decompositions can be constructed in numerous ways, including singular-value decomposition, Cholesky decomposition, or via factor models.
Letting $F_j$ denote the $j$-th column of $F$, adding  auxiliary variables $t\in \R^r$, $j=1,\dots,r$, and using the perspective reformulation, problem \eqref{eq:QOI} can be cast as
\begin{subequations}\label{eq:QOI2}
\begin{align}
\min_{x,y,t}\;&\sum_{j=1}^rt_j+\sum_{i=1}^nD_{ii}\frac{y_i^2}{x_i}\\
\text{s.t.}\;
&(F_j'y)^2\leq t_j, \ j =1, \ldots, r\\
&(x,y)\in \{0,1\}^N\times \R_+^n.
\end{align}
\end{subequations}
Formulation \eqref{eq:QOI2} arises naturally, for example, in portfolio risk minimization \cite{B:miqp},  where the covariance matrix $Q$ is the sum of a low-rank factor covariance matrix and an idiosyncratic (diagonal) variance matrix. 
When the entries of the diagonal matrix $D$ are small, the perspective reformulation is not effective in strengthening the formulation.
However, noting that $(x,F_j\circ y,t_j)\in X$, where $(F_j\circ y)_i=F_{ij}y_i$, for each $j=1,\ldots,r$, one can employ strong relaxations based on the rank-one quadratic with indicators, $X$. Our approach for decomposing $y'Qy$ into a sum of rank-one quadratics and utilizing strong relaxations of epigraphs of rank-one quadratics is analogous to employing cuts separately from individual rows of a constraint matrix $Ax \le b$ in mixed-integer linear programming.  

\ignore{
It is unclear how to use use commonly used techniques to derive good relaxations of $X$. Disjunctive programming formulations for $X$ would involve an exponential number of additional variables for $X$, which are neither practical nor useful for obtaining insights into the convex hull of $X$, $\conv(X)$; projecting out the additional variables in closed form, using for example Fourier-Motzkin elimination, 
is often a difficult process. Since the function under study is rank-one, it cannot be decomposed into sums of single- or two- variable terms. Moreover, convexifications involving free continuous relaxations are weak for $X$ and do not yield useful insights into $\conv(X)$.
}

In this paper, we present a generic framework for obtaining valid inequalities for mixed-integer nonlinear optimization problems by exploiting 
 supermodularity of the underlying set function. To do so, we project out the continuous variables and derive valid inequalities for the corresponding pure integer set and then lift these inequalities to the space of continuous variables as in \citet{nguyen2018deriving,richard2010lifting}. It turns out that for the rank-one quadratic with indicators, the corresponding set function is supermodular and holds much of the structure of $X$.
The \textit{lifted supermodular inequalities} derived in this paper are \emph{nonlinear} in both the continuous and discrete variables. 

We show that this approach encompasses several previously known convexifications for quadratic optimization with indicator variables.  
Moreover, the well-known inequalities in the mixed-integer \emph{linear} optimization literature given \cite{wolsey1989submodularity}, which include flow cover inequalities as a special case, can also be obtained via the lifted supermodular inequalities. 

Finally, and more importantly, we show that the lifted supermodular inequalities and bound constraints are sufficient to describe $\clconv(X)$.
Such convex hull descriptions of high-dimensional nonlinear sets are rare in the literature.
In particular, we give a characterization in the original space of variables. This description is defined by a piecewise valid function with exponentially many pieces; therefore, it cannot be used by the convex optimization solvers directly. To overcome this difficulty, we also give a conic quadratic representable description in an extended space, with exponentially many valid conic quadratic inequalities, along with a polynomial-time separation algorithm. 

The rank-one quadratic sets $X$ and $X_+$  appear very similar to their relaxation
\[
X_f =\bigg\{(x,y,t)\in \{0,1\}^N\times \R^N\times \R: \bigg(\sum_{i\in N}y_i\bigg)^2\leq t,\ y_i(1-x_i)=0, \  i\in N\bigg\},
\] \normalsize
where the non-negativity constraints on the continuous variables $y \ge 0$ are dropped. However, while only one additional inequality 
$\frac{\left ( \sum_{i\in N} y_i\right )^2}{\sum_{i \in N} x_i} \leq t$ is needed to describe $\clconv(X_f)$ \citep{atamturk2019rank} , the convex hulls of $X$ and $X_+$ are substantially more complicated and rich. Indeed, $\clconv(X_f)$ provides a weak relaxation for $\clconv(X_+)$, 
as illustrated in the next example.


\begin{example}\label{ex:freeVsPos}
Consider set $X_+$ with $n=3$. For the relaxation $X_f$, the closure of the convex hull is described by $0\leq x \leq 1$ and inequality $t\geq \frac{(y_1+y_2+y_3)^2}{\min\{1,x_1+x_2+x_3\}}$. 
Figure~\ref{fig:example} (A) depicts this inequality as a function of $(x_1,y_1)$ for $x_2=0.6$, $x_3=0.3$, $y_2=0.5$, and $y_3=0.2$ (fixed).
In Proposition~\ref{prop:validPos}, we give the function $f$ describing $\clconv(X_+)$. 
 Figure~\ref{fig:example} (B) depicts  $f(x,y)$ (truncated at 5) as a function of $(x_1,y_1)$ when other variables are fixed as before. 
 	We find that $\clconv(X_f)$ is a very weak relaxation of $\clconv(X_+)$ for low values of $x_1$. For example, for $x=0.01$ and $y=1$,  we find that $\frac{(1+0.5+0.2)^2}{0.01+0.6+0.3}\approx 3.18$, whereas $f(x,y)\approx 100.55$. The computation of $f$ for this example is described after Proposition~\ref{prop:validPos}.
\qed
\end{example}

\begin{figure}[!h]	
		\subfloat[$\clconv(X_f)$]{\includegraphics[width=0.49\textwidth,trim={12.7cm 5.5cm 12.8cm 5.5cm},clip]{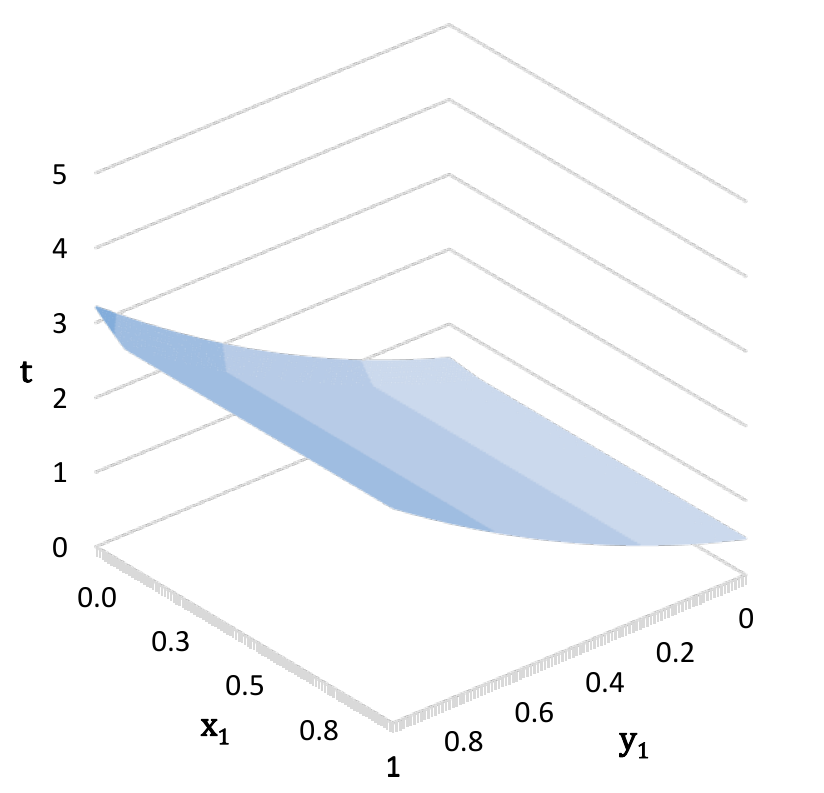}}\ \hfill
		\subfloat[$\clconv(X_+)$ (truncated)]{\includegraphics[width=0.49\textwidth,trim={12.7cm 5.5cm 12.8cm 5.5cm},clip]{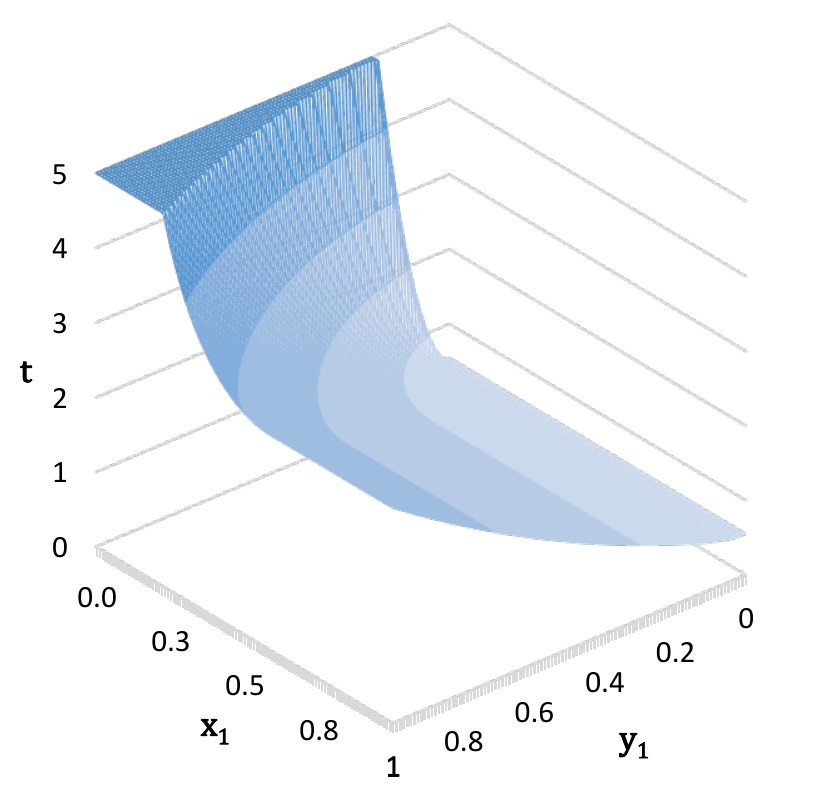}}
		
		\caption{Comparison of $\clconv(X_f)$ and $\clconv(X_+)$. Variables $x_2=0.6$, $x_3=0.3$, $y_2=0.5$, and $y_3=0.2$ are fixed.}
		\label{fig:example}
\end{figure}

\ignore{In addition to Theorem~\ref{theo:hullGeneral} and \ref{theo:hullPosiitve} above, we also provide in this paper a description of the convex hulls in the original space of variables via piecewise-valid inequalities, and give a polynomial time algorithm for solving the separation problem for inequalities \eqref{eq:extendedGenIntroPlus}, \eqref{eq:extendedGenIntroMinus} and \eqref{eq:extendedPosIntroPlus}. }

\subsection*{Outline} The rest of the paper is organized as follows. In \S\ref{sec:preliminaries} we review the valid inequalities for supermodular set functions and present the general form of the lifted supermodular inequalities. In \S\ref{sec:previous} we re-derive known ideal formulations in the literature for quadratic optimization using the lifted supermodular inequalities. In \S\ref{sec:strength} we show that the lifted supermodular inequalities are sufficient to  describe the convex hull of $X$. In \S\ref{sec:derivation} we provide the explicit form of the lifted supermodular inequalities for $X$, both in the original space of variables and in conic quadratic representable form in an extended space, and discuss the separation problem. In \S\ref{sec:computations} we present computational results, and in \S\ref{sec:conclusions} we conclude the paper.

\subsection*{Notation} For a set $S\subseteq N$, define $x_S$ as the indicator vector of $S$, and define $S_x$ as the support set of a vector $x\in \{0,1\}^N$. By abusing notation, we use $x_S$ and $S_x$ interchangeably, e.g., given a set function $g:2^N\to \R$, we may equivalently write $g(S)$ or $g(x_S)$. To simplify the notation, given $i\in N$ and $S\subseteq N$, we write $S\cup i$ instead of $S\cup\{i\}$ and $S\setminus i$ instead of $S\setminus\{i\}$. For a set $Y\subseteq \R^N$, let $\conv(Y)$ denote the convex hull of $Y$ and \clconv(Y) denote its closure. We adopt the convention that $a/0=\infty$ if $a>0$ and $a/0=0$ if $a=0$.
For a $a\in \R$, let $a_+=\max\{a,0\}$. For a vector $c\in \R^N$ and a set $S\subseteq N$, we let $c(S)=\sum_{i\in S}c_i$ and $\max_c(S)=\max_{i\in S}c_i$ (by convention, $\max_c(\emptyset)=0$).

\section{Preliminaries}\label{sec:preliminaries}
In this section we cover a few preliminary results for the paper and, at the end, give the general form of the lifted supermodular inequalities (Theorem~\ref{theo:liftedSupermodular}).
\subsection{Supermodularity and valid inequalities}
A set function $g:2^N\to \R$ is \emph{supermodular} if
\begin{align*}
\rho(i,S)\leq \rho (i,T) \quad\forall i\in N\text{ and } \forall S\subseteq T\subseteq N\setminus i,
\end{align*}
where $\rho(i,S)=g(S\cup i)-g(S)$ is the increment function. 

\begin{proposition}[\citet{nemhauser1978analysis}]\label{prop:supermodularity}
If $g$ is a supermodular function, then
\begin{enumerate}
	\item $g(T)\geq g(S)+\sum\limits_{i\in T\setminus S}\rho(i,S)-\sum\limits_{i\in S\setminus T}\rho(i,N\setminus i )$ for all $S,T\subseteq N$
	\item $g(T)\geq g(S)+\sum\limits_{i\in T\setminus S}\rho(i,\emptyset)-\sum\limits_{i\in S\setminus T}\rho(i,S\setminus i )$ for all $S,T\subseteq N$.
\end{enumerate}
\end{proposition}

As a direct consequence of Proposition~\ref{prop:supermodularity}, one can construct valid inequalities for the epigraph of a supermodular function $g$, i.e., 
$$Z=\left\{(x,t)\in \{0,1\}^N\times \R: g(x)\leq t\right\}.$$
Specifically, for any $S\subseteq N$, the linear supermodular inequalities \cite{nemhauser1988}
\begin{subequations}\label{eq:supermodularDiscrete}
\begin{align}
g(S)+\sum\limits_{i\in N\setminus S}\rho(i,S)x_i-\sum\limits_{i\in S}\rho(i,N\setminus i )(1-x_i)&\leq t,\text{ and}\label{eq:supermodularDiscrete1}\\
g(S)+\sum\limits_{i\in N\setminus S}\rho(i,\emptyset)x_i-\sum\limits_{i\in S}\rho(i,S\setminus i )(1-x_i)&\leq t\label{eq:supermodularDiscrete2}
\end{align}
\end{subequations}
are valid for $Z$. 

\subsection{Lifted supermodular inequalities}\label{sec:lifting}
We now describe a family of lifted supermodular inequalities, using a lifting approach similar to the ones used in \cite{gomez2018submodularity,richard2010lifting}. 
Let $h:\{0,1\}^N\times \R^N\to \R\cup\{\infty\}$ be a function defined over a mixed 0-1 domain and consider its epigraph 
\[
H=\left\{(x,y,t)\in \{0,1\}^N\times \R^N\times \R: h(x,y)\leq t\right\}.
\]
Observe that $H$ allows for arbitrary constraints, which can be encoded via function $h$. For example, 
nonnegativity and complementary constraints can be included by letting $h(x,y)=\infty$ whenever $y_i<0$ or $y_i(1-x_i)\neq 0$ for some $i\in N$. 

For $\alpha\in \R^N$, define the set function $g_\alpha:\{0,1\}^N\to \R\cup \{\infty,-\infty\}$ as
\begin{equation}\label{eq:defG}
g_\alpha(x)=\min_{y\in \R^N}-\alpha'y+h(x,y), 
\end{equation} 
and let $B\subseteq \R^N$ be the set of values of $\alpha$ for which problem \eqref{eq:defG} is bounded, i.e., 
\begin{align*}
B=\left\{\alpha\in \R^N: |g_\alpha(x)|<\infty, \ \forall x\in \{0,1\}^N \right\}.
\end{align*}
Although supermodularity is defined for set functions only, we propose in Definition~\ref{def:supermodular} below an extension for functions involving continuous variables as well.
\begin{definition}\label{def:supermodular}
Function $h$ is supermodular if the set function $g_\alpha$ defined in \eqref{eq:defG} is supermodular for all $\alpha\in B$. 
\end{definition}
\begin{remark}
	Suppose that $h$ does not depend on the continuous variables $y$, i.e., $h(x,y)=g(x)$. In this case problem \eqref{eq:defG} is unbounded unless $\alpha=0$, i.e., $B=\{0\}$, and we find that $h(x,y)$ is supermodular if and only if $g_0(x)=g(x)$ is supermodular. Thus, Definition~\ref{def:supermodular} includes the usual definition of supermodularity for set functions as a special case. \qed
\end{remark}

\begin{proposition}\label{prop:supermodularValid}
If function $h$ is supermodular, then for any $\alpha\in B$ and $S\subseteq N$, the inequalities
\begin{subequations} \label{eq:supermodularLinear}
\begin{align}
\alpha'y+g_\alpha(S)+\sum\limits_{i\in N\setminus S}\rho_\alpha(i,S)x_i-\sum\limits_{i\in S}\rho_\alpha(i,N\setminus i )(1-x_i)&\leq t,\text{ and}\label{eq:supermodularLinear1}\\
\alpha'y+g_\alpha(S)+\sum\limits_{i\in N\setminus S}\rho_\alpha(i,\emptyset)x_i-\sum\limits_{i\in S}\rho_\alpha(i,S\setminus i )(1-x_i)&\leq t\label{eq:supermodularLinear2}
\end{align}
\end{subequations}
are valid for $H$, where $\rho_\alpha(i,S)=g_\alpha(S\cup i)-g_\alpha(S)$.
\end{proposition}
\begin{proof}
For any $\alpha\in B$, $S\subseteq N$, and $(x,y,t)\in H$, we find 
$$t-\alpha'y\geq h(x,y)-\alpha'y\geq g_\alpha(x)\geq g_\alpha(S)+\sum\limits_{i\in N\setminus S}\rho_\alpha(i,S)x_i-\sum\limits_{i\in S}\rho_\alpha(i,N\setminus i )(1-x_i),$$ 
where the first inequality follows directly from the definition of $H$, the second inequality follows by minimizing $h(y)-\alpha'y$ with respect to $y$, and the third inequality follows from the validity of \eqref{eq:supermodularDiscrete1}. Thus, by adding $\alpha'y$ on both sides, we find that inequality \eqref{eq:supermodularLinear1} is valid. The validity of \eqref{eq:supermodularLinear2} is proven identically. 
\end{proof}

Since inequalities \eqref{eq:supermodularLinear} are valid for any $\alpha\in B$, one can obtain stronger valid inequalities by optimally choosing vector $\alpha$. 
\begin{theorem}[Lifted supermodular inequalities]\label{theo:liftedSupermodular}
	If  $h$ is supermodular, then for any $S\subseteq N$, the lifted supermodular inequalities
	\begin{subequations}\label{eq:supermodular}
	\begin{align}
	\max_{\alpha\in B}\;g_\alpha(S)+\sum\limits_{i\in N\setminus S}\rho_\alpha(i,S)x_i-\sum\limits_{i\in S}\rho_\alpha(i,N\setminus\{i\})(1-x_i)+\alpha'y&\leq t,\text{ and}\label{eq:supermodular1}\\
	\max_{\alpha\in B}\;g_\alpha(S)+\sum\limits_{i\in N\setminus S}\rho_\alpha(i,\emptyset)x_i-\sum\limits_{i\in S}\rho_\alpha(i,S\setminus\{i\})(1-x_i)+\alpha'y&\leq t\label{eq:supermodular2}
	\end{align}
	\end{subequations}
	are valid for $H$. 
\end{theorem}
Observe that while inequalities \eqref{eq:supermodularLinear} are linear, inequalities \eqref{eq:supermodular} are \textit{nonlinear} in $x$ and $y$. Moreover, each inequality \eqref{eq:supermodular} is convex since it is defined as a supremum of linear inequalities.
In addition, if the base supermodular inequalities \eqref{eq:supermodularDiscrete} are strong for the convex hull of epi $g_\alpha$, then the lifted supermodular inequalities \eqref{eq:supermodular} are strong for $H$ as well, as formalized next. Given $\alpha\in B$, define $$G_\alpha=\left\{(x,t)\in \{0,1\}^N\times \R: g_\alpha(x)\leq t\right\}.$$ Note that $\conv(G_\alpha)$ is a polyhedron. Theorem~\ref{theo:hullSupermodular} below is a direct consequence of Theorem 1 in \cite{richard2010lifting}.  

\begin{theorem}[\citep{richard2010lifting}] \label{theo:hullSupermodular}
	If inequalities \eqref{eq:supermodularDiscrete} and bound constraints $0\leq x\leq 1$ describe $\conv(G_\alpha)$ for all $\alpha\in B$, then the lifted supermodular inequalities \eqref{eq:supermodular} and bound constraints $0\leq x\leq 1$ describe $\clconv(H)$. 
\end{theorem}

Although Definition~\ref{def:supermodular} may appear to be too restrictive to arise in practice, we show in \S\ref{sec:networks} that supermodular functions are in fact widespread in a class of well-studied problems in mixed-integer \emph{linear} optimization. In \S\ref{sec:previous} we show that several existing results for quadratic optimization with indicators can be obtained as lifted supermodular inequalities. 
Perhaps, more surprisingly, for the rank-one quadratic with indicators
$$h(x,y)=\begin{cases}
\big(y(N^+) - y(N^-) \big)^2 &\text{if }y\geq 0\text{ and }y_i(1-x_i)=0, \ \forall i\in N^+ \cup N^-\\
\infty &\text{otherwise,}
\end{cases}
$$
we show in \S\ref{sec:strength} that
conditions in Definition~\ref{def:supermodular} and Theorem~\ref{theo:hullSupermodular} are satisfied as well.

\subsection{Supermodular inequalities and fixed-charge networks}\label{sec:networks}

\ignore{Although we focused on quadratic functions in this paper, the supermodular inequalities \eqref{eq:supermodular1}-\eqref{eq:supermodular2} can be used to derive valid inequalities for any set of the form 
$$H=\left\{(x,y,t)\in \{0,1\}^N\times \R_+^N\times \R: h(x,y)\leq t  \right\},$$ 
provided that the function 
\begin{equation*}
	g_\alpha(x)=\min_{y\in \R_+^N}-\alpha'y+h(x,y)\end{equation*}
is supermodular for all $\alpha\in B$. In particular,}
Given  $b\in \R$, $u\in \R_+^N$, and a partition $N=N^+\cup N^-\cup A^+\cup A^-$,
define for all $x\in \{0,1\}^N$ the \emph{fixed-charge network set}
\begin{align*}FC(x)=\Big\{y\in \R_+^N: \ &  y(N^+)+y( A^+)-y(A^-)-y(N^-)\leq b,\; y_i\leq u_i, \; i\in N,\\
	& y_i(1-x_i)=0, \ i\in N^+,\; y_ix_i=0, \ i\in N^- \Big\} \cdot
\end{align*} 
\citet{wolsey1989submodularity} uses $FC(x)$ to describe network structures arising in flow problems with fixed charges on the arcs: $N^+$ denotes the incoming arcs into a given subgraph, $N^-$ denotes the outgoing arcs, and whereas $A^+\cup A^-$ denotes the internal arcs in the subgraph,
and $b$ represents the supply/demand of the subgraph. 
Finally, define $$h(x,y)=\begin{cases}0 & \text{if }y\in FC(x)\\
	\infty & \text{otherwise.}\end{cases}$$

\begin{proposition}[\citep{wolsey1989submodularity}]\label{prop:network}For any $\alpha\in \R^N$, the function $$v_\alpha(x)=\max_{y\in \R_+^N}\alpha'y-h(x,y)$$ is submodular.
\end{proposition}
It follows that the function $g_\alpha(x)=-v_\alpha(x)=\min_{y\in \R_+^N}-\alpha'y+h(x,y)$ is supermodular, and inequalities \eqref{eq:supermodularLinear} and \eqref{eq:supermodular} are valid. Moreover, \citet{wolsey1989submodularity} shows that the linear supermodular inequalities \eqref{eq:supermodularLinear} with $\alpha\in \{-1,0,1\}^N$ include as special cases well-known inequalities for mixed-integer linear optimization such as flow-cover inequalities \cite{PVRW:fc,van1986valid} and inequalities for capacitated lot-sizing \cite{AM:ls-poly,pochet1988valid}; several other classes for fixed-charge network flow problems are special cases as well \cite{A:fp,ANS:avub,atamturk2017path}. Therefore, the inequalities presented in this paper can be interpreted as nonlinear generalizations of the aforementioned inequalities.

\section{Previous results as lifted supermodular inequalities}\label{sec:previous}

In order to illustrate the approach, in this section, we show how existing results for quadratic optimization with indicators can be derived using the lifted supermodular inequalities \eqref{eq:supermodular}.

\subsection{The single-variable case}\label{sec:perspective}

Consider, first, the single-variable case
$$X^1=\left\{(x,y,t)\in \{0,1\}\times \R_+\times \R: y^2\leq t,\; y(1-x)=0 \right\}$$  
for which $\clconv(X^1)$ is given by the perspective reformulation \cite{akturk2009strong,Ceria1999,Frangioni2006,Gunluk2010}:
$$\clconv(X^1)=\left\{(x,y,t)\in [0,1]\times \R_+\times \R\cup \infty: \frac{y^2}{x}\leq t \right\}.$$ 
We now derive the perspective reformulation as a special case, in fact, using a modular inequality. Note have that $g_\alpha(0)=0$ and $g_\alpha(1) =\min_{y\in \R_+}-\alpha y+y^2 = -\frac{\alpha_+^2}{4}$
since $y^*=\alpha/2$ if $\alpha\geq 0$ and $y^*=0$ otherwise. Thus,  $g_\alpha$ is a modular function for any $\alpha\in \R^N$, and inequalities \eqref{eq:supermodularDiscrete} reduce to 
$$t \ge -\frac{1}{4}\alpha_+^2x.$$ 
Then, we find that inequalities \eqref{eq:supermodular} reduce to the perspective of $y^2$:
\begin{align*}
t\geq \max_{\alpha\in \R^N}-\frac{1}{4}\alpha_+^2x+\alpha y=\frac{y^2}{x}\tag{with $\alpha^*=2y/x$} \cdot
\end{align*}

\subsection{The rank-one case with free continuous variables}\label{sec:free}

Consider the relaxation of $X$ obtained by dropping the non-negativity constraints $y \ge 0$:
$$X_f=\big\{(x,y,t)\in \{0,1\}^N\times \R^N\times \R: y(N)^2\leq t,\; y_i(1-x_i)=0, \ \forall i\in N\big\} \cdot$$
Observe that any rank-one quadratic constraint of the form 
$\left(\sum_{i\in N}c_iy_i\right)^2\leq t$
with $c_i\neq 0$ can be transformed into the form given in $X_f$ by scaling the continuous variables (so that $|c_i|=1$) and negating variables as $\bar y_i:=-y_i$ if $c_i<0$. 
The closure of the convex hull of $X_f$ is derived in \cite{atamturk2019rank}, and the effectiveness of the resulting inequalities is demonstrated on sparse regression problems. We now re-derive the description of $\clconv(X_f)$ using lifted supermodular inequalities. 

For $S \subseteq N$, we have  
\begin{align*}g_\alpha(x_S)&=\min_{y\in \R^S}-\alpha'y+y(S)^2.\end{align*}
It is easy to see that $g_\alpha(x_S)=-\infty$ unless $\alpha_i=\alpha_j$ for all $i\neq j$, see \cite{atamturk2019rank}. Therefore, letting $\bar \alpha=\alpha_i$ for all $i\in N$, we find that 
\begin{align*}g_{\bar \alpha}(x_S)&=\min_{y\in \R^S}-\bar\alpha y(S)+y(S)^2=\begin{cases}0 & \text{if }S=\emptyset \\
-\bar \alpha^2/4&\text{otherwise,}\end{cases}\end{align*}
where the optimal solution is found by setting $y(S)=\bar \alpha/2$. The function $g_\alpha$ is supermodular since $\rho_{\bar \alpha}(i,\emptyset)=-\bar\alpha^2/4$ and $\rho_{\bar \alpha}(i,S)=0$ for any $S\neq \emptyset$. 

Letting $S=\{1\}$, inequality \eqref{eq:supermodular1} reduces to 
\begin{align*}
	\max_{\bar \alpha\in \R}\;-\frac{\bar \alpha^2}{4}+\bar\alpha y(N)\leq t \ 
	\Leftrightarrow \ y(N)^2\leq t \tag{with $\bar \alpha = 2y(N)$}.
\end{align*}
Also letting $S=\{1\}$, inequality \eqref{eq:supermodular2} reduces to 
\begin{align*}
\max_{\bar \alpha\in \R}\;-\frac{\bar \alpha^2}{4}x(N)+\bar \alpha y(N)\leq t \ 
\Leftrightarrow \ \frac{y(N)^2}{x(N)}\leq t \tag{with $\bar \alpha = 2y(N)/x(N)$}.
\end{align*}
These two supermodular inequalities are indeed sufficient to describe $\conv(X_f)$ \citep{atamturk2019rank}.
As we shall see in \S\ref{sec:strength}, incorporating the non-negativity constraints $y \ge 0$, $\conv(X)$ is substantially more complex than $\conv(X_f)$. Nonetheless, as shown in Example~\ref{ex:freeVsPos}, the resulting convexification is substantially stronger as well.

\ignore{
\begin{proposition}[\citep{atamturk2019rank}]\label{prop:free}
	$$\clconv(X_f)=\left\{(x,y,t)\in [0,1]^N\times \R^N\times \R\cup\infty: y(N)^2\leq t,\; \frac{y(N)^2}{x(N)}\leq t\right\}.$$ 
\end{proposition}
}

\subsection{The rank-one case with a negative off-diagonal}\label{sec:neg}

Consider the special case of $X$ with two continuous variables ($N=\{1,2\}$) with a negative off-diagonal:
$$X_-^2=\left\{(x,y,t)\in \{0,1\}^2\times \R_+^2\times \R: (y_1-y_2)^2\leq t,\; y_i(1-x_i)=0, \ i=1,2\right\}.$$ 
Observe that any quadratic constraint of the form $\left(c_1y_1-c_2y_2\right)^2\leq t$ with $c_1, c_2>0$ can be written
as in $X_-^2$ by scaling the continuous variables. 

For $\alpha \in \R^2$, observe that if $\alpha_1 + \alpha_2 > 0$, 
$$g_\alpha(x)=\min_{y \in \R^2_+} -\alpha_1y_1-\alpha_2y_2+(y_1-y_2)^2$$
is unbounded. Otherwise,
\begin{align*}
g_\alpha(\emptyset)&=0,\\
g_\alpha(\{1\})&=-\frac{\alpha_1^2}{4}\text{ if }\alpha_1\geq 0 \text{ and }g_\alpha(\{1\})=0\text{ otherwise},\\
g_\alpha(\{2\})&=-\frac{\alpha_2^2}{4}\text{ if }\alpha_2\geq 0 \text{ and }g_\alpha(\{2\})=0\text{ otherwise},\\
g_\alpha(\{1,2\})&=\begin{cases}-\frac{\alpha_1^2}{4}&\text{if }\alpha_1\geq 0 \\
-\frac{\alpha_2^2}{4}&\text{if }\alpha_2\geq 0 \\
0 & \text{if }\alpha_1\leq 0 \text{ and }\alpha_2\leq 0.\end{cases}
\end{align*}
In particular, $g_\alpha$ is supermodular (and in fact modular) for any fixed $\alpha$ such that $\alpha_1+\alpha_2\geq 0$: for any $i=1,2$ and $S\subseteq N\setminus i$, $\rho_\alpha(i,S)=-\frac{\max\{0,\alpha_i\}^2}{4}$.  
Letting $S=\emptyset$, inequality \eqref{eq:supermodular1} reduces to 
\begin{align}\label{eq:maxCases}
\max_{\alpha_1+\alpha_2\leq 0}\;-\frac{\max\{0,\alpha_1\}^2}{4}x_1-\frac{\max\{0,\alpha_2\}^2}{4}x_2+\alpha_1y_1+\alpha_2y_2&\leq t.
\end{align}
An optimal solution of \eqref{eq:maxCases} can be found as follows. If $y_1\geq y_2$, then set $\alpha_1>0$ and $\alpha_2=-\alpha_1 < 0$. Moreover, in this case, the optimal value is given by $$\max -\frac{\alpha_1^2}{4}x_1+\alpha_1(y_1-y_2)=\frac{(y_1-y_2)^2}{x_1}.$$
The case $y_2\geq y_1$ is identical. The resulting piecewise valid inequality 
	\begin{equation}t\geq \begin{cases}\frac{(y_1-y_2)^2}{x_1} & \text{if }y_1\geq y_2\\
		\frac{(y_1-y_2)^2}{x_2} & \text{if }y_2\geq y_1\end{cases}
	\end{equation}
along with the bound constraints $0\leq x\leq 1$, $0\leq y$  describe $\clconv(X_-^2)$ \citep{atamturk2018strong}.
We point that a conic quadratic representation for $\clconv(X_-^2)$ and generalizations to (not necessarily rank-one) quadratic functions with negative off-diagonals are given in \cite{atamturk2018signal}.

\ignore{
with two regions corresponding to the relative values of $y_1$ and $y_2$, is sufficient to describe $\clconv(X_-^2)$.

\begin{proposition}[\citep{atamturk2018strong}] Bound constraints $0\leq x\leq 1$, $0\leq y$, and inequality 
	\begin{equation}t\geq \begin{cases}\frac{(y_1-y_2)^2}{x_1} & \text{if }y_1\geq y_2\\
	\frac{(y_1-y_2)^2}{x_2} & \text{if }y_2\geq y_1\end{cases}\end{equation}
	describe $\clconv(X_-^2)$.
\end{proposition}
}

\subsection{Outlier detection with temporal data}\label{sec:outlier}

In the context of outlier detection with temporal data, \citet{gomez2019outlier} studies the set 
\begin{align*}X_T=\Big\{(x,y,t)\in \{0,1\}^2\times \R^4\times \R:& \frac{a_1}{2}(y_3-y_1)^2+(y_3-y_4)^2+\frac{a_2}{2}(y_4-y_2)^2\leq t,\\
& y_1(1-x_1)=0,\; y_2(1-x_2)=0\Big\}
\end{align*}
where $a_1,a_2>0$ are constants. While we refer the reader to \cite{gomez2019outlier} for details on the derivation of $\clconv(X_T)$, we point out that it can in fact be described by lifted supermodular inequalities. Indeed, in this case, function $g_\alpha$ is given by 
$$g_\alpha(x)=K_1(\alpha)-K_2(\alpha)\max\{x_1,x_2\},$$
where $K_1(\alpha)$ and $K_2(\alpha)$ are constants that do not depend on $x$ and $K_2(\alpha)\geq 0$. Since $\max\{x_1,x_2\}$ is a submodular function, it follows that $g_\alpha$ is supermodular.

\section{Convex hull via lifted supermodular inequalities}\label{sec:strength}

\ignore{

In general, inequalities \eqref{eq:supermodularDiscrete} are not particularly strong and may provide a poor approximation of $\conv(Z)$. As a consequence, the lifted supermodular inequalities \eqref{eq:supermodular} are not guaranteed to be strong for arbitrary sets $H$. Nonetheless, as shown in Section~\ref{sec:previous}, the lifted supermodular inequalities do result in ideal formulations when applied to well-understood rank-one quadratic functions, and thus they may also be effective for sets $X_+$ and $X$. This section is devoted to proving that this is indeed the case, i.e., proving Theorem~\ref{theo:hullRankOne} below.

\todo{We can skip this theorem as we give explicit theorems later. Then there is no need to mention that (8b) are weak. }
\begin{theorem}\label{theo:hullRankOne}
Inequalities \eqref{eq:supermodular1} and bound constraints describe $\conv(X)$ and $\conv(X_+)$. 
\end{theorem}
Inequalities \eqref{eq:supermodular2} are weak in this case and are not needed to describe  $\conv(X)$.
}


We now turn our attention to the rank-one sets $X$ and $X_+$. This section is devoted to showing that the lifted supermodular inequalities \eqref{eq:supermodular} are sufficient to describe $\clconv(X)$ and $\clconv(X_+)$. By Theorem~\ref{theo:hullSupermodular}, it suffices to derive an explicit form of the projection function $g_\alpha$ and show that inequalities \eqref{eq:supermodularDiscrete} describe the convex hull of its epigraph $G_\alpha$. 
The rest of this section is organized as follows.  In \S\ref{sec:gAlpha} we  derive the set function $g_\alpha$ defined in \eqref{eq:defG} for the rank-one quadratic function and then show that it is supermodular. In \S\ref{sec:facetMin} we describe the convex hull of $G_\alpha$ using only a small subset of the supermodular inequalities \eqref{eq:supermodularDiscrete}.  

\ignore{
\todo{I think this write-up is too detailed for intro to the section. No need to define the notation yet.}
 In \S\ref{sec:gAlpha} we show that the function $g_\alpha$ given in \eqref{eq:defG}, when bounded, reduces to the supermodular function
\begin{equation}\label{eq:gAlpha}
g_\alpha(x_S)=-\frac{\max_\alpha(T)^2}{4},
\end{equation}
where $T=S$ for $X_+$, and either $T=S\cap N^+$ or $T=S\cap N^-$ for $X$. We point out, however, that the set $B$ where $|g_\alpha(x)|<\infty$ is different for $X_+$ and $X$.  

In \S\ref{sec:facetMin} we study the convex hull of the epigraph of function $g_\alpha$, i.e., 
\begin{equation}\label{eq:defW}W=\left\{(x,t)\in \{0,1\}^N\times \R: -\frac{\max_{i\in N}\{\alpha_i^2x_i\}}{4}\leq t\right\},\end{equation}
and show that the supermodular inequalities \eqref{eq:supermodularDiscrete1} describe $\conv(W)$. In \S\ref{sec:hullXg} we conclude that inequalities \eqref{eq:supermodular1} indeed describe $\conv(X)$ and $\conv(X_+)$.
}

\subsection{The set function $g_\alpha$}\label{sec:gAlpha}

We present the derivation of  set function $g_\alpha$ for $X_+$ and $X$ separately, and then verify that $g_\alpha$ is indeed supermodular.

\subsubsection{Derivation for $X_+$}\label{sec:gPositive} For $X_+$,  
$$h(x,y)=\begin{cases}
	y(N)^2 &\text{if }y\geq 0\text{ and }y_i(1-x_i)=0, \ \forall i\in N\\
	\infty &\text{otherwise.}
\end{cases}
$$
Therefore, for $S \subseteq N$,
\begin{align}g_\alpha(x_S)&=\min_{y\in \R_+^S}-\alpha'y_S+y(S)^2.\label{eq:projectionPos}
\end{align} 
Note that \eqref{eq:projectionPos} is bounded for all $\alpha\in \R^S$, thus $B=\R^N$. Since, for $\alpha_i<0$,  $y_i=0$ in any optimal solution, 
we assume for simplicity that $\alpha\geq 0$ and $B=\R_+^N$. 
From the KKT conditions corresponding to variable $y_k\geq 0$ in \eqref{eq:projectionPos}, we find that
\begin{align}
&2y(S)\geq \alpha_k\label{eq:max},
\end{align}
and, by complementary slackness, \eqref{eq:max} holds at equality whenever $y_k>0$. Moreover, let $j\in S$ such that $\alpha_j=\max_\alpha(S)$; setting $y_j=\alpha_j/2$ and $y_i=0$ for $i\in S\setminus j$, we find a feasible solution for \eqref{eq:projectionPos} that satisfies all dual feasibility conditions \eqref{eq:max} and complementary slackness, and therefore is optimal for the convex optimization problem \eqref{eq:projectionPos}. Thus, we conclude that
$$g_\alpha(x_S)=-\frac{\max_\alpha(S)^2}{4} \cdot$$

\subsubsection{Derivation for $X$} \label{sec:gGeneral} For the general case of $X$,
$$h(x,y)=\begin{cases}
	\big(y(N^+) - y(N^-) \big)^2 &\text{if }y\geq 0\text{ and }y_i(1-x_i)=0, \ \forall i\in N^+ \cup N^-\\
	\infty &\text{otherwise.}
\end{cases}
$$
Therefore, for $S \subseteq N^+ \cup N^-$,
\begin{align}g_\alpha(x_S)&=\min_{y\in \R_+^S}-\alpha'y+\Big(y(N^+\cap S)-y(N^-\cap S)\Big)^2.\label{eq:projectionGen}
\end{align}
If $S\cap N^-=\emptyset$ or $S\cap N^+=\emptyset$, then we find from \S\ref{sec:gPositive} that $g_\alpha(x_S)= -\max_\alpha(S)^2/4$. Now let $S^+:=S\cap N^+$ and $S^-:=S\cap N^-$, and assume $S^+\neq\emptyset$ and $S^-\neq\emptyset$.

\begin{proposition}
Problem 	\eqref{eq:projectionGen} is bounded if and only if
\begin{equation}\max_\alpha(S^+)\leq -\max_\alpha(S^-).
	\label{eq:conditionB}
\end{equation}
\end{proposition}
\begin{proof} 
Let $p = \argmax_{i \in S^+} \alpha_i$ and $q =  \argmax_{i \in S^-} \alpha_i$. If $\alpha_p + \alpha_q > 0$ , then
$e_p + e_q$ is an unbounded direction.	Otherwise,
\begin{align*}
-\alpha'y+\Big(y(S^+)-y(S^-)\Big)^2 & \ge -\alpha_p y(S^+) - \alpha_q y(S^-) + \Big(y(S^+)-y(S^-)\Big)^2 \\
& \ge -\alpha_p (y(S^+)-y(S^-)) + \Big(y(S^+)-y(S^-)\Big)^2 \\
& \ge -\alpha_p^2/4,
\end{align*}
where the second inequality follows from  $\alpha_p + \alpha_q \le 0$.
\end{proof}
\ignore{
A necessary condition for $|g_\alpha(x_S)|<\infty$ is that 
\begin{equation}\label{eq:boundedCondition}
\infty>\max_{y\in \R_+^S}\;\alpha'y
\text{  subject to  }y(S^+)-y(S^-)=\zeta
\end{equation} for any fixed $\zeta\in \R$. The linear programming dual of \eqref{eq:boundedCondition} is
\begin{equation}\label{eq:boundedCondition_dual}\min_{w\in \R}\zeta w \text{  subject to  }w\geq\alpha_i \ \ \forall i\in S^+\text{ and } -w\geq \alpha_i \ \ \forall i\in S^-.\end{equation}
 Thus we see that \eqref{eq:boundedCondition} is bounded if and only if the dual \eqref{eq:boundedCondition_dual} is feasible,  i.e., if 
\begin{equation}\max_\alpha(S^+)\leq -\max_\alpha(S^-).
	\label{eq:conditionB}
\end{equation}
}

Note that for \eqref{eq:conditionB} to hold,
if there exists $j\in S^-$ such that $\alpha_j\geq 0$, then $\alpha_i\leq 0$ for all $i\in S^+$. Therefore, either $\alpha_i\leq 0$ for all $i\in S^+$ or $\alpha_j\leq 0$ for all $j\in S^-$. Also note that we may equivalently rewrite \eqref{eq:conditionB} as
\begin{equation*}\alpha_i+\alpha_j\leq 0,\text{ for all } i\in S^+,\; j\in S^-.
\end{equation*}

First, assume that $\alpha_j\leq 0$ for all $j\in S^-$. In this case, there exists an optimal solution of \eqref{eq:projectionGen} where $y(S^-)=0$ and 
\eqref{eq:projectionGen} reduces to \eqref{eq:projectionPos}. Then, we may assume that $\alpha_i\geq 0$ for all $i\in S^+$ as in \S\ref{sec:gPositive}, and arrive at
\begin{align*}g_\alpha(x_S)&=-\frac{\max_\alpha(S^+)^2}{4} \cdot\end{align*}
By symmetry, if $\alpha_i\leq 0$ for all $i\in S^+$,  we may assume that $\alpha_j\geq 0$ for all $i\in S^-$ and
\begin{align*}g_\alpha(x_S)&=-\frac{\max_\alpha(S^-)^2}{4} \cdot\end{align*}

From the discussion above, we see that we can assume in \eqref{eq:supermodular} that
$$B=\left\{\alpha\in \R^N:\alpha_i\alpha_j\leq 0 \text{ and } \alpha_i+\alpha_j\leq0 \ \text{ for all } i\in N^+,j\in N^-\right\}.$$
It is convenient to partition $B$ into two sets so that $B=B^+\cup B^-$, where 
\begin{align*}B^+&=\left\{\alpha\in \R^N:\alpha_i\geq 0 \ \forall i\in N^+,\alpha_j\leq 0 \ \forall j\in N^-,\text{ and } \alpha_i+\alpha_j\leq0 \ \forall i\in N^+,j\in N^-\right\}\\
B^-&=\left\{\alpha\in \R^N:\alpha_i\leq 0 \ \forall i\in N^+,\alpha_j\geq 0 \ \forall j\in N^-,\text{ and } \alpha_i+\alpha_j\leq0 \ \forall i\in N^+,j\in N^-\right\}
\end{align*} 
and analyze the inequalities separately for each set.
Figure~\ref{fig:regionB} depicts regions $B^+$ and $B^-$ for a two-dimensional case.
\begin{figure}[!h]	
\includegraphics[width=0.47\textwidth,trim={5.8cm 2cm 5.8cm 2cm},clip]{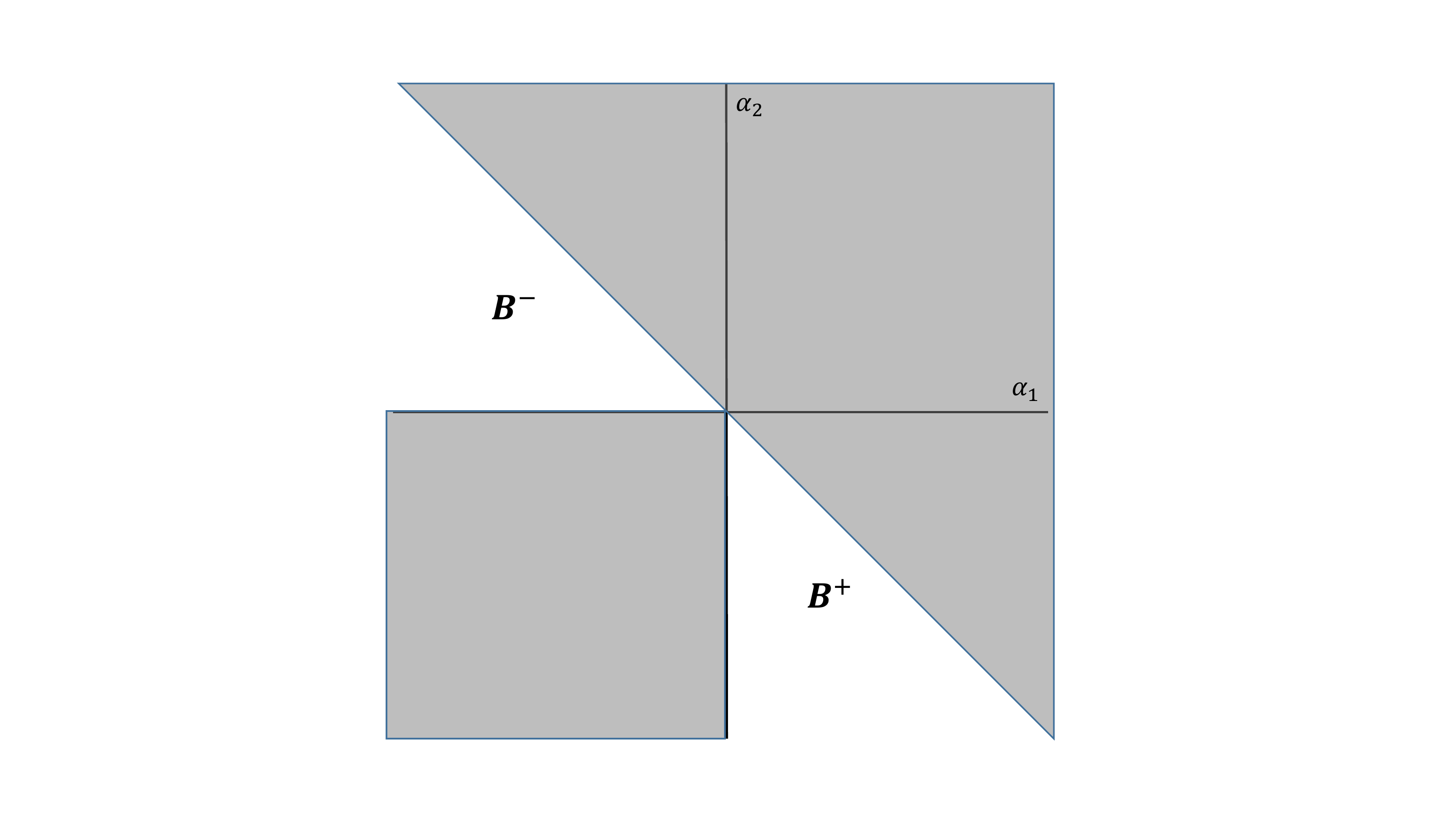}
	\caption{Depiction of $B^+$ and $B^-$ in a two-dimensional example with $N^+=\{1\}$ and $N^-=\{2\}$. The upper right shaded region (triangle) corresponds to the region where $g_\alpha(x)=-\infty$; the lower left shaded region (square) corresponds to the region discarded, as optimal solutions of \eqref{eq:supermodular} can be found in either $B^+$ or $B^-$.}
	\label{fig:regionB}
\end{figure}

Therefore, instead of studying inequalities \eqref{eq:supermodular} directly, one can equivalently study their relaxation where either $\alpha\in B^+$ or $\alpha\in B^-$; consequently, each inequality \eqref{eq:supermodular} corresponds to (the maximum of) two simpler inequalities. Since the sets $B^+$ and $B^-$ are symmetric, and inequalities \eqref{eq:supermodular} corresponding to $\alpha\in B^-$ are simply inequalities where the role of $N^+$ and $N^-$ is interchanged (and $\alpha\in B^+$), the analysis and derivation of the inequalities is simplified. Therefore, in the sequel, we will derive the inequalities for $\alpha\in B^+$ only and then state the inequalities corresponding to $B^-$ by interchanging $N^+$ and $N^-$.

\subsubsection*{Supermodularity} 
For $\alpha \in B^+$, the set function $g_\alpha(x)$ for $X$ is monotone non-increasing, also it is supermodular as 
$\max_\alpha(S^+)$ is submodular. The case for  $\alpha\in B^-$ is analogous. 

\ignore{
and for $\alpha\in B^+$ we find
$$\rho_\alpha(i,S)=\begin{cases}0 & \text{if }i\in N^-\\
0 & \text{if }\alpha_i\leq \max_\alpha(S^+)\\ -\frac{\alpha_i^2-\max_\alpha(S^+)^2}{4}&\text{otherwise.}\end{cases}$$ 
We find that for $S\subseteq T \subseteq N\setminus i$ and $i\in N^+$,
$$\rho_\alpha(i,T)-\rho_\alpha(i,S)=\begin{cases} 0 & \text{if }\alpha_i\leq \max_\alpha(S^+)\leq \max_{\alpha}(T^+)\\
\frac{\alpha_i^2-\max_\alpha(S^+)^2}{4}&\text{if }\max_{\alpha}(S^+)\leq \alpha_i\leq \max_\alpha(T^+)\\
\frac{\max_\alpha(T^+)^2-\max_\alpha(S^+)^2}{4}&\text{if }\max_\alpha(S^+)\leq  \max_\alpha(T^+)\leq \alpha_i,\end{cases}$$\
and in particular $\rho_\alpha(i,T)-\rho_\alpha(i,S)\geq 0$ and $g_\alpha$ is supermodular. The case where $\alpha\in B^-$ is analogous, and function $g_\alpha^p$ is also supermodular since it is a special case.
}

\subsection{Convex hull of epi $g_\alpha$}\label{sec:facetMin}

In this section we show that a small subset of the supermodular inequalities \eqref{eq:supermodularDiscrete1} are sufficient 
to describe the convex hull of the epigraph of the set function $g_\alpha$, i.e., 
\begin{equation}\label{eq:defW} G_\alpha=\left\{(x,t)\in \{0,1\}^N\times \R: -\frac{\max_{i\in N}\{\alpha_i^2x_i\}}{4}\leq t\right\} \cdot
\end{equation}


Given nonempty $ S\subseteq N$,  $\ell\in \argmax_{i\in S}\{\alpha_i\}$,  $k\in \argmax_{i\in N\setminus \ell}\{\alpha_i\}$, and  $T=\left\{i\in N\setminus S:\alpha_i>\alpha_\ell\right\}$; observe that $T=\emptyset$ if and only if $\alpha_\ell\geq \alpha_k$. Then, valid inequalities \eqref{eq:supermodularDiscrete1} for $G_\alpha$ reduce to
\begin{align}\label{eq:supermodularSimplified}
t\geq \begin{cases}-\frac{\alpha_\ell^2}{4}+\frac{\alpha_\ell^2-\alpha_k^2}{4}(1-x_\ell)&\text{if } \alpha_\ell\geq \alpha_k\\
-\frac{\alpha_\ell^2}{4}-\sum\limits_{i\in T}\frac{\alpha_i^2-\alpha_\ell^2}{4}x_i & \text{if } \alpha_\ell\leq \alpha_k.
\end{cases}
\end{align}
If $S=\emptyset$, then valid inequalities \eqref{eq:supermodularDiscrete1} reduce to 
$$t\geq -\sum\limits_{i\in N}\frac{\alpha_i^2}{4}x_i.$$

\begin{remark}\label{rem:necessary}
	Observe that if $\alpha_\ell\geq \alpha_k$, then the inequality $$t\geq -\frac{\alpha_\ell^2}{4}+\frac{\alpha_\ell^2-\alpha_k^2}{4}(1-x_\ell)=-\frac{\alpha_k^2}{4}-\frac{\alpha_\ell^2-\alpha_k^2}{4}x_\ell$$ can also be obtained by setting $S=N\setminus \ell$ (or by choosing any $S\subseteq N\setminus \ell$ such that $k\in S$). Therefore, when considering inequalities \eqref{eq:supermodularSimplified}, we can assume without loss of generality that there exists $k\in \argmax_{i\in N}\{\alpha_i\}$ such that $k\not\in S$ and, thus, the case $\alpha_\ell\geq \alpha_k$ can be ignored.
	\qed
\end{remark}

\begin{remark}\label{rem:number}
	Suppose that the variables are indexed such that $\alpha_1\leq\ldots\leq \alpha_n$, let $\alpha_0=0$, and let $\ell= \max_{i\in S}\{i \}$ if $S\neq \emptyset$ and $\ell=0$ otherwise. Observe that we can assume without loss of generality that $i\in S$ for all $i\leq \ell$, since inequalities \eqref{eq:supermodularSimplified} are the same whether $i\in S$ or not. Therefore, it follows that there are only $n$ inequalities \eqref{eq:supermodularSimplified} given by
	\begin{equation}\label{eq:supermodularSorted} t \ge -\frac{\alpha_\ell^2}{4}-\sum_{i=\ell+1}^n\frac{\alpha_i^2-\alpha_\ell^2}{4}x_i, \quad \ell=0,\ldots,n-1.\end{equation}
\end{remark}

We now show that inequalities \eqref{eq:supermodularSimplified} characterize the convex hull of $G_\alpha$. 
\begin{proposition}\label{prop:hullDiscrete} 
Inequalities \eqref{eq:supermodularSimplified} 
and bound constraints describe $\conv(G_\alpha)$.
\end{proposition}
\begin{proof}
Let $(x,t)\in [0,1]^N\times \R$. By definition, $(x,t)\in \conv(G_\alpha)$ if and only
\begin{subequations}\label{eq:hullPrimal}
\begin{align}t\geq \min_{\lambda}\;&-\sum_{S\subseteq N}\frac{\max_\alpha(S)^2}{4}\lambda_S\\
\text{s.t.}\;&\sum_{S\subseteq N: i\in S}\lambda_S=x_i, \ \  i\in N\label{eq:hullPrimal_x}\\
&\sum_{S\subseteq N}\lambda_S=1\label{eq:hullPrimal_convexComb}\\
&\lambda_S\geq 0, \ \  S\subseteq N,
\end{align}
\end{subequations}
where constraints \eqref{eq:hullPrimal_x} can be restated as $x=\sum_{i\in S}\lambda_Sx_S$. 
From linear programming duality, we find the equivalent condition 
\begin{subequations}\label{eq:hullDual}
	\begin{align}t\geq \max_{\mu,\gamma}\;&\sum_{i\in N}x_i\mu_i+\gamma\label{eq:hullDual_obj}\\
	\text{s.t.}\;&\sum_{i\in S}\mu_i+\gamma\leq -\frac{\max_\alpha(S)^2}{4}, \ \ S\subseteq N\label{eq:hullDual_constS}\\
	&\mu\in \R^N,\; \gamma\in \R. 
	\end{align}
\end{subequations}
Any feasible solution $(\mu,\gamma)$ of \eqref{eq:hullDual} yields a valid inequality for $\conv(G_\alpha)$. Moreover, characterizing the optimal solutions of \eqref{eq:hullDual} (for all $x\in [0,1]^N$) results in the convex hull description of $G_\alpha$. 

Suppose, without loss of generality, that $\alpha_1\leq \ldots\leq \alpha_n$, let $\alpha_0=0$, and let
 $\ell\in \{0,\ldots,n-1\}$ be the smallest index such that $\sum_{i=\ell+1}^nx_i\leq 1$; thus, if $\ell>0$, then $\sum_{i=\ell}^nx_i>1$.
We claim that the dual solution given by $\hat\gamma=-\frac{\alpha_\ell^2}{4}$, $\hat\mu_i=0$ for $i\leq \ell$ and $\hat\mu_i=-\frac{\alpha_i^2-\alpha_\ell^2}{4}$ for $i>\ell$ is optimal for \eqref{eq:hullDual}.

First, we verify that $(\hat\mu,\hat\gamma)$ is feasible for \eqref{eq:hullDual}. Observe that for any $S\subseteq\{1,\ldots,\ell\}$, constraint \eqref{eq:hullDual_constS} reduces to $-\frac{\alpha_\ell^2}{4}\leq -\frac{\max_\alpha(S)^2}{4}$, which is indeed satisfied. For any $S$ such that the maximum element $j>\ell$, we find that \eqref{eq:hullDual_constS} reduces to
$\sum_{i\in S:i\neq j}\hat\mu_i\leq 0$; since $\hat\mu\leq 0$, the constraint is satisfied. For $S=\emptyset$, constraint \eqref{eq:hullDual_constS} reduces to $\gamma\leq 0$, which is satisfied.
To verify complementary slackness (later), note
that constraints \eqref{eq:hullDual_constS} corresponding to sets (a) 
$S = T \cup \{j\}$, where $T \subseteq \{1, \ldots, \ell\}$ and $j > \ell$ (i.e., containing exactly one element greater than $\ell$), and (b) $S = T \cup \{\ell\}$, where $T \subseteq \{1, \ldots, \ell-1\}$ (i.e., containing $\ell$ but no greater element) are satisfied at equality.

Finally, for $(\hat \mu,\hat \gamma)$, the objective function \eqref{eq:hullDual_obj} is of the form \eqref{eq:supermodularSorted}:
$$t\geq -\frac{\alpha_\ell^2}{4}-\sum_{i=\ell+1}^n\frac{\alpha_i^2-\alpha_\ell^2}{4}x_i.$$

To verify that $(\hat \mu, \hat \gamma)$ is optimal for \eqref{eq:hullDual}, we construct a primal solution $\hat\lambda$ feasible for \eqref{eq:hullPrimal} satisfying complementary slackness.
The greedy algorithm for constructing $\hat\lambda$ is presented  in Algorithm~\ref{alg:primal} and illustrated with an example in Figure~\ref{fig:greedyAlgorithm}.

\ignore{
First, we provide some insights into the structure of the optimal solutions of \eqref{eq:hullPrimal}. Note that variables $\hat \lambda_S$ for sets $S$ containing an element with a ``large" index value are preferable (assuming $\alpha_1\leq \ldots\leq \alpha_n$), as they have a better objective coefficient. Intuitively, in order to achieve a good objective value while satisfying the constraints \eqref{eq:hullPrimal_x}--\eqref{eq:hullPrimal_convexComb}, one should set $\hat \lambda_S>0$ for sets containing a single large element (and many small ones). An ideal vector $\hat\lambda$ with such characteristics can be done via a greedy algorithm, as illustrated in Figure~\ref{fig:greedyAlgorithm} in an example with $n=5$. A formal description of the algorithm is given in Algorithm~\ref{alg:primal}.
}

\begin{figure}[!h]	
	\begin{adjustbox}{minipage=\linewidth,scale=0.8}
	\subfloat[Buffer $1-x_4-x_5=0.1$ for index $\ell=3$, $\hat\lambda=0$]{\includegraphics[width=0.47\textwidth,trim={6cm 2cm 10cm 2cm},clip]{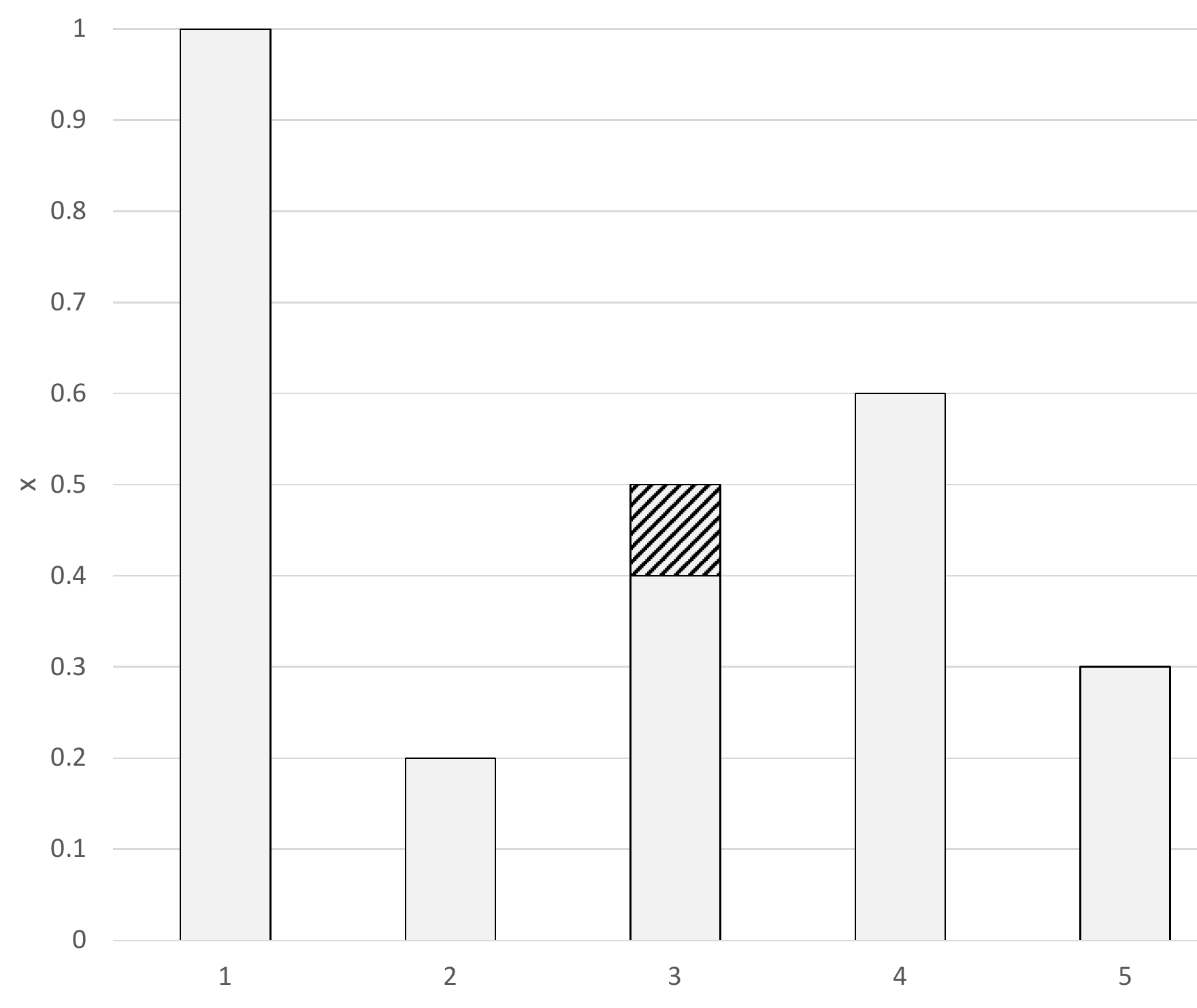}}\ \hfill
	\subfloat[$\hat\lambda_{1,2,3,5}\leftarrow x_2=0.2$. Constraint \eqref{eq:hullPrimal_x} for $x_2$ is met]{\includegraphics[width=0.47\textwidth,trim={6cm 2cm 10cm 2cm},clip]{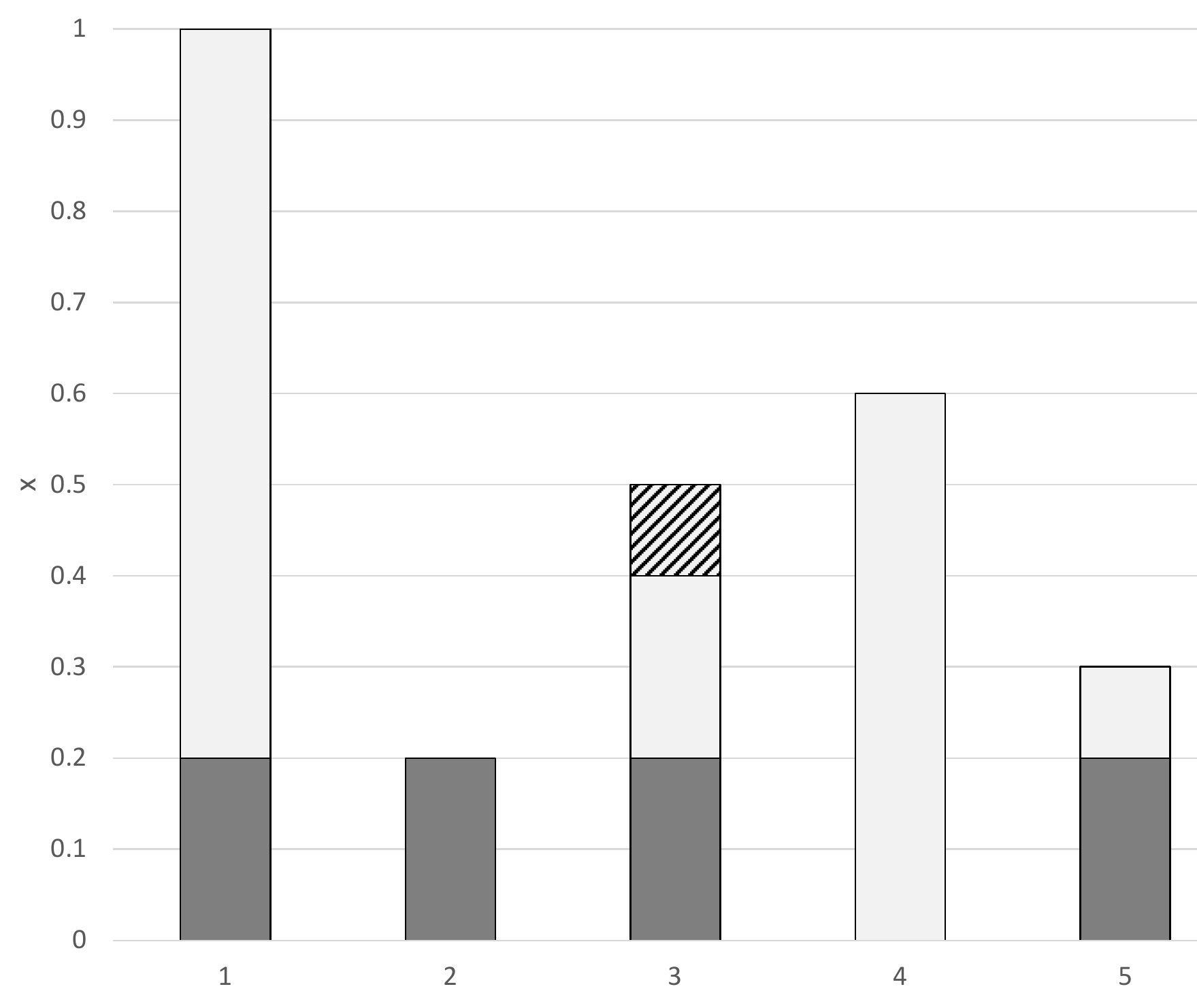}}\\
	\subfloat[$\hat\lambda_{1,3,5}\leftarrow x_5-x_2=0.1$. Constraint \eqref{eq:hullPrimal_x} for $x_5$ is met]{\includegraphics[width=0.47\textwidth,trim={6cm 2cm 10cm 2cm},clip]{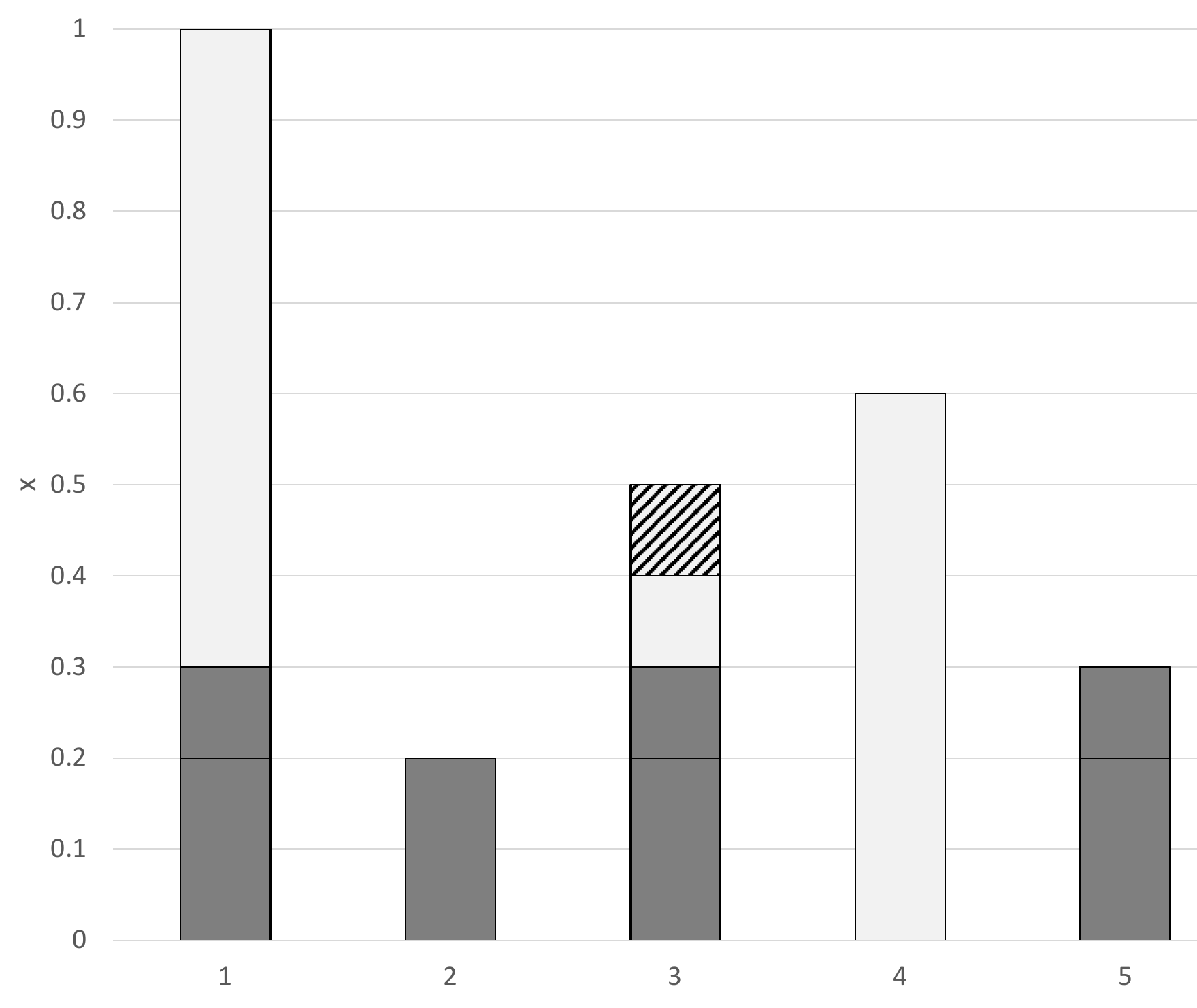}}\ \hfill
	\subfloat[$\hat\lambda_{1,3,4}\leftarrow 0.1$. Constraint \eqref{eq:hullPrimal_x} for $x_3$ is not met due to the buffer]{\includegraphics[width=0.47\textwidth,trim={6cm 2cm 10cm 2cm},clip]{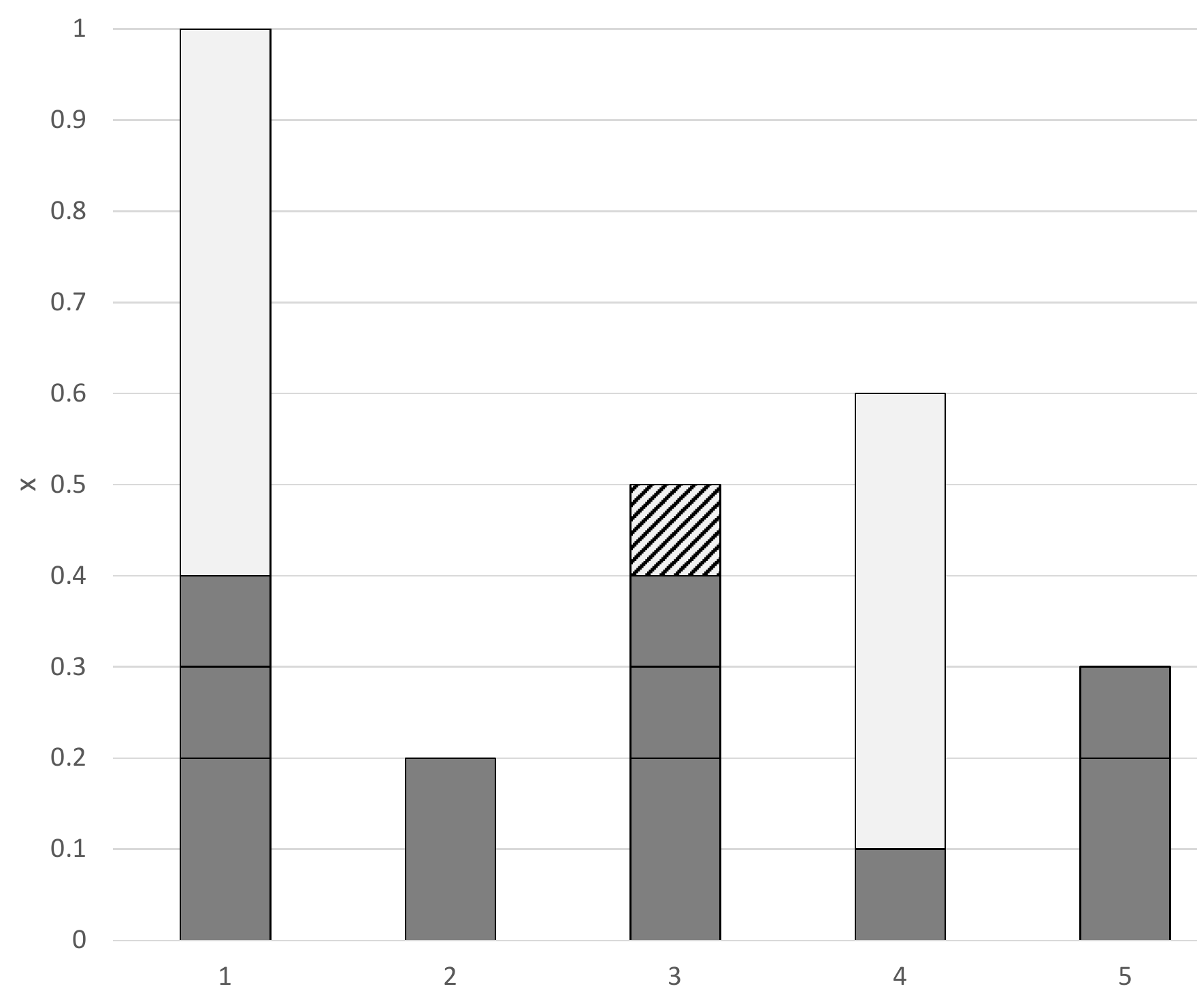}}\\
	\subfloat[$\hat\lambda_{1,4}\leftarrow 0.5$. Constraint \eqref{eq:hullPrimal_x} for $x_4$ is met and the buffer  is removed]{\includegraphics[width=0.47\textwidth,trim={6cm 2cm 10cm 2cm},clip]{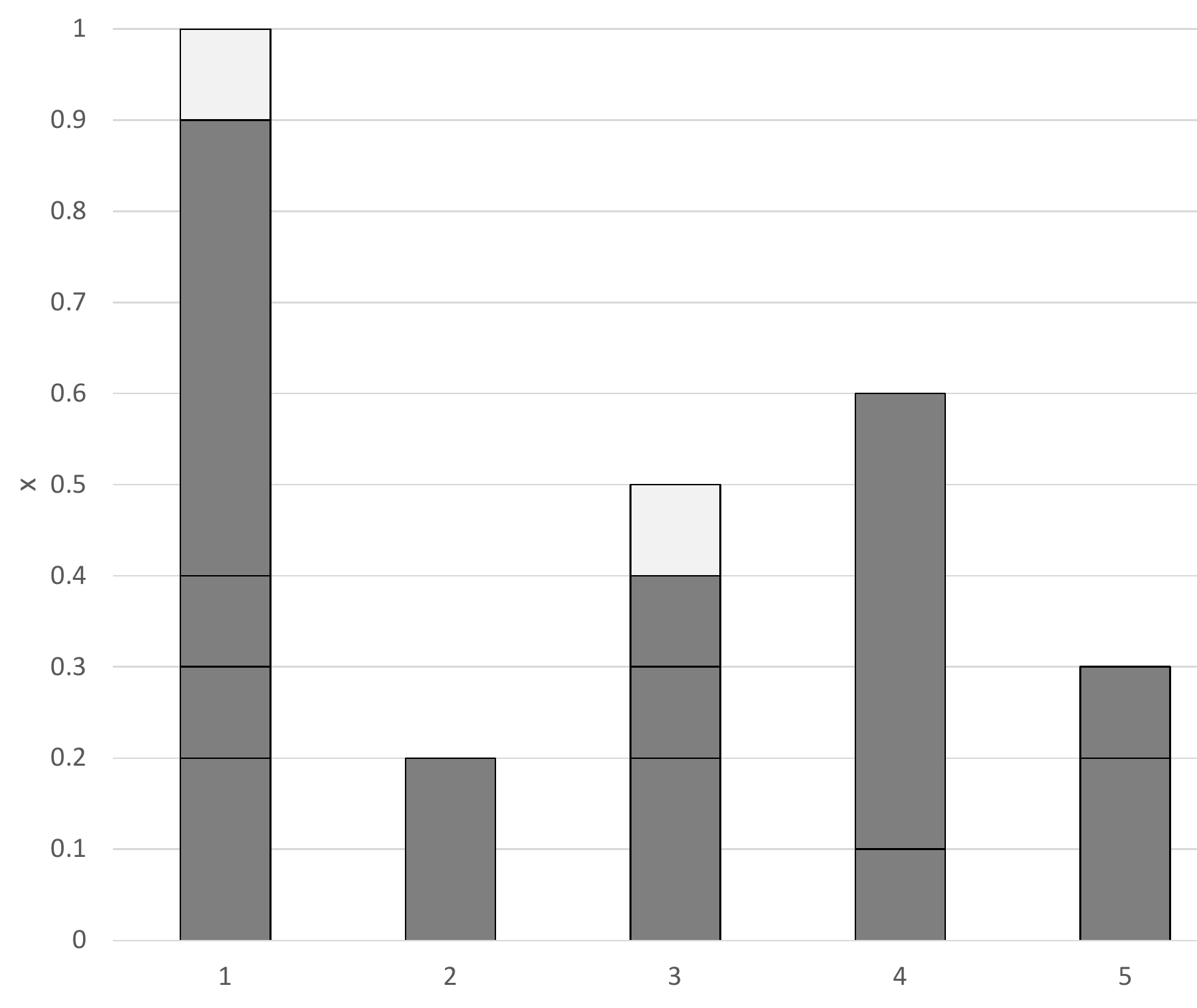}}\ \hfill
	\subfloat[$\hat\lambda_{1,3}\leftarrow 0.1$. Constraints  \eqref{eq:hullPrimal_x} for $x_1,x_3$ and \eqref{eq:hullPrimal_convexComb} are met]{\includegraphics[width=0.47\textwidth,trim={6cm 2cm 10cm 2cm},clip]{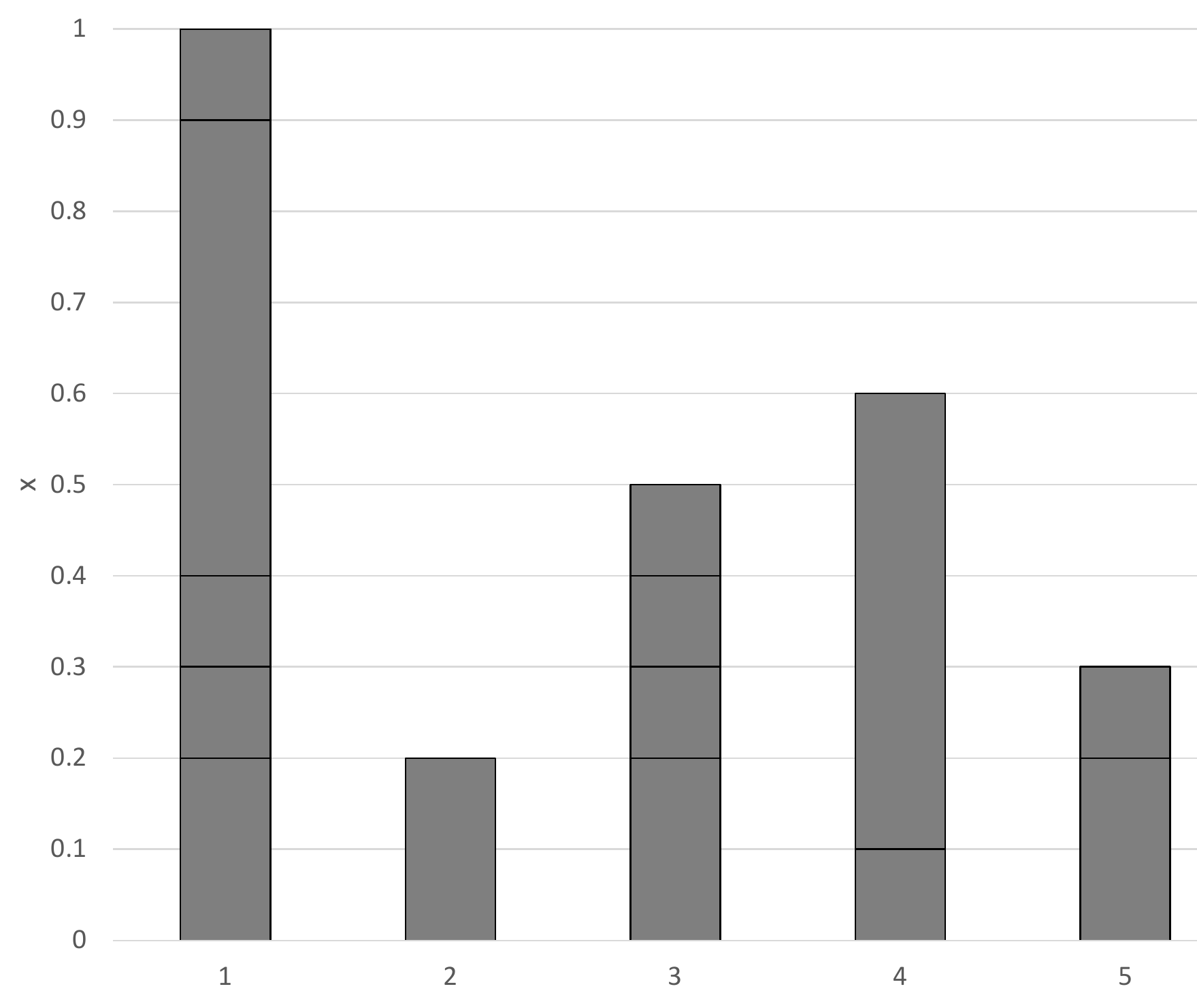}}
	\caption{Algorithm~\ref{alg:primal}  with $x =(1, 0.2, 0.5, 0.6, 0.3)$ and $\ell=3$.}
	\label{fig:greedyAlgorithm}
	\end{adjustbox}
\end{figure}

\begin{algorithm}[h]
	\caption{Algorithm for problem \eqref{eq:hullPrimal}}
	\label{alg:primal} \small
	\begin{algorithmic}[1]
		\renewcommand{\algorithmicrequire}{\textbf{Input:}}
		\renewcommand{\algorithmicensure}{\textbf{Output:}}
		\Require $x_1,\dots, x_n$ with $0 =: \alpha_0 \le \alpha_1\leq\dots\leq \alpha_n$.
		\Ensure $\hat \lambda$ optimal for \eqref{eq:hullPrimal}.
		
		\State $\ell \leftarrow \min\{k: \sum_{i=k+1}^nx_i\leq 1\}$
		\State $\hat \lambda_S \leftarrow 0$ for all $S\subseteq N$ 
		\State $\hat x_i\leftarrow x_i$ for all $i\in N\setminus \ell$
		\If{$\ell>0$}
		\State $\hat x_\ell \leftarrow x_\ell-\left(1-\sum_{i=\ell+1}^n x_i\right)$ \Comment{Buffer for $\ell$; $\hat x_\ell>0$ by definition of $\ell$}
		\EndIf
		\State $\Lambda \leftarrow 0$ \Comment{Variable for $\sum_S \hat \lambda_S$}
		
		\For{$j= n \ldots \ell+1$}
			\While{$\hat x_j > 0$}
			\State $S\leftarrow \{i\leq \ell: \hat x_i>0 \}\cup j$ \label{line:updateS1}
	        \State   \Call{Allocate}{$S$}
				\EndWhile \label{line:endWhile1}
		\EndFor  \label{line:endFor} \Comment{At this point $\Lambda=\sum_{i={\ell + 1}}^n x_i$ }\label{line:saturation}
		\If{$\ell=0$}
		\State $\hat \lambda_{\emptyset} \leftarrow 1-\Lambda$ \Comment{Constraint \eqref{eq:hullPrimal_convexComb} is satisfied}
		\State $\Lambda\leftarrow 1$ \label{line:lambdaEq1}
		\Else

		\State $\hat x_\ell\leftarrow 1-\sum_{i=\ell+1}^nx_i$\label{line:updateXl}\Comment{Buffer is removed}
		\While{$\hat x_\ell>0$}
		\State $S\leftarrow \{i< \ell: \hat x_i>0 \}\cup \ell$ \label{line:updateS2}
			        \State   \Call{Allocate}{$S$}

		\EndWhile \label{line:endWhile2} \Comment{$\Lambda=1$}
		\EndIf
		\State \Return $\hat \lambda$
		
		\State 
		\Function{Allocate}{$S$}
		\State $v\leftarrow \min_{i\in S}\hat x_i$
		\State $\Lambda \leftarrow \Lambda +v$
		\State $\hat\lambda_S\leftarrow v$ \label{line:updateLambda1}
		\State $\hat x_i\leftarrow \hat x_i-v$ for all $i\in S$ \Comment{At this point $\hat x_i=\left(x_i-\Lambda \right)_+$ for all $i<\ell$}
		\EndFunction
	\end{algorithmic}
\end{algorithm}

\ignore{
\begin{algorithm}[h]
	\caption{Algorithm for problem \eqref{eq:hullPrimal}}
	\label{alg:primal}
	\begin{algorithmic}[1]
		\renewcommand{\algorithmicrequire}{\textbf{Input:}}
		\renewcommand{\algorithmicensure}{\textbf{Output:}}
		\Require $x_1\leq \ldots\leq x_n$.
		\Ensure $\hat \lambda$ optimal for \eqref{eq:hullPrimal}.
		
		\State $\hat \lambda_S \leftarrow 0$ for all $S\subseteq N$ 
		\State $\hat x_i\leftarrow x_i$ for all $i\in N\setminus \ell$
		\State $\hat x_\ell \leftarrow x_\ell-\left(1-\sum_{i=\ell+1}^n x_i\right)$\Comment{``Protection", $\hat x_\ell>0$ by definition of $\ell$}
		\State $j\leftarrow n$ \Comment {Index}
		\State $v_{\text{tot}}\leftarrow 0$ \Comment{Auxiliary variable used in the analysis of the algorithm}
		\While{$j>\ell$}
		\State $S\leftarrow \{i\leq \ell: \hat x_i>0 \}\cup j$ \label{line:updateS1}
		\State $v\leftarrow \min_{i\in S}\hat x_i$
		\State $v_{\text{tot}}\leftarrow v_{\text{tot}}+v$
		\State $\hat\lambda_S\leftarrow v$ \label{line:updateLambda1}
		\State $\hat x_i\leftarrow \hat x_i-v$ for all $i\in S$ \Comment{At this point $\hat x_i=\left(x_i-v_{\text{tot}}\right)_+$ for all $i<\ell$}
		\If{$\hat x_j=0$} \Comment{$\hat x_j$ is ``saturated", at this point $v_{\text{tot}}=\sum_{i=j}^n x_i$ }\label{line:saturation}
		\State $j \leftarrow j-1$
		\EndIf
		\EndWhile \label{line:endWhile1}
		\State $\hat x_\ell\leftarrow 1-\sum_{i=\ell+1}^nx_i$\label{line:updateXl}\Comment{Removes the ``protection".}
		\While{$\hat x_\ell>0$}
		\State $S\leftarrow \{i< \ell: \hat x_i>0 \}\cup \ell$ \label{line:updateS2}
		\State $v\leftarrow \min_{i\in S}\hat x_i$
		\State $v_{\text{tot}}\leftarrow v_{\text{tot}}+v$
		\State $\hat\lambda_S\leftarrow v$ \label{line:updateLambda2}
		\State $\hat x_i\leftarrow \hat x_i-v$ for all $i\in S$\Comment{At this point $\hat x_i=\left(x_i-v_{\text{tot}}\right)_+$ for all $i<\ell$}
		\EndWhile \label{line:endWhile2} \Comment{$v_{\text{tot}}=1$}
		\State \Return $\hat \lambda$
	\end{algorithmic}
\end{algorithm}
}

We now check that constraint \eqref{eq:hullPrimal_convexComb} is satisfied. At the end of the algorithm, $\sum_{S\subseteq N}\hat \lambda_S=\Lambda$ (since variable $\Lambda$ is updated each time $\hat\lambda$ is updated). Moreover, at the end of the first cycle (line~\ref{line:endFor}) we have  $\Lambda=\sum_{i=\ell+1}^n x_i$. If $\ell=0$, then $\Lambda=1$ trivially (line~\ref{line:lambdaEq1}); otherwise, at the end of the second cycle (line~\ref{line:endWhile2}) and additional value of $\hat x_\ell=1-\sum_{i=\ell+1}^nx_i$ 
(line~\ref{line:updateXl}) is added to $\Lambda$. Hence, at the end of the algorithm
 $$\Lambda=\sum_{S\subseteq N}\hat \lambda_S=\sum_{i=\ell+1}^n x_i+\left(1-\sum_{i=\ell+1}^nx_i\right)=1.$$

Next, we verify that constraints \eqref{eq:hullPrimal_x} are satisfied. For $i\in \{1,\ldots,\ell-1\}$, at any point in the algorithm, we have that $\sum_{S\subseteq N: i\in S}\lambda_S=x_i-\hat x_i$. Since, at any point, $\hat x_i=\left(x_i-\Lambda\right)_+$ and $\Lambda=1$ at the end of the algorithm, it follows that  $\sum_{S\subseteq N: i\in S}\lambda_S=x_i$. For $i\in \{\ell+1,\ldots,n\}$ we also have that $\sum_{S\subseteq N: i\in S}\lambda_S=x_i-\hat x_i$, and $\hat x_i=0$ at the end (line~\ref{line:saturation}). Finally, for $i=\ell>0$, we have that $$\sum_{S\subseteq N: \ell\in S}\lambda_S=  \left(x_\ell-\left(1-\sum_{i=\ell+1}^n x_i\right)\right)+\left(1-\sum_{i=\ell+1}^nx_i\right)=x_\ell.$$

Finally, to check that $\hat\lambda$ satisfies complementary slackness, it suffices to observe that all updates of $\hat\lambda$ 
correspond to sets $S$ such that exactly one element of $S$ is greater than $\ell$ (line~\ref{line:updateS1}), or to sets $S$ with no element greater than $\ell$ and where $\ell\in S$ (line~\ref{line:updateS2}), where the corresponding dual constraints are satisfied at equality.

Therefore, we conclude that $\hat \lambda$ and $(\hat\mu,\hat \gamma)$ are an optimal primal-dual pair. Since problem \eqref{eq:hullDual} admits for any $x\in [0,1]$ an optimal solution of the form \eqref{eq:supermodularSorted}, it follows that those inequalities and bound constraints describe $\conv(G_\alpha)$. 
\end{proof}

Finally, we obtain the main result of this section: that the (nonlinear) lifted supermodular inequalities 
\begin{align} \label{eq:lifted-plus}
	t\geq & \max_{\alpha\in B^+}-\frac{\max_\alpha(S^+)^2}{4}-\sum\limits_{i\in N^+\setminus S^+}\frac{\left(\alpha_i^2-\max_\alpha(S^+)^2\right)_+}{4}x_i+\alpha'y, \ \ \forall S^+\subseteq N^+ \\
	\label{eq:lifted-minus}
		t\geq & \max_{\alpha\in B^-}-\frac{\max_\alpha(S^-)^2}{4}-\sum\limits_{i\in N^-\setminus S^-}\frac{\left(\alpha_i^2-\max_\alpha(S^-)^2\right)_+}{4}x_i+\alpha'y, \ \ \forall S^-\subseteq N^-
\end{align}
are sufficient to describe the closure of the convex hull of $X$. 
\begin{proposition}
	Lifted supermodular inequalities \eqref{eq:lifted-plus}--\eqref{eq:lifted-minus} and the bound constraints $0\leq x\leq 1$, $y \ge 0$ 
	describe $\clconv(X)$.
\end{proposition}
\begin{proof}
	Follows immediately from Proposition~\ref{prop:hullDiscrete} and Theorem~\ref{theo:hullSupermodular}.
\end{proof}

\begin{remark} 

	We end this section with the remark that optimization of a linear function over $X$ can be done easily using the projection function $g_\alpha$. Consider
	\[ \min \big \{ - \alpha' y + \beta' x  + t: (x,y,t) \in X \big \} \cdot
	\]
	 Projecting out the continuous variables using $g_\alpha$, the problem reduces to
	 \[ 
	 \min_{x \in \{0,1\}^N} \beta' x - \max_{i \in N} \{\alpha_i^2 x_i\}/4,
	 \]
	 which can be solved in linear time.
\end{remark}

\ignore{
\subsection{Convex hull of $X$}\label{sec:hullXg}

\todo{I think we should simply refer to the general result of Thm1 of [40]. The paper is already long and technical.}

Theorem~\ref{theo:hullRankOne} follows directly from Proposition~\ref{prop:hullDiscrete} and \cite[Theorem 1]{richard2010lifting}.
For completeness, we provide here a proof (which uses different arguments than the proof in \cite{richard2010lifting}). 

\begin{proof}[Proof of Theorem~\ref{theo:hullRankOne}]
	Consider optimization of an arbitrary linear function over $X$, i.e., 
	\begin{subequations}\label{eq:optimizationDiscrete}
	\begin{align}
	\min_{(x,y)\in \{0,1\}^N\times \R_+}\;&a'x-b'y+t\\
	\text{s.t.}\;&\left(y(N^+)-y(N^-)\right)^2\leq t\\
	& y_i(1-x_i)=0 \quad \forall i\in N,
	\end{align}
	\end{subequations}
 where $a,b\in \R^N$. Consider also its convex relaxation
	\begin{subequations}\label{eq:optimizationConvex}
	\begin{align}
	\min_{(x,y)\in [0,1]^N\times \R_+}\;&a'x-b'y+t\\
		\text{s.t.}\;&\max_{\alpha\in B}\;g_\alpha(S)+\sum\limits_{i\in N\setminus S}\rho_\alpha(i,S)x_i-\sum\limits_{i\in S}\rho_\alpha(i,N\setminus\{i\})(1-x_i)+\alpha'y\leq t \quad\forall S\subseteq N,
	\end{align}
	\end{subequations}
	where $\rho_\alpha$ correspond to the increments of function $g_\alpha$ defined in \eqref{eq:projectionGen}. We prove that problems \eqref{eq:optimizationDiscrete} and \eqref{eq:optimizationConvex} are equivalent. 
	
	First note that if \eqref{eq:optimizationDiscrete} is unbounded (iff $b\not\in B$) then \eqref{eq:optimizationConvex} is also unbounded since it is a relaxation. If \eqref{eq:optimizationDiscrete} is not unbounded (we assume without loss of generality that $b\in B^+$, as the case where $b\in B^-$ is identical), then noting that $t=\left(y(N^+)-y(N^-)\right)^2$ in an optimal solution and projecting out the continuous variables $y$ (using a process identical to the one used in \S\ref{sec:gGeneral}), we find that \eqref{eq:optimizationDiscrete} is equivalent to 
		\begin{align}\label{eq:optimizationDiscrete1}
		\min_{x\in \{0,1\}^N}\;&a'x -\frac{\max_{i\in N^+}\{b_i^2x_i\}}{4}. 
		\end{align}

Moreover, consider the relaxation of \eqref{eq:optimizationConvex} where, in all constraints, $\alpha=b$, i.e.,
\begin{subequations}\label{eq:optimizationConvex1}
	\begin{align}
	\min_{(x,y)\in [0,1]^N\times \R_+}\;&a'x-b'y+t\\
	\text{s.t.}\;&g_b(S)+\sum\limits_{i\in N\setminus S}\rho_b(i,S)x_i-\sum\limits_{i\in S}\rho_b(i,N\setminus\{i\})(1-x_i)+b'y\leq t \quad\forall S\subseteq N.
	\end{align}
\end{subequations}
The terms $b'y$ in the objective and constraints cancel out, and since $\rho_b(i,S)=0$ when $i\in N^-$, we find that \eqref{eq:optimizationConvex1} reduces to 
\begin{subequations}\label{eq:optimizationConvex2}
\begin{align}
\min_{x\in [0,1]^N}\;&a'x+t\\
	\text{s.t.}\;&g_b(S\cap N^+)+\sum\limits_{i\in N^+\setminus S}\rho_b(i,S\cap N^+)x_i-\sum\limits_{i\in S\cap N^+}\rho_b(i,N^+\setminus\{i\})(1-x_i)\leq t \quad\forall S\subseteq N.
\end{align}
\end{subequations}
Proposition~\ref{prop:hullDiscrete} states precisely that \eqref{eq:optimizationDiscrete1} and \eqref{eq:optimizationConvex2} are equivalent, i.e., \eqref{eq:optimizationConvex2} has an optimal integer solution that is also optimal for \eqref{eq:optimizationConvex1} (with same objective value).

\end{proof}
}

\section{Explicit form of the lifted supermodular inequalities}\label{sec:derivation}

In this section we derive explicit forms of the lifted supermodular inequalities \eqref{eq:lifted-plus}--\eqref{eq:lifted-minus}. In \S\ref{sec:originalSpace} we describe the inequalities in the original space of variables, and describe how to solve the separation problem. 
In \S\ref{sec:extendedSpace} we provide conic quadratic representable inequalities in an extended space, which can then be implemented with off-the-shelf conic solvers. 

\subsection{Inequalities and separation in the original space of variables}\label{sec:originalSpace}

\subsubsection{Lifted inequalities for $X$}\label{sec:originalSpaceXg}

\ignore{
Recall that for set $X$, it suffices to consider inequalities \eqref{eq:supermodular1} with $\alpha\in B_+$ where 
$$B^+=\left\{\alpha\in \R^N:\alpha_i\geq 0, \ i\in N^+;\alpha_j\leq 0, \ j\in N^-; \  \alpha_i+\alpha_j\leq 0, \ i\in N^+,j\in N^-\right\}.$$ 
For a fixed $S^+\subseteq N^+$ problem \eqref{eq:supermodular1} over $B^+$ reduces to 
\begin{align*}
t\geq \max_{\alpha\in B^+}-\frac{\max_\alpha(S^+)^2}{4}-\sum\limits_{i\in N^+\setminus S^+}\frac{\left(\alpha_i^2-\max_\alpha(S^+)^2\right)_+}{4}x_i+\alpha'y. 
\end{align*}
}

We first present the inequalities for the more general set $X$. 
Finding a closed form expression for the lifted supermodular inequalities \eqref{eq:lifted-plus} for all $S^+ \subseteq N^+$
amounts to solving the maximum lifting problem
\begin{align}\label{eq:liftingGen}
t\geq \max_{S^+\subseteq N^+, \alpha\in B^+}-\frac{\max_\alpha(S^+)^2}{4}-\sum\limits_{i\in N^+\setminus S^+}\frac{\left(\alpha_i^2-\max_\alpha(S^+)^2\right)_+}{4}x_i+\alpha'y. 
\end{align}

\begin{proposition}\label{prop:validGen}
	Given $(x,y)\in [0,1]^N\times \R_+^N$, if there exist disjoint sets $L,U\subseteq N^+$ such that 
	\begin{subequations}\label{eq:conditions}
		\begin{align}
		& 1-x(N^+\setminus L)\geq 0\label{eq:conditions_nonnegL}\\
		&\frac{y(L)}{1-x(N^+\setminus L)}< \frac{y_i}{x_i}, \quad i\in N^+\setminus L\label{eq:conditions_validL}\\
		&\frac{y(L)}{1-x(N^+\setminus L)}\geq \frac{y_i}{x_i}, \quad  i\in  L \label{eq:conditions_optL}\\
		& y(U)-y(N^-)\geq 0 \label{eq:conditions_nonnegU}\\
		&\frac{y(U)-y(N^-)}{x(U)}> \frac{y_i}{x_i}, \quad  i\in N^+\setminus U \label{eq:conditions_validU}\\
		&\frac{y(U)-y(N^-)}{x(U)}\leq \frac{y_i}{x_i}, \quad i\in  U\label{eq:conditions_optU}\\
		& \frac{y(L)}{1-x(N^+\setminus L)}< \frac{y(U)-y(N^-)}{x(U)},\label{eq:conditions_valid}
		\end{align}
	\end{subequations}
then inequality \eqref{eq:liftingGen} reduces to 
\begin{equation}\label{eq:originalSpaceGen}
t\geq \frac{y(L)^2}{1-x(N^+\setminus L)}+\sum\limits_{i\in N^+\setminus (L\cup U)}\frac{y_i^2}{x_i}+\frac{\big(y(U)-y(N^-)\big)^2}{x(U)};
\end{equation}
otherwise, inequality \eqref{eq:liftingGen} reduces to $t\geq \big(y(N^+)-y(N^-)\big)^2$.
\end{proposition}

Below we state two remarks on Proposition~\ref{prop:validGen}, and then we prove the result.
\begin{remark}\label{rem:valid}
	Inequalities \eqref{eq:originalSpaceGen} are neither valid for $\clconv(X)$ nor convex for all $(x,y)\in [0,1]^N\times \R_+^N$. Indeed, if condition \eqref{eq:conditions_nonnegL} is not satisfied, then \eqref{eq:originalSpaceGen} may not be convex. Moreover, suppose that $L=\{j\}$ and $U=\{k\}$ for some $j,k\in S^+$: note that setting $x_i=y_i=0$ for all $i\in N\setminus\{j,k\}$, $x_j=x_k=1$, $y_j,y_k>0$, and $t=(y_j+y_k)^2$ is feasible for $X$, but this point is cut off by inequality \eqref{eq:originalSpaceGen} since $\frac{y(L)^2}{1-x(N^+\setminus L)}=\frac{y_j^2}{1-x_k}=\infty$.
	
	In fact, if $(x,y,t)\in \clconv(X)$, then \eqref{eq:originalSpaceGen} is guaranteed to hold only when conditions \eqref{eq:conditions_nonnegL}, \eqref{eq:conditions_validL}, \eqref{eq:conditions_nonnegU}, \eqref{eq:conditions_validU}, and \eqref{eq:conditions_valid} are satisfied. Conditions \eqref{eq:conditions_optL} and \eqref{eq:conditions_optU} do not affect the validity of \eqref{eq:originalSpaceGen} but if they are not satisfied then \eqref{eq:originalSpaceGen} is weak, i.e., a stronger inequality can be obtained from another choice of $L$ and $U$. 
	\qed
\end{remark}

\begin{remark}\label{rem:sign}
If $y(N^+)<y(N^-)$, there exists no $L$ and $U$ satisfying condition \eqref{eq:conditions_nonnegU} in Proposition~\ref{prop:validGen}. However, in this case, the role of $N^+$ and $N^-$ can be interchanged to satisfy \eqref{eq:conditions_nonnegU}; interchanging $N^+$ and $N^-$ is equivalent to letting $\alpha\in B^-$. \qed
\end{remark}


\begin{proof}[Proof of Proposition~\ref{prop:validGen}]
   Let us define variables auxiliary variables $\beta,\gamma\in \R$ as $\beta=\max_\alpha(N^-)$ and $\gamma=\max_\alpha(S^+)$, respectively. Then, inequality \eqref{eq:liftingGen} reduces to
	\begin{subequations}\label{eq:liftingOriginalGen}
		\begin{align}
		t\geq \max_{S^+\subseteq N^+}\max_{\alpha,\beta,\gamma}& -\frac{\gamma^2}{4}-\sum\limits_{i\in N^+\setminus S^+}\frac{\left(\alpha_i^2-\gamma^2\right)_+}{4}x_i+\alpha'y \label{eq:liftingOriginalGen_obj}\\
		\text{s.t.}\;& \alpha_i\leq \gamma, \ \ \ \ \  \ \forall i\in S^+\label{eq:liftingOriginalGen_gamma}\\
		 &\alpha_i \leq \beta, \ \ \ \ \ \forall i\in N^-\label{eq:liftingOriginalGen_beta}\\
		& \beta \leq -\alpha_i, \ \ \forall i\in N^+ \label{eq:liftingOriginalGen_bounded}\\
		&\alpha\in \R^N,\;\gamma\in \R_+,\beta\in \R_-,\label{eq:liftingOriginalGen_bounds}
		\end{align}
	\end{subequations}
where 
constraints \eqref{eq:liftingOriginalGen_gamma} and \eqref{eq:liftingOriginalGen_beta} enforce the definitions of $\gamma$ and $\beta$, and constraints \eqref{eq:liftingOriginalGen_bounded} and \eqref{eq:liftingOriginalGen_bounds} enforce that $\alpha\in B^+$. 

First, observe that there exists an optimal solution of \eqref{eq:liftingOriginalGen} with $\gamma\leq \alpha_i$ for all $i\in N^+$: if $\alpha_i<\gamma$ for some $i\in N^+$, then setting $\alpha_i=\gamma$ results in a feasible solution with improved objective value. Therefore, the value of $S^+$ is completely determined by $\gamma$ since $S^+=\left\{i\in N^+: \alpha_i\leq \gamma\right\}$ 
Also note that $\alpha_i=\beta$ for all $i\in N^-$: if $\alpha_i<\beta$ for some $i\in N^-$, then setting $\alpha_i=\beta$ results in an improved objective value. We now consider two cases:

\paragraph{\textit{Case 1}} Suppose in an optimal solution of \eqref{eq:liftingOriginalGen} we have $\gamma=-\beta$, which implies that $\alpha_i=\gamma$ for all $i\in N^+$ and $\alpha_i=-\gamma$ for all $i\in N^-$. In this case, \eqref{eq:liftingOriginalGen} simplifies to 
$
	t\geq \max_{\gamma\in \R_+} \gamma\big(y(N^+)-y(N^-)\big)-\frac{\gamma^2}{4},
	$
	which, after optimizing for $\gamma$, further reduces to the original rank-one quadratic inequality
	$t\geq \big(y(N^+)-y(N^-)\big)^2.$

\paragraph{\textit{Case 2}} Now suppose $\gamma<-\beta$ in an optimal solution.  Let $L=\left\{i\in N^+:\alpha_i=\gamma \right\}$  and $U=\{i\in N^+:\alpha_i=-\beta \}$. Then, from the discussion above, \eqref{eq:liftingOriginalGen} reduces to 
\begin{subequations}\label{eq:liftingOriginalOptSets}
	\begin{align}
	t\geq \max_{\alpha,\beta,\gamma}\;& \gamma \cdot y(L)-\frac{\gamma^2}{4}\big(1-x(N^+\setminus L)\big)+\sum\limits_{i\in N^+\setminus (L\cup U)}\left(\alpha_iy_i-\frac{\alpha_i^2}{4}x_i\right)\notag\\
	&-\beta\big(y(U)-y(N^-)\big)-\frac{\beta^2}{4} x(U)\\
	\text{s.t.}\;& \gamma<\alpha_i< -\beta, \ \ \ \forall i\in N^+\setminus (L\cup U)\label{eq:liftingOriginalOptSets_constraint}\\
	&\alpha\in \R^N,\;\gamma,\beta\in \R_+.
	\end{align}
\end{subequations}
Observe that for $(L,U,\gamma)$ to correspond to an optimal solution, we must have $1-x(N^+\setminus L)\geq 0$ (otherwise, $\gamma$ can be increased to another $\alpha_i$ while improving the objective value) and $y(U)-y(N^-)\geq 0$ (otherwise, $-\beta$ can be decreased to another $\alpha_i$ while improving the objective value).
When both conditions are satisfied, from first-order conditions we see that $\alpha_i=2y_i/x_i$ for $i\in  N^+\setminus (L\cup U)$, $\gamma= 2 y(L)/\big(1-x(N^+\setminus L)\big)$ and $\beta=-2\big(y(U)-y(N^-)\big)/x(U)$, and \eqref{eq:liftingOriginalOptSets} simplifies to \eqref{eq:originalSpaceGen}.
The constraints $\gamma<\alpha_i$ are satisfied for all $i\in N^+\setminus (L\cup U)$ if and only if \eqref{eq:conditions_validL} hold, constraints $\alpha_i\leq -\beta$ are satisfied for all $i\in N^+\setminus (L\cup U)$ if and only if \eqref{eq:conditions_validU} hold, and constraint $\alpha<-\beta$, which may not be implied if $N^+\setminus (L\cup U)=\emptyset$, is satisfied if and only if \eqref{eq:conditions_valid} holds.

\ignore{
{\color{red}
Note that if the conditions above, i.e., \eqref{eq:conditions_validL},  \eqref{eq:conditions_validU}, \eqref{eq:conditions_valid}, \eqref{eq:conditions_nonnegL} and \eqref{eq:conditions_nonnegU}, are satisfied, then sets $L$ and $U$ and $(\alpha,\beta,\gamma)$ correspond to a \emph{feasible} solution for \eqref{eq:liftingOriginalGen}, and thus \eqref{eq:originalSpaceGen} is a lower bound of the optimal objective value of \eqref{eq:liftingOriginalGen}. Therefore, in this case, inequality \eqref{eq:originalSpaceGen} holds for $(x,y,t)\in \conv(X)$. However, this solution may be suboptimal, and another choices of $L$ and $U$ may yield better objective values. 
}
\todo{This paragraph is unclear to me\\AG: Agree. Added a little more explanation below, and red paragraph can be deleted.}
}

Finally, we verify that first order conditions are satisfied for $j\in L$, this is, setting $\alpha_j>\gamma$ results in a worse solution.
If condition \eqref{eq:conditions_optL}
$$\frac{y(L)}{1-x(N^+\setminus L)}\geq \frac{y_i}{x_i}, \quad \forall i\in  L$$
does not hold for some $j\in L$, then increasing $\alpha_j$ from $\gamma=2\frac{y(L)}{1-x(N^+\setminus L)}$ to $2y_j/x_j$ 
improves the objective value. Similarly, we verify that first order conditions for $j\in U$: if condition \eqref{eq:conditions_optU} 
$$\frac{y(U)-y(N^-)}{x(U)}\leq \frac{y_i}{x_i}, \quad \forall i\in  U$$
does not hold for some $j\in U$, then $\alpha_j$ can be decreased from $\beta = \frac{y(U)-y(N^-)}{x(U)}$
to improve the objective value.
\end{proof}

\subsubsection{Lifted inequalities for $X_+$}\label{sec:originalSpaceXp}
We now present the inequalities for $X_+$, which can be interpreted as a special cases of the inequalities for $X$ given in \S\ref{sec:originalSpaceXg}.
 Recall that for set $X_+$, the set $B$ used in \eqref{eq:supermodular1} is simply $B=\R^N$ (we can assume $B=\R_+^N$ without loss of generality) and a closed form expression for \eqref{eq:supermodular1} requires solving the lifting problem 
\begin{align}\label{eq:liftingPos}
t\geq \max_{S\subseteq N}\max_{\alpha\in \R_+^N}-\frac{\max_\alpha(S)^2}{4}-\sum\limits_{i\in N\setminus S}\frac{\left(\alpha_i^2-\max_\alpha(S)^2\right)_+}{4}x_i+\alpha'y. 
\end{align}

Note that in the proof of Proposition~\ref{prop:validGen}, set $U$ corresponds to the set of variables in $N^+$ where constraint $\alpha_i\leq-\max_{\alpha}(N^-)$ is tight in an optimal solution of \eqref{eq:liftingOriginalGen}. Intuitively, set $X_+$ can be interpreted as a special case of $X$ where $N^+=N$ and $N^-=\emptyset$, and such constraints can be dropped from the lifting problem. Therefore, we may assume $U=\emptyset$ in Proposition~\ref{prop:validGen}. Proposition~\ref{prop:validPos} formalizes this intuition; note however that it is slightly stronger as, unlike Proposition~\ref{prop:validGen}, it guarantees the existence of a set satisfying the conditions of the proposition.

\begin{proposition}\label{prop:validPos} 
	Given any $(x,y)\in [0,1]^N\times \R_+^N$, there exists a (possibly empty) set $L\subseteq N$ such that 
	\begin{subequations}\label{eq:conditionsPos}
		\begin{align}
		& 1-x(N\setminus L)\geq 0\label{eq:conditionsPos_nonnegL}\\
		&\frac{y(L)}{1-x(N\setminus L)}< \frac{y_i}{x_i}, \quad i\in N\setminus L\label{eq:conditionsPos_validL}\\
		&\frac{y(L)}{1-x(N\setminus L)}\geq \frac{y_i}{x_i}, \quad  i\in  L \label{eq:conditionsPos_optL}
		\end{align}
	\end{subequations}
	and inequality \eqref{eq:liftingPos} reduces to 
	\begin{equation}\label{eq:originalSpacePos}
	t\geq \frac{y(L)^2}{1-x(N\setminus L)}+\sum\limits_{i\in N\setminus L}\frac{y_i^2}{x_i}.
	\end{equation}
\end{proposition}
The proof of Proposition~\ref{prop:validPos} is given in Appendix~\ref{sec:appendix}. 

\setcounter{example}{0}
\begin{example}[cont] \label{ex1:cont}
	Consider $X_+$ with $n=3$, and assume $x_2=0.6$, $x_3=0.3$, $y_2=0.5$ and $y_3=0.2$. Note that $y_2/x_2\approx 0.83>0.67\approx y_3/x_3$. We now compute the minimum values $t$ such $(x,y,t)\in \clconv(X_+)$, for different values of $(x_1,y_1)$. 
	
	$\bullet$ Let $(x_1,y_1)=(0.01,1)$ and $y_1/x_1=100$. Then $L=\emptyset$ satisfies all conditions \eqref{eq:conditions}: $x(N)=0.91<1$, conditions \eqref{eq:conditionsPos_validL} are trivially satisfied since $y(\emptyset)=0$, and conditions \eqref{eq:conditionsPos_optL} are void. In this case, we find that $(x,y,t)\in \clconv(X_+)$ iff $t\geq 1^2/0.01+0.5^2/0.6+0.2^2/0.3\approx 100.55$. In contrast, $(x,y,t)\in \clconv(X_f)$ iff $t\geq \left(0.01+0.5+0.2\right)^2/0.91\approx 3.18$.
	
	$\bullet$ Let $(x_1,y_1)=(0.1,0.5)$ and $y_1/x_1=5$. Then $L=\{3\}$ satisfies all conditions \eqref{eq:conditions}: $x_1+x_2=0.7<1$, $0.2/0.3\approx0.67<y_2/x_2$ and $0.2/0.3\approx 0.67= y_3/x_3$. In this case,  $(x,y,t)\in \clconv(X_+)$ iff $t\geq 0.2^2/0.3+0.5^2/0.1+0.5^2/0.6\approx 3.05$. In contrast, $(x,y,t)\in \clconv(X_f)$ iff $t\geq \left(0.5+0.5+0.2\right)^2/1= 1.44$.
	
	$\bullet$ Let $(x_1,y_1)=(0.4,0.1)$ and $y_1/x_1=0.25$. Then $L=\{1,3\}$ satisfies all conditions \eqref{eq:conditions}: $x_2=0.6<1$, $(0.1+0.2)/0.4=0.75<y_2/x_2$ and $(0.1+0.2)/0.4=0.75\geq y_3/x_3$. In this case,  $(x,y,t)\in \clconv(X_+)$ iff $t\geq (0.1+0.2)^2/0.4+0.5^2/0.6\approx 0.642$. In contrast, $(x,y,t)\in \clconv(X_f)$ iff $t\geq \left(0.1+0.5+0.2\right)^2= 0.640$.
	
	$\bullet$ Let $(x_1,y_1)=(0.5,0.2)$ and $y_1/x_1=0.4$. Then $L=\{1,2,3\}$ satisfies all conditions \eqref{eq:conditions}: \eqref{eq:conditionsPos_nonnegL} is trivially satisfied, \eqref{eq:conditionsPos_validL} is void and $(0.2+0.5+0.2)/1=0.9\geq y_2/x_2$. In this case,  $(x,y,t)\in \clconv(X_+)$ iff $t\geq (0.2+0.5+0.2)^2= 0.81$, which coincides with $\clconv(X_f)$ and the natural inequality $t\geq y(N)^2$.
	
	Figure~\ref{fig:example} plots the minimum values of $t$ as a function of $(x_1,y_1)$ for $\clconv(X_f)$ and $\clconv(X_+)$. \qed
	
	\end{example}

\subsubsection{Separation}\label{sec:separation}

We now consider the separation problem for inequalities \eqref{eq:liftingGen} and \eqref{eq:liftingPos}, i.e., 
given a point $(\bar x,\bar y)\in [0,1]^N\times \R_+^N$, finding sets $L,U\subseteq N^+$ satisfying the conditions in Proposition~\ref{prop:validGen} or finding $L\subseteq N$ satisfying the conditions in Proposition~\ref{prop:validPos}, respectively. 

\subsubsection*{\textbf{Separation for \eqref{eq:liftingGen}}} First, as pointed out in Remark~\ref{rem:sign}, we verify whether $\bar y(N^+)\geq \bar y(N^-)$ or $\bar y(N^+)<\bar y(N^-)$; in the first case, we use directly the conditions in Proposition~\ref{prop:validGen}, and in the second one, we interchange the roles of $N^+$ and $N^-$ so that $\bar y(N^+)\geq \bar y(N^-)$. Next, index the variables so that $\bar y_1/ \bar x_1\leq \bar y_2/ \bar x_2\leq \ldots\leq \bar y_m/\bar x_m$, where $m=|N^+|$, 
which is done in $O(m\log m)$ by sorting. It follows from the conditions in Proposition~\ref{prop:validGen} that if such sets $L,U\subseteq N^+$ exist, then $L=\left\{i\in N^+: i\leq \ell\right\}$ and $U=\left\{i\in N^+: i\geq u\right\}$ for some $\ell,u\in \{1,\ldots,m\}$ with $\ell<u$. 
Therefore, one can simply enumerate all $m(m-1)/2$ possible values of $(\ell,u)$ and verify whether conditions \eqref{eq:conditions} are satisfied for each candidate set $L$ and $U$. Hence, the separation algorithm runs in $O(n^2)$ time.

\subsubsection*{\textbf{Separation for \eqref{eq:liftingPos}}} First, we index the variables so that $\bar y_1/ \bar x_1\leq  \bar y_2/ \bar  x_2\leq \ldots\leq \bar y_n/\bar x_n$; the indexing process can be accomplished in $O(n\log n)$ time by sorting. It follows from the conditions in Proposition~\ref{prop:validPos} that $L=\left\{i\in N^+: i\leq \ell\right\}$ for some $\ell\in \{1,\ldots,n\}$. Therefore, one can simply enumerate all $n$ possible values of $\ell$ and verify whether conditions \eqref{eq:conditionsPos} are satisfied for each candidate set $L$. Since the sorting step dominates the complexity, the separation algorithm runs in $O(n\log n)$.

\subsection{Conic quadratic valid inequalities in an extended formulation}\label{sec:extendedSpace}
Inequalities \eqref{eq:originalSpaceGen} and \eqref{eq:originalSpacePos}
given in the original space of variables are valid only over restricted parts of the domain. They are neither valid nor convex over the entire domain of the variables, e.g., \eqref{eq:originalSpaceGen} is not convex whenever $x(N^+\setminus L)\geq 1$. 
Thus, such inequalities are difficult to utilize directly by the optimization solvers.
In order to address this challenge, in this section, we give valid conic quadratic reformulations in an extended space, 
which can be readily used by conic quadratic solvers. 

 For a partitioning $(L,R,U)$ of $N^+$ consider the inequality
\small\begin{subequations}\label{eq:extendedGen}
	\begin{align}t\geq \min_{\lambda,\mu,\zeta}\;&\frac{\Big(y(L)-\lambda_0\Big)^2}{1-x(R)-x(U)+\mu(R)+\mu_0}+\sum_{i\in R}\frac{(y_i-\lambda_i)^2}{x_i-\mu_i}+\frac{\Big(y(U)-y(N^-)+\lambda_0+\lambda(R)+\zeta\Big)^2}{x(U)-\mu_0}\label{eq:extendedGen_obj}\\
	\text{s.t.}\;&1-x(R)-x(U)+\mu(R)+\mu_0\geq 0\label{eq:extendedGen_nonNeg}\\
	&\mu_i\leq x_i, \quad i\in R\label{eq:extendedGen_bounds}\\
	&\mu_0\leq x(U)\\
	&\lambda,\mu\in \R_+^{R},\; \lambda_0,\mu_0,\zeta\in \R_+.\label{eq:extendedGen_vars}
	\end{align}
\end{subequations}\normalsize
Note that each inequality \eqref{eq:extendedGen} requires $O(n)$ additional variables and constraints. Moreover, although not explicitly enforced, it is easy to verify that there exists an optimal solution to \eqref{eq:extendedGen} with $\lambda_i\leq y_i$ and $\lambda_0\leq y(L)$. Inequalities \eqref{eq:extendedGen} are convex as they involve linear constraints and sums of ratios of convex quadratic terms and nonnegative linear terms, thus  conic quadratic representable \cite{Alizadeh2003,Lobo1998}. 
We show, in Proposition~\ref{prop:equivalence}, that inequalities \eqref{eq:extendedGen} imply
the strong formulations described in Proposition~\ref{prop:validGen}, 
and, in Proposition~\ref{prop:correct}, that they are valid for $X$. 

\begin{proposition}\label{prop:equivalence}
	If conditions \eqref{eq:conditions_nonnegL}, \eqref{eq:conditions_validL}, \eqref{eq:conditions_nonnegU}, \eqref{eq:conditions_validU} and \eqref{eq:conditions_valid} are satisfied, then $\lambda=\mu = 0$ and $\lambda_0=\mu_0=\zeta=0$ in an optimal solution of \eqref{eq:extendedGen}.
\end{proposition}
\begin{proof}
	Observe that $\zeta$ does not appear in any constraint of \eqref{eq:extendedGen}. Thus,
	since $y(U)-y(N^-)\geq 0$ and $\lambda,\lambda_0\geq 0$, it follows that $\zeta=0$ in an optimal solution. Moreover, since \eqref{eq:conditions_nonnegL} is satisfied, then setting $\mu=0$ is feasible for \eqref{eq:extendedGen}. 
	Finally, find that KKT conditions are satisfied for $\lambda,\mu=0$ and $\lambda_0=\mu_0=0$ if
	\begin{align*}
	&-\frac{y(L)}{1-x(R)-x(U)}+\frac{y(U)-y(N^-)}{x(U)}\geq 0\tag{$\lambda_0$}\\
	&-\left(\frac{y(L)}{1-x(R)-x(U)}\right)^2+\left(\frac{y(U)-y(N^-)}{x(U)}\right)^2\geq 0\tag{$\mu_0$}\\
	&-\frac{y_i}{x_i}+\frac{y(U)-y(N^-)}{x(U)}\geq 0,\quad\forall i\in R\tag{$\lambda_i$}\\
	&-\left(\frac{y(L)}{1-x(R)-x(U)}\right)^2+\left(\frac{y_i}{x_i}\right)^2\geq 0, \quad\forall i\in R\tag{$\mu_i$}.
	\end{align*}
	The KKT condition above for $\lambda_0$ is precisely \eqref{eq:conditions_valid}. Since $x(R)+x(U)\leq 1$ by \eqref{eq:conditions_nonnegL}, and $y(U)-y(N^-)\geq 0$ by \eqref{eq:conditions_nonnegU}, the KKT condition for $\mu_0$ is equivalent to $\frac{y(L)}{1-x(R)-x(U)}+\frac{y(U)-y(N^-)}{x(U)}\geq 0$, and thus reduces to \eqref{eq:conditions_valid}. The KKT conditions for $\lambda_i$ are satisfied since \eqref{eq:conditions_validU} holds. Finally,
	the KKT conditions for $\mu_i$ can be equivalently stated as $\frac{y(L)}{1-x(R)-x(U)}\leq \frac{y_i}{x_i}$ (since $x(R)+x(U)\leq 1$ and $x,y \ge 0$), which are satisfied since \eqref{eq:conditions_validL} holds.
\end{proof}
Note that when $\lambda=\mu =0$ and $\lambda_0=\mu_0=\zeta=0$, inequality \eqref{eq:extendedGen} reduces to \eqref{eq:originalSpaceGen}. Thus, if sets $L,U$ satisfy the conditions of Proposition~\ref{prop:validGen} for a given $(x,y)$, then there exists $t\in \R$ such that $(x,y,t)\in \conv(X)$ and \eqref{eq:extendedGen} holds at equality. 
It remains to prove that inequalities \eqref{eq:extendedGen} do not cut-off any points in $X$ for any choice of partition $(L,R,U)$. 

\begin{proposition}\label{prop:correct}
 For any partitioning $(L,R,U)$ of $N^+$, inequalities \eqref{eq:extendedGen} are valid for $X$.
\end{proposition}
\begin{proof}
	It suffices to show that for any $(x,y)\in X$, i.e., $x_i\in\{0,1\}$ and $x_i(1-y_i)=0$ for all $i\in N$, there exists $(\lambda,\mu,\lambda_0,\mu_0, \zeta)$ satisfying \eqref{eq:extendedGen_nonNeg}--\eqref{eq:extendedGen_vars} such that inequality \eqref{eq:extendedGen_obj} is valid. We prove the result by cases.
	
	\paragraph{\textit{Case 1}} $y(N^+)< y(N^-)$: In this case, we can set $\lambda_i=y_i$ and $\mu_i=x_i$ for $i\in R$, $\lambda_0=y(L)$, $\mu_0=x(U)$, $\zeta=y(N^-)-y(U)-y(L)-y(R)$, and inequality \eqref{eq:extendedGen_obj} reduces to $t\geq 0$, which is valid. 
	
	\paragraph{\textit{Case 2}} $y(N^+)\geq y(N^-)$, $x(R)=0$ and $x(U)=0$: In this case, $y_i=0$, $i\in R\cup U$. Setting $\mu_i=\lambda_i=0$ for $i\in R$, $\lambda_0=y(N^-)$, $\mu_0=0$ and $\zeta=0$, we find that inequality \eqref{eq:extendedGen_obj} reduces to
	$t\geq \big(y(L)-y(N^-)\big)^2=\big(y(N^+)-y(N^-)\big)^2$, which is valid. 
	
	\paragraph{\textit{Case 3}} $y(N^+)\geq y(N^-)$ and $x(U)\geq 1$: Setting $\lambda_i=y_i$ and $\mu_i=x_i$ for $i\in R$, $\lambda_0=y(L)$, $\mu_0=x(U)-1$, and $\zeta=0$,  inequality \eqref{eq:extendedGen_obj} reduces to
	$t\geq \big(y(N^+)-y(N^-)\big)^2$, which is valid. 
	
	\paragraph{\textit{Case 4}} $y(N^+)\geq y(N^-)$, $x(U)= 0$, $x(R)\geq 1$, $y(N^-)<y_i$ for all $i\in R$ and $y(N^-)<y(L)$: In this case, $y_i=0$, for all $i\in U$ and $x_i=1$, for all $i\in R$, we can set $\mu_0=0$, and inequality \eqref{eq:extendedGen} reduces to 
	\begin{subequations}\label{eq:extendedGenSimple}
		\begin{align}t\geq \min_{\lambda,\mu}\;&\frac{\Big(y(L)-\lambda_0\Big)^2}{1-|R|+\mu(R)}+\sum_{i\in R}\frac{(y_i-\lambda_i)^2}{1-\mu_i}\label{eq:extendedGenSimple_obj}\\
		\text{s.t.}\;&1-|R|+\mu(R)\geq 0\label{eq:extendedGenSimple_nonNeg}\\
		&\mu_i\leq 1 \quad \forall i\in R\label{eq:extendedGenSimple_bounds}\\
		& -y(N^-)+\lambda_0+\lambda(R)+\zeta=0\label{eq:extendedGenSimple_zeta0}\\
		&\lambda,\mu\in \R_+^{R},\; \lambda_0,\zeta\in \R_+.\label{eq:extendedGenSimple_vars}
		\end{align}
	\end{subequations}
Constraint \eqref{eq:extendedGenSimple_zeta0} is obtained since the denominator of the third term in \eqref{eq:extendedGen_obj} is zero, thus constraining the numerator to vanish as well. Moreover, since variable $\zeta\geq 0$ only appears in \eqref{eq:extendedGenSimple_zeta0}, after projecting $\zeta$ out we find that constraint \eqref{eq:extendedGenSimple_zeta0} reduces to 
\begin{equation}
\lambda_0+\lambda(R)\leq y(N^-)\label{eq:extendedGenSimple_zeta}.
\end{equation}
Note that constraint \eqref{eq:extendedGenSimple_zeta}, and assumptions $y(N^-)<y_i$ for all $i\in R$ and $y(N^-)<y(L)$, imply that $\lambda_i\leq y_i$ and $\lambda_0\leq y(L)$. Observe that we can set 
\begin{equation*}
\mu_i= 1-\frac{y_i-\lambda_i}{y(L)+y(R)-\lambda(R)-\lambda_0}\quad \forall i\in  R.
\end{equation*}
Indeed, for any feasible $\lambda$, $y(L)+y(R)-\lambda(\bar R)-\lambda_0\geq y(L)+y(R)-y(N^-)\geq 0$; thus $\mu_i\leq 1$. Moreover, $$\frac{y_i-\lambda_i}{y(L)+y(R)-\lambda(R)-\lambda_0}\leq \frac{y_i-\lambda_i}{y(L)+y(R\setminus i)+y_i-\lambda_i}\leq 1$$
thus $\mu_i\geq 0$. For this choice of $\mu$, we find that 
$$1-|R|+\mu(R)=\frac{y(L)-\lambda_0}{y(L)+y(R)-\lambda(R)-\lambda_0}\geq 0.$$

Finally, substituting $1-|R|+\mu(R)$ and $\mu_i$ in \eqref{eq:extendedGenSimple_obj} with their respective values,  \eqref{eq:extendedGenSimple_obj} reduces to 
\begin{align*}t\geq &\min_{\lambda}\;\Big(y(L)-\lambda_0\Big)\Big(y(L)+y(R)-\lambda(R)-\lambda_0\Big)\\
&+\Big(y(L)+y(R)-\lambda(R)-\lambda_0\Big)\sum_{i\in  R}(y_i-\lambda_i)\\
\Leftrightarrow t\geq &\min_{\lambda}\;\Big(y(L)+y(R)-\lambda(R)-\lambda_0\Big)^2=\Big(y(L)+y(R)-y(N^-)\Big)^2,
\end{align*}
and since $y(L)+y(R)=y(N^+)$, this inequality is valid.

\paragraph{\textit{Case 5}} $y(N^+)\geq y(N^-)$, $x(U)= 0$, $x(R)\geq 1$, $y(N^-)< y(L)$ but $y(N^-)\geq y_j$ for some $j\in R$:  In this case, $y_i=0$ for all $i\in U$, and we set $\mu_0=0$. Note that, in \eqref{eq:extendedGen}, we can set $\lambda_j=y_j$ and $\mu_j=x_j$, resulting in the inequality
		\begin{align*}t\geq \min_{\lambda,\mu,\zeta}\;&\frac{\Big(y(L)-\lambda_0\Big)^2}{1-x(R\setminus j)-x(U)+\mu(R\setminus j)}+\sum_{i\in R\setminus j}\frac{(y_i-\lambda_i)^2}{x_i-\mu_i}\\
		&+\frac{\Big(y(U)-y(N^-)+y_j+\lambda_0+\lambda(R\setminus j)+\zeta\Big)^2}{x(U)}\\
		\text{s.t.}\;&1-x(R\setminus j)-x(U)+\mu(R\setminus j)\geq 0\\
		&\mu_i\leq x_i \quad \quad\forall i\in R\setminus j\\
		&\lambda,\mu\in \R_+^{R\setminus j},\; \lambda_0,\zeta\in \R_+.
		\end{align*}
This inequality of the same form as \eqref{eq:extendedGen} but with $\hat R=R\setminus j$ and $\hat y(N^-)=y(N^-) - y_j$. After repeating sequentially this process so that $\lambda_i=y_i$ and $\mu_i=x_i$ for some subset $T \subseteq R$, such that $y(N^-)-y(T)\leq y_i$ for all $i\in R\setminus T$,
and applying a similar strategy as in \textit{Case 4}, we obtain either an inequality of the form
	$$t\geq \Big(y(L)+y(R\setminus T)-\big(y(N^-)- y(T)\big)\Big)^2=\Big(y(N^+)-y(N^-)\Big)^2,$$
	 which is valid.
	 
\paragraph{\textit{Case 6}}  $y(N^+)\geq y(N^-)$, $x(U)= 0$, $x(R)\geq 1$, and $y(N^-)\geq y(L)$: In this case, we can set $\lambda_0=y(L)$, $\mu_0=0$, and \eqref{eq:extendedGen} reduces to	
	\begin{align*}t\geq \min_{\lambda,\mu}\;&\sum_{i\in R}\frac{(y_i-\lambda_i)^2}{x_i-\mu_i}\\
	\text{s.t.}\;&1-x(R)+\mu(R)\geq 0\\
	&\mu_i\leq x_i \quad \quad\forall i\in R\\
	&\lambda(R)\leq y(N^-)-y(L)\\ 
	&\lambda,\mu\in \R_+^{R}.
	\end{align*}
Moreover, if $y(N^-)-y(L)\geq y_j$ for some $j\in R$, then we can set $\lambda_j=y_j$, $\mu_j=y_j$ as done in \textit{Case 5}. After repeating this process, we obtain an inequality of the form 
\begin{subequations}\label{eq:extendedGenSimple3}
\begin{align}t\geq \min_{\lambda,\mu}\;&\sum_{i\in R\setminus T}\frac{(y_i-\lambda_i)^2}{x_i-\mu_i}\label{eq:extendedGenSimple3_obj}\\
\text{s.t.}\;&1-x(R\setminus T)+\mu(R\setminus T)\geq 0 \label{eq:extendedGenSimple3_nonneg}\\
&\mu_i\leq x_i \quad \forall i\in R\setminus T \label{eq:extendedGenSimple3_bound}\\
& \lambda(R\setminus T)\leq y(N^-)-y(L)-y(T)\label{eq:extendedGenSimple3_zeta}\\
&\lambda,\mu\in \R_+^{R\setminus T},
\end{align}
\end{subequations}
where $y(N^-)-y(L)-y(T)<y_i$ for all $i\in R\setminus T$, and therefore $x_i=1$ for all $i\in R\setminus T$.  

Note that constraint \eqref{eq:extendedGenSimple3_zeta} and $y(N^-)-y(L)-y(T)<y_i$ imply that $\lambda_i< y_i$ in any feasible solution.
Then, for all $i \in R\setminus T$, we can set
 $$\mu_i=x_i-\frac{y_i-\lambda_i}{y(R\setminus T)-\lambda (R\setminus T)}.$$
 Clearly, $\mu_i\leq x_i$. Moreover, for all $i \in R\setminus T$, 
  $$\frac{y_i-\lambda_i}{y(R\setminus T)-\lambda (R\setminus T)}\leq \frac{y_i-\lambda_i}{y(R\setminus (T \cup i))+y_i-\lambda_i}\leq 1=x_i,$$
 thus $\mu_i\geq 0$. Finally, 
 $$1-x(R\setminus \Lambda)+\mu(R\setminus T)=1-\frac{y(R\setminus T)-\lambda(R\setminus \lambda)}{y(R\setminus T)-\lambda (R\setminus T)}=0,$$
 and constraint \eqref{eq:extendedGenSimple3_nonneg} is satisfied.
Substituting $x_i-\lambda_i$, $ i \in R \setminus T$, with their explicit form in \eqref{eq:extendedGenSimple3_obj}, we find the equivalent form
\begin{align*}
t\geq \min_{\lambda}\;\Big(y(R\setminus T)-\lambda (R\setminus T)\Big)\sum_{i\in R\setminus T}(y_i-\lambda_i)=&\min_{\lambda}\;\Big(y(R\setminus T)-\lambda (R\setminus T)\Big)^2\\
=&\Big(y(N^+)-y(N^-)\Big)^2,
\end{align*}
which is valid.
\end{proof}

\ignore{

	Specifically, for any $L^+,U^+\subseteq N^+$ with $L^+\cap U^+=\emptyset$, let $S^+=N^+\setminus(L^+\cup U^+)$ and consider the SOCP-representable inequality (in an extended formulation) \todo{While we want to state the results as clearly as possible in the introduction, it is impossible for a reader to make sense of these formulations without reading the rest of the paper. Therefore, I am skeptical of the effectiveness of listing the formulations here. It may turn off some readers...I would rather emphasize how to make use the results in the introduction as motivation (reformulating $yQy$ through rank-one quadratics ).\\
		AG: I may agree. Since that change requires updating section 4 as well, I have not done it yet.}
	\small\begin{subequations}\label{eq:extendedGenIntroPlus}
		\begin{align}t\geq \min_{\lambda,\mu,\zeta}\;&\frac{\left(\sum\limits_{i\in L^+}y_i-\lambda_0\right)^2}{1-\sum\limits_{i\in S^+\cup U^+}x_i+\sum\limits_{i\in S^+}\mu_i+\mu_0}+\sum\limits_{i\in S^+}\frac{(y_i-\lambda_i)^2}{x_i-\mu_i}+\frac{\left(\sum\limits_{i\in U^+}y_i-\sum\limits_{i\in N^-}y_i+\lambda_0+\sum\limits_{i\in S^+}\lambda_i+\zeta\right)^2}{\sum\limits_{i\in U^+}x_i-\mu_0}\\
			\text{s.t.}\;&1-\sum_{i\in S^+\cup U^+}x_i+\sum_{i\in S^+}\mu_i+\mu_0\geq 0\\
			&\mu_i\leq x_i \quad \quad\forall i\in S^+\\
			&\mu_0\leq \sum_{i\in U^+}x_i\\
			&\lambda,\mu\in \R_+^{S^+},\; \lambda_0,\mu_0,\zeta\in \R_+;
		\end{align}
	\end{subequations}\normalsize
	and for any $L^-,U^-\subseteq N^-$ with $L^-\cap U^-=\emptyset$, let $S^-=N^-\setminus(L^+\cup U^+)$, and 
	
}

To derive the corresponding lifted inequalities for $B^-$, it suffices to
interchange $N^+$ and $N^-$. Therefore, for a partitioning $(L, R, U)$ of  $N^-$,
we find the conic quadratic inequalities: 
\small\begin{subequations}\label{eq:extendedGenMinus}
	\begin{align}t\geq \min_{\lambda,\mu,\zeta}\;&\frac{\left(y(L) -\lambda_0\right)^2}{1- x(R) + x(U) + \mu(R) + \mu_0}+
		\sum\limits_{i\in R}\frac{(y_i-\lambda_i)^2}{x_i-\mu_i}+\frac{\left(y(U)-y(N^+) +\lambda_0+\lambda(R)+\zeta\right)^2}{x(U)-\mu_0}\\
		\text{s.t.}\;&1- x(R) -x(U)+ \mu(R)+\mu_0\geq 0\\
		&\mu_i\leq x_i, \quad  i\in R\\
		&\mu_0\leq x(U)\\
		&\lambda,\mu\in \R_+^{R},\; \lambda_0,\mu_0,\zeta\in \R_+.
	\end{align}
\end{subequations}\normalsize

The main result of the paper is stated below.
\begin{theorem}\label{theo:hullGeneral} 
	$\clconv(X)$ is given by bound constraints $0 \le x\le 1$, $y \ge 0$, and inequalities \eqref{eq:extendedGen} and \eqref{eq:extendedGenMinus}. 
\end{theorem}

\ignore{
\todo{Do we need to say closure? Is conv hull not closed?\\
	AG: Consider $X_-^2$, which is a special case of $X$. Point $(x_1,x_2,y_1,y_2,t)=(0,0,1,1,0)$ is not a convex combination of points in $X_-^2$, and thus not in $\conv(X_-^2)$. However, $(0,0,1,1,0)=\lim_{\lambda\to 0} (1-\lambda)(0,0,0,0,0)+\lambda (1,1,1/\lambda,1/\lambda,0)$, thus this point belongs to the closure $\clconv(X_-^2)$.}
}

\ignore{
To derive conic quadratic inequalities for $X_+$ using the inequalities for $X$, it suffices to write $X_+$ as 
\begin{align*}X_+=\Big\{(x,y,t)&\in \{0,1\}^{n+2}\times \R_+^{n+2}\times \R_+: \left(\sum_{i=1}^ny_i+y_{n+1}-y_{n+2}\right)^2\leq t,\\
	& y_{n+1}=y_{n+2}=x_{n+1}=x_{n+2}=0,\; y_i(1-x_i)=0 \ \forall i\in \{1,\ldots,n+2\}\Big\}
\end{align*}
Then, letting $N^+=\left\{1,\ldots,n+1\right\}$, $N^-=\{n+2\}$, and $U=\{n+1\}$,  for $L\subseteq N$ 
}

For the positive case of $X_+$ with $N^- =\emptyset$, for a partitioning $(L, R)$ of $N$, 
inequalities \eqref{eq:extendedGen}  reduce to
\begin{subequations}\label{eq:extendedPos}
	\begin{align}t\geq \min_{\mu}\;&\frac{y(L)^2}{1-x(R)+\mu(R)}+\sum_{i\in R}\frac{y_i^2}{x_i-\mu_i}\label{eq:extendedPos_obj}\\
	\text{s.t.}\;&1-x(R)+\mu(R)\geq 0\label{eq:extendedPos_nonNeg}\\
	&\mu_i\leq x_i,  \quad \quad i\in R \label{eq:extendedPos_bounds}\\
	&\mu\in \R_+^{R}.\label{eq:extendedPos_vars}
	\end{align}
\end{subequations}
Note that each inequality \eqref{eq:extendedPos} also requires $O(n)$ additional variables and constraints but is significantly simpler compared to \eqref{eq:extendedGen}.

\begin{theorem}\label{theo:hullPosiitve} 
	$\clconv(X_+)$ is given by bound constraints $0 \le x \le 1$, $y \ge 0$, and inequalities \eqref{eq:extendedPos}.
\end{theorem}

\ignore{

		\small\begin{subequations}\label{eq:extendedPosIntroPlus}
		\begin{align}t\geq \min_{\mu}\;&\frac{\left(\sum\limits_{i\in L}y_i\right)^2}{1-\sum\limits_{i\in N\setminus L}(x_i-\mu_i)}+\sum\limits_{i\in N\setminus L}\frac{y_i^2}{x_i-\mu_i}\\
			\text{s.t.}\;&1-\sum\limits_{i\in N\setminus L}(x_i-\mu_i)\geq 0\\
			&\mu_i\leq x_i \quad \quad\forall i\in N\setminus L\\
			&\mu\in \R_+^{N\setminus L}
		\end{align}
	\end{subequations}\normalsize
	
	\todo{Need closure ?\\
		AG: Yes. We are implicitly working here with the extended real line, and have points where $x_i=0$, $y_i>0$ and $t=\infty$ that are not convex combinations of points in $X_+$. In this case, it is quite possible that if we remove those points we actually get the convex hull, but that would require a new proof.}

Note that $X_+$ can be equivalently written as 
\begin{align*}X_+=\Big\{(x,y,t)&\in \{0,1\}^{n+2}\times \R_+^{n+2}\times \R_+: \left(\sum_{i=1}^ny_i+y_{n+1}-y_{n+2}\right)^2\leq t,\\
& y_{n+1}=y_{n+2}=x_{n+1}=x_{n+2}=0,\; y_i(1-x_i)=0 \ \forall i\in \{1,\ldots,n+2\}\Big\}.\end{align*}
Then inequalities \eqref{eq:extendedGen} with $N^+=\left\{1,\ldots,n+1\right\}$, $N^-=\{n+2\}$ and $U=\{n+1\}$ reduce to \eqref{eq:extendedPos} since $x_{n+1}=0$ forces $\mu_0=\lambda_0=\zeta=\lambda(S)=0$. Therefore, the validity of \eqref{eq:extendedPos} follows directly from Proposition~\ref{prop:correct}. It can also be shown that inequalities \eqref{eq:extendedPos} can be used to represent the strong formulations given in Proposition~\ref{prop:validPos}; the proof is identical to the proof of Proposition~\ref{prop:equivalence} and is omitted for brevity. 
}

\section{Computational experiments}\label{sec:computations}

In this section, we test the computational effectiveness of the conic quadratic inequalities given in \S\ref{sec:extendedSpace} in solving convex quadratic minimization problems with indicators. 
In particular, we solve portfolio optimization problems with fixed-charges.
All experiments are run with CPLEX 12.8 solver on a laptop with a 1.80GHz Intel\textregistered Core\textsuperscript{TM} i7 CPU and 16 GB main memory on a single thread. We use CPLEX default settings but turn on the numerical emphasis parameter, unless stated otherwise. The data for the instances and problem formulations in \texttt{.lp} format can be found online at 
{\small \url{https://sites.google.com/usc.edu/gomez/data}.}

\ignore{

First, in \S\ref{sec:comp_unconstr}, we test the formulations in optimization problems over $X_g$ and $X_p$, in order to evaluate the strengthening obtained without of the effect of additional constraints not accounted for in the derivation of the inequalities. Then, in \S\ref{sec:comp_portfolio}, we test the strength of the formulations in portfolio optimization instances. 

Next, in \S\ref{sec:branchBound}, we briefly comment on the use of the new formulations in a branch-and-bound algorithm. Finally, in \S\ref{sec:future}, we give some pointers for future work. 

\subsection{Unconstrained problems}\label{sec:comp_unconstr}
In this section, we report on computational results evaluating the strength of the formulations for unconstrained problem instances.

\subsubsection{Instances}\label{sec:unconstr_instances}We consider optimization problems of the form 
\begin{subequations}\label{eq:unconstrained}
	\begin{alignat}{1}
	\min_{x,y}\;& a'x-b'y+y'(FF')y+ \sum_{i=1}^n(d_iy_i)^2\label{eq:unconstrained_obj}\\
	\text{s.t.}\;&y_i(1-x_i)=0, \ \   i\in N\label{eq:unconstrained_compl}\\
	&x\in \{0,1\}^N, y\in\R_+^N
	\end{alignat}
\end{subequations}
where $F\in \R_+^{n\times r}$ with $r<n$ and $e$ is an $n$-dimensional vector of ones, $a,b,d\in \R_+^N$. We test two classes of instances, \texttt{general} and \texttt{positive}, where either $F$ has both positive and negative entries, or $F$ has only non-negative entries, respectively. The parameters are generated as follows -- we use the notation $Y\sim U[\ell,u]$ as ``$Y$ is generated from a continuous uniform distribution between $\ell$ and $u$":
\begin{description} 
	\item[$F$] Matrix $F=EG$ where $E\in \R_+^{n\times r}$ is an exposure matrix such that $E_{ij}=0$ with probability $0.8$ and $E_{ij}\sim U[0,1]$ otherwise, and $G\in \R_+^{r\times r}$ such that: $G_{ij}\sim U[-1,1]$ for \texttt{general} instances, and $G_{ij}\sim U[0,1]$ for \texttt{positive} instances. 
	\item[$d$] Let $\delta$ be a diagonal dominance parameter. Define $v=(1/n)\sum_{i=1}^n (FF')_{ii}$ to be the average diagonal element of $FF'$; then $d_i^2\sim U[0,v\delta]$. 
	\item[$b$] We generate entries $b_i\sim U[0.25,0.75]\times \sqrt{(FF')_{ii}+d_i^2}$. Note that if the terms $b_i$ and $((FF')_{ii}+d_i^2)$ are interpreted as the expectation and variance of a random variable,
	then expectations are approximately proportional to the standard deviations. This relation aims to avoid trivial instances, where one term dominates the other.
	\item[$a$] {Let $\alpha$ be a fixed cost parameter} and  $a_i=\alpha(e'b)/n^2$, $i\in N$, where $e$ is an $n$-dimensional vector of ones.
\end{description} 
 It is well-documented in the literature that for matrices with large diagonal dominance the perspective reformulation achieves close to $100\%$ gap improvement. Therefore, we choose a low diagonal dominance $\delta=0.01$ to generate instances hard for the perspective reformulation. In our computations we use $n=200$.

\subsubsection{Methods}\label{sec:unconstr_methods}

We test the following methods: \todo{Objective term $a'x$ is missing in the formulations below.}
\begin{description}
	\item[Basic] Problem \eqref{eq:unconstrained} formulated as 
	\begin{subequations}\label{eq:form_basic}
	\begin{align}
	\min_{y,q}\;& \|q\|_2^2+\sum_{i=1}^n(d_iy_i)^2\\
	\text{s.t.}\;&q=F'y\\
	&y\in \R_+^n,\; q\in \R^r
	\end{align}
	\end{subequations}
	Observe that the complementary constraints \eqref{eq:unconstrained_compl} are dropped from the formulation. Nonetheless, the continuous relaxation of \eqref{eq:form_basic} is the same as the formulation where big-$M$ constraints $y\leq Mx$ are added and $M\to\infty$. \todo{Need big-$M$ with $x$ variables in the formulation.}
	
	\item[Perspective] Problem \eqref{eq:unconstrained} formulated as 
	\begin{subequations}\label{eq:form_perspective}
	\begin{alignat}{2}
	\min_{x,y,p,q}\;& \|q\|_2^2+\sum_{i=1}^nd_i^2p_i\\
	\text{s.t.}\;&q=F'y\\
	&y_i^2\leq p_ix_i, &&i=1,\ldots,n\label{eq:form_perspective_rotated}\\
	&x\in \{0,1\}^n,\; y\in \R_+^n,\;&& p\in \R_+^n,\; q\in \R^r
	\end{alignat}
	\end{subequations}
	The complementary constraints \eqref{eq:unconstrained_compl} are enforced via the rotated cone constraints \eqref{eq:form_perspective_rotated} (since $d>0$). 
	\todo{I don't think we need to refer to $d > 0$}
	
	\item[Supermodular] Problem \eqref{eq:unconstrained} formulated as 
	\begin{subequations}\label{eq:form_supermodular}
	\begin{alignat}{2}
	\min_{x,y,p,t}\;& \sum_{j=1}^r t_j+\sum_{i=1}^nd_i^2p_i\\
	\text{s.t.}\;
	& \left(F_j'y\right)^2\leq t_j, &&j=1,\ldots,r\label{eq:form_supermodular_rankone}\\
	&y_i^2\leq p_ix_i, &&i=1,\ldots,n\label{eq:form_supermodular_rotated}\\
	&x\in \{0,1\}^n,\; y\in \R_+^n,\; && p\in \R_+^n,\; t\in \R_+^r,
	\end{alignat}
	\end{subequations}
	where $F_j$ denotes the $j$-th column of $F$. Additionally, lifted supermodular inequalities \eqref{eq:extendedGen} are added to strengthen the relaxations. Note that the convex relaxation of \eqref{eq:form_supermodular} without any additional inequalities is equivalent to the convex relaxation of \eqref{eq:form_perspective}. 
\end{description}

Cuts \eqref{eq:extendedGen} or \eqref{eq:extendedPos} for method \textbf{Supermodular} are added as follows:
\begin{enumerate}
	\item We solve the convex relaxation of \eqref{eq:form_supermodular} to obtain a solution $(\bar x, \bar y, \bar t)$. By default, the convex relaxation is solved with an interior point method.
	\item We find a most violated inequality \eqref{eq:extendedGen} or \eqref{eq:extendedPos} for each constraint \eqref{eq:form_supermodular_rankone} using the separation algorithm given in \S\ref{sec:separation}. Denote by $\bar\nu_j$ the rhs value of \eqref{eq:originalSpaceGen} or \eqref{eq:originalSpacePos} if sets $L$ and $U$ satisfying \eqref{eq:conditions} exist; otherwise, let $\bar\nu=-\infty$. 
	\item Let $\epsilon=10^{-3}$ be a precision parameter. Inequalities found in step (2) are added if either $\bar t_j<\epsilon$ and $(\bar \nu_j-\bar t_j)>\epsilon$; or $\bar t_j\geq \epsilon$ and $(\bar \nu_j-\bar t_j)/\bar t_j>\epsilon$. At most $r$ inequalities are added per iteration, one for each constraint \eqref{eq:form_supermodular_rankone}.
	\item This process is repeated until either no inequality is added in step (3) or max number of cuts ($3r$) is reached. 
\end{enumerate}

\subsubsection{Results}\label{sec:unconstr_results}

Tables~\ref{tab:unconstrainedGen01} and \ref{tab:unconstrainedPos01} present the results 
for \texttt{general} and \texttt{positive} instances, respectively, for $n=200$ and $\delta=0.01$.
They show, for different ranks $r$ and values of the fixed cost parameter $\alpha$, the optimal objective value (\texttt{opt}) and, for each method, the optimal objective value for the convex relaxation (\texttt{val}), the integrality relaxation gap (\texttt{gap}) computed as $\texttt{gap}=\frac{\texttt{opt}-\texttt{val}}{|\texttt{opt}|}\times 100$, the time required to solve the relaxation in seconds (\texttt{time}) and the number of cuts added (\texttt{cuts}). The optimal solutions are computed using CPLEX branch-and-bound method. The values \texttt{opt} and \texttt{val} are scaled as follows: if $\texttt{opt}\neq 0$ in a given instance, then the values are scaled so that $|\texttt{opt}|=100$; otherwise, if $\texttt{opt}= 0$, then the values are scaled so that the optimal objective value of the convex relaxation of \textbf{Basic} is $-100$ (these instances are indicated with the symbol $\dagger$). Each row corresponds to the average of five instances generated with the same parameters. 

\begin{table}[!h]
	\setlength{\tabcolsep}{4pt}
	\caption{Computational results for unconstrained problems with \texttt{general} instances.}
	\label{tab:unconstrainedGen01}	\small
	\begin{tabular}{c c c c|c c| c c }
		\hline \hline
		\multirow{2}{*}{$r$}&\multirow{2}{*}{$\alpha$}&\multirow{2}{*}{\texttt{opt}}&\multirow{2}{*}{\texttt{method}}&\multicolumn{2}{c|}{\texttt{strength}}&
		\multicolumn{2}{c}{\texttt{performance}}\\
		&&&&\texttt{val}&\texttt{gap(\%)}&\texttt{time(s)}&\texttt{cuts}\\	
		\hline
		\multirow{11}{*}{1}&\multirow{3}{*}{2}&\multirow{3}{*}{-100.0}&Basic&-100.4& 0.4 & $<$0.1 & - \\
		&&&Perspective&-100.0& 0.0 & $<$0.1 & - \\
		&&&Supermodular&-100.0& 0.0 &0.1 & 1 \\
		&&&&&&&\\
		&\multirow{3}{*}{10}&\multirow{3}{*}{-100.0}&Basic&-101.9& 1.9 & $<$0.1 & - \\
		&&&Perspective&-100.0& 0.0 & $<$0.1 & - \\
		&&&Supermodular&-100.0& 0.0 &0.1 & 1 \\
		&&&&&&&\\
			&\multirow{3}{*}{50}&\multirow{3}{*}{-100.0}&Basic&-110.3& 10.3 & $<$0.1 & - \\
		&&&Perspective&-100.0& 0.0 & $<$0.1 & - \\
		&&&Supermodular&-100.0& 0.0 &0.1 & 1 \\
		&&&&&&&\\
		\hline
		
		\multirow{11}{*}{5}&\multirow{3}{*}{2}&\multirow{3}{*}{-100.0}&Basic&-103.6& 3.6 & $<$0.1 & - \\
		&&&Perspective&-100.0& 0.0 & $<$0.1 & - \\
		&&&Supermodular&-100.0& 0.0 &$<$0.1 & $<$1 \\
		&&&&&&&\\
		&\multirow{3}{*}{10}&\multirow{3}{*}{-100.0}&Basic&-120.3& 20.3 & $<$0.1 & - \\
		&&&Perspective&-100.4& 0.4 & $<$0.1 & - \\
		&&&Supermodular&-100.3& 0.3 &0.1 & 3 \\
		&&&&&&&\\
		&\multirow{3}{*}{50}&\multirow{3}{*}{-100.0}&Basic&-296.1& 196.1 & $<$0.1 & - \\
		&&&Perspective&-107.8& 7.8 & $<$0.1 & - \\
		&&&Supermodular&-104.2& 4.2 &0.3 & 5 \\
		&&&&&&&\\
		\hline
		
			\multirow{11}{*}{10}&\multirow{3}{*}{2}&\multirow{3}{*}{-100.0}&Basic&-105.2& 5.2 & $<$0.1 & - \\
		&&&Perspective&-100.1& 0.1 & $<$0.1 & - \\
		&&&Supermodular&-100.1& 0.1 &$<$0.1 & $<$1 \\
		&&&&&&&\\
		&\multirow{3}{*}{10}&\multirow{3}{*}{-100.0}&Basic&-129.7& 29.7 & $<$0.1 & - \\
		&&&Perspective&-101.8& 1.8 & $<$0.1 & - \\
		&&&Supermodular&-101.6& 1.6 &0.1 & 2 \\
		&&&&&&&\\
		&\multirow{3}{*}{50}&\multirow{3}{*}{-100.0}&Basic&-446.2& 342.2 & $<$0.1 & - \\
		&&&Perspective&-162.4& 62.4 & $<$0.1 & - \\
		&&&Supermodular&-145.1& 45.1 &0.3 & 7 \\
		&&&&&&&\\
		\hline \hline
	\end{tabular}
\end{table}

\begin{table}[!h]
	\setlength{\tabcolsep}{4pt}
	\caption{Computational results for unconstrained problems with \texttt{positive} instances.}
	\label{tab:unconstrainedPos01}	
	\small
	\begin{tabular}{c c c c|c c| c c }
		\hline
		\multirow{2}{*}{$r$}&\multirow{2}{*}{$\alpha$}&\multirow{2}{*}{\texttt{opt}}&\multirow{2}{*}{\texttt{method}}&\multicolumn{2}{c|}{\texttt{strength}}&
		\multicolumn{2}{c}{\texttt{performance}}\\
		&&&&\texttt{val}&\texttt{gap(\%)}&\texttt{time(s)}&\texttt{cuts}\\	
		\hline
		\multirow{11}{*}{1}&\multirow{3}{*}{2}&\multirow{3}{*}{-100.0}&Basic&-100.6& 0.6 & $<$0.1 & - \\
		&&&Perspective&-100.0& 0.0 & $<$0.1 & - \\
		&&&Supermodular&-100.0& 0.0 &$<$0.1 & 1 \\
		&&&&&&&\\
		&\multirow{3}{*}{10}&\multirow{3}{*}{-100.0}&Basic&-103.1& 3.1 & $<$0.1 & - \\
		&&&Perspective&-100.0& 0.0 & $<$0.1 & - \\
		&&&Supermodular&-100.0& 0.0 &0.1 & 1 \\
		&&&&&&&\\
		&\multirow{3}{*}{50}&\multirow{3}{*}{-100.0}&Basic&-117.9& 17.9 & $<$0.1 & - \\
		&&&Perspective&-100.1& 0.1 & $<$0.1 & - \\
		&&&Supermodular&-100.0& 0.0 &0.1 & 1 \\
		&&&&&&&\\
		\hline
		
		\multirow{11}{*}{5}&\multirow{3}{*}{2}&\multirow{3}{*}{-100.0}&Basic&-105.4& 5.4 & $<$0.1 & - \\
		&&&Perspective&-100.1& 0.1 & $<$0.1 & - \\
		&&&Supermodular&-100.0& 0.0 &0.1 & 3 \\
		&&&&&&&\\
		&\multirow{3}{*}{10}&\multirow{3}{*}{-100.0}&Basic&-134.1& 34.1 & $<$0.1 & - \\
		&&&Perspective&-100.6& 0.6 & $<$0.1 & - \\
		&&&Supermodular&-100.2& 0.3 &0.1 & 5 \\
		&&&&&&&\\
		&\multirow{3}{*}{50}&\multirow{3}{*}{-100.0}&Basic&-742.1& 642.1 & $<$0.1 & - \\
		&&&Perspective&-117.0& 17.0 & $<$0.1 & - \\
		&&&Supermodular&-100.5& 0.5 &0.3 & 5 \\
		&&&&&&&\\
		\hline
		
		\multirow{11}{*}{10}&\multirow{3}{*}{2}&\multirow{3}{*}{-100.0}&Basic&-114.5& 14.5 & $<$0.1 & - \\
		&&&Perspective&-100.4& 0.4 & $<$0.1 & - \\
		&&&Supermodular&-100.3& 0.3 &$<$0.1 & 9 \\
		&&&&&&&\\
		&\multirow{3}{*}{10}&\multirow{3}{*}{-100.0}&Basic&-218.7& 118.7 & $<$0.1 & - \\
		&&&Perspective&-105.8& 5.8 & $<$0.1 & - \\
		&&&Supermodular&-101.0& 1.0 &0.6 & 10 \\
		&&&&&&&\\
		&\multirow{3}{*}{50$^\dagger$}&\multirow{3}{*}{0.0}&Basic&-100.0& - & $<$0.1 & - \\
		&&&Perspective&-8.8& - & $<$0.1 & - \\
		&&&Supermodular&$>$-0.1& - &0.3 & 10 \\
		&&&&&&&\\
		\hline
	\end{tabular}
\end{table}

\todo{The instances appear to be too simple, most with 0\% gap even trivial for the basic formulation. I think we want to show rank-one cases with large gap for perspective, but 0\% with supermodular (as expected)}
We see that in instances (both \texttt{general} and \texttt{positive}) with $r=1$ the gaps corresponding to all methods are in general low, the convex relaxations of \textbf{Supermodular} are tight with $0\%$ gap, and \textbf{Perspective} is also able to achieve gaps close to $0\%$. As the rank of $F$ increases, the gap of \textbf{Basic} increases and the simple use of  \textbf{Perspective} achieves a considerable improvement: indeed, the complementary constraints \eqref{eq:unconstrained_compl} are actually \emph{enforced} while using the perspective reformulation, and completely discarded with \textbf{Basic} (which essentially ignores the fixed costs). The additional gap improvement achieved by \textbf{Supermodular} differs depending on the type of instance: in \texttt{general} instances the additional improvement is modest, and the gap of \textbf{Supermodular} is roughly 2/3 the gap of \textbf{Perspective}; in contrast, in \texttt{positive} instances \textbf{Supermodular} achieves a substantial improvement over \textbf{Perspective}, often reducing the gaps by a factor of five or more, especially in instances with large fixed costs ($\alpha=50$). 
}

\subsection{Instances}We consider optimization problems of the form 
\begin{subequations}\label{eq:portfolio}
\begin{align}
\min_{y,x}\;& y'(FF')y+\sum_{i=1}^n(d_iy_i)^2\\
\text{s.t.}\;&e'y=1\label{eq:portfolio_budget}\\
&b'y-a'x\geq \beta\\
&y_i\leq x_i, \ \ i\in N\label{eq:portfolio_compl}\\
&x\in \{0,1\}^N, y\in \R_+^N
\end{align}
\end{subequations}
where $F\in \R_+^{n\times r}$ with $r<n$, $a,b,d\in \R_+^N$. 
We test two classes of instances, \texttt{general} and \texttt{positive}, where either $F$ has both positive and negative entries, or $F$ has only non-negative entries, respectively. Note that constraints \eqref{eq:portfolio_compl} are in fact a big-M reformulation of complementary constraint $y_i(1-x_i)=0$: indeed, constraint \eqref{eq:portfolio_budget} and $y\geq 0$ imply the upper bound $y\leq 1$. The parameters are generated as follows -- we use the notation $Y\sim U[\ell,u]$ as ``$Y$ is generated from a continuous uniform distribution between $\ell$ and $u$":
\begin{description} 
	\item[$F$] Let $\rho$ be a positive weight parameter. Matrix $F=EG$ where $E\in \R_+^{n\times r}$ is an exposure matrix such that $E_{ij}=0$ with probability $0.8$ and $E_{ij}\sim U[0,1]$ otherwise, and $G\in \R_+^{r\times r}$ such that: $G_{ij}\sim U[\rho,1]$. If $\rho\geq 0$, then matrix $F$ is guaranteed to be positive, and we refer to such instances as \texttt{positive}. Otherwise, for $\rho<0$, we refer to the instances as \texttt{general}.  
	\item[$d$] Let $\delta$ be a diagonal dominance parameter. Define $v=(1/n)\sum_{i=1}^n (FF')_{ii}$ to be the average diagonal element of $FF'$; then $d_i^2\sim U[0,v\delta]$. 
	\item[$b$] We generate entries $b_i\sim U[0.25,0.75]\times \sqrt{(FF')_{ii}+d_i^2}$. Note that if the terms $b_i$ and $((FF')_{ii}+d_i^2)$ are interpreted as the expectation and variance of a random variable,
	then expectations are approximately proportional to the standard deviations. This relation aims to avoid trivial instances, where one term dominates the other.
	\item[$a$] {Let $\alpha$ be a fixed cost parameter} and  $a_i=\alpha(e'b)/n$, $i\in N$, where $e$ is an $n$-dimensional vector of ones.
\end{description} 
It is well-documented in the literature that for matrices with large diagonal dominance the perspective reformulation achieves close to $100\%$ gap improvement. Therefore, we choose a low diagonal dominance $\delta=0.01$ to generate instances hard for the perspective reformulation. In our computations, unless stated otherwise, we use $n=200$ and $\beta=(e'b)/n$.

\newcommand{\Basic}{\texttt{Basic }}
\newcommand{\Perspective}{\texttt{Perspective }}
\newcommand{\Supermodular}{\texttt{Supermodular }}

\subsection{Methods} 
We test the following methods:
\begin{itemize}
	\item \Basic: Problem \eqref{eq:portfolio} formulated as 
	\begin{subequations}\label{eq:form_basic}
		\begin{align}
		\min\;& \|q\|_2^2+\sum_{i=1}^n(d_iy_i)^2\\
		\text{s.t.}\;&q=F'y\\
		&\eqref{eq:portfolio_budget}-\eqref{eq:portfolio_compl}\\
		&x\in \{0,1\}^n,\; y\in \R_+^n,\; q\in \R^r.
		\end{align}
	\end{subequations}
	
	\item \Perspective:  Problem \eqref{eq:portfolio} formulated as 
	\begin{subequations}\label{eq:form_perspective}
		\begin{alignat}{2}
		\min \;& \|q\|_2^2+\sum_{i=1}^nd_i^2p_i\\
		\text{s.t.}\;&q=F'y\\
		&y_i^2\leq p_ix_i, &&i=1,\ldots,n\label{eq:form_perspective_rotated}\\
		&\eqref{eq:portfolio_budget}-\eqref{eq:portfolio_compl}\\
		&x\in \{0,1\}^n,\; y\in \R_+^n,\;&& p\in \R_+^n,\; q\in \R^r.
		\end{alignat}
	\end{subequations}

	\item \Supermodular:  Problem \eqref{eq:portfolio} formulated as 
	\begin{subequations}\label{eq:form_supermodular}
		\begin{alignat}{2}
		\min \;& \sum_{j=1}^r t_j+\sum_{i=1}^nd_i^2p_i\\
		\text{s.t.}\;
		& \left(F_j'y\right)^2\leq t_j, &&j=1,\ldots,r\label{eq:form_supermodular_rankone}\\
		&y_i^2\leq p_ix_i, &&i=1,\ldots,n\label{eq:form_supermodular_rotated}\\
		&\eqref{eq:portfolio_budget}-\eqref{eq:portfolio_compl}\\
		&x\in \{0,1\}^n,\; y\in \R_+^n,\; && t\in \R_+^r,
		\end{alignat}
	\end{subequations}
	where $F_j$ denotes the $j$-th column of $F$. Additionally, lifted supermodular inequalities \eqref{eq:extendedGen} are added to strengthen the relaxations. Note that the convex relaxation of \eqref{eq:form_supermodular} without any additional inequalities is equivalent to the convex relaxation of \eqref{eq:form_perspective}. 
\end{itemize}

Cuts \eqref{eq:extendedGen} (for \texttt{general} instances) or \eqref{eq:extendedPos} (for \texttt{positive} instances) for method \Supermodular  are added as follows:
\begin{enumerate}
	\item We solve the convex relaxation of \eqref{eq:form_supermodular} to obtain a solution $(\bar x, \bar y, \bar t)$. By default, the convex relaxation is solved with an interior point method.
	\item We find a most violated inequality \eqref{eq:extendedGen} or \eqref{eq:extendedPos} for each constraint \eqref{eq:form_supermodular_rankone} using the separation algorithm given in \S\ref{sec:separation}. Denote by $\bar\nu_j$ the rhs value of \eqref{eq:originalSpaceGen} or \eqref{eq:originalSpacePos} if sets $L$ and $U$ satisfying \eqref{eq:conditions} exist; otherwise, let $\bar\nu=-\infty$. 
	\item Let $\epsilon=10^{-3}$ be a precision parameter. Inequalities found in step (2) are added if either $\bar t_j<\epsilon$ and $(\bar \nu_j-\bar t_j)>\epsilon$; or $\bar t_j\geq \epsilon$ and $(\bar \nu_j-\bar t_j)/\bar t_j>\epsilon$. At most $r$ inequalities are added per iteration, one for each constraint \eqref{eq:form_supermodular_rankone}.
	\item This process is repeated until either no inequality is added in step (3) or max number of cuts ($3r$) is reached. 
\end{enumerate}

We point out that convexification based on $X_f$ \cite{atamturk2019rank}, described in Proposition~\ref{prop:free}, is not effective with formulation \eqref{eq:form_supermodular} since  $t_j\geq (F_j'y)^2/\min\{1,e'x\}$ reduces to $t_j\geq (F_j'y)^2$
due to \eqref{eq:portfolio_budget} and \eqref{eq:portfolio_compl}.

\subsection{Results} 
Tables~\ref{tab:portfolioGen01}--\ref{tab:portfolioPos01} 
present the results for $\rho=\{-1,-0.5,-0.2,0\}$. They show, for different ranks $r$ and values of the fixed cost parameter $\alpha$, the optimal objective value (\texttt{opt}) and, for each method, the optimal objective value for the convex relaxation (\texttt{val}), the integrality gap (\texttt{gap}) computed as $\texttt{gap}=\frac{\texttt{opt}-\texttt{val}}{\texttt{opt}}\times 100$, the improvement (\texttt{imp}) of \Supermodular over \Perspective computed as 
$$
\texttt{imp}=\frac{\texttt{gap}_{\texttt{Persp.}}-\texttt{gap}_\texttt{Supermod.}}{\texttt{gap}_{\texttt{Persp.}}},
$$ 
the time required to solve the relaxation in seconds (\texttt{time}) and the number of cuts added (\texttt{cuts}). The optimal solutions are computed using CPLEX branch-and-bound method. The values \texttt{opt} and \texttt{val} are scaled so that, in a given instance, $\texttt{opt}=100$. Each row corresponds to the average of five instances generated with the same parameters. 

\newcommand{\rowspace}{1.2em}
\begin{table}[!h]
	\setlength{\tabcolsep}{4pt}
	\small
	\caption{Computational results for  \texttt{general} instances, $\rho=-1$.}
	\label{tab:portfolioGen01}	\small
	\begin{tabular}{c c c c|c c c | c c }
		\hline
		\multirow{2}{*}{$r$}&\multirow{2}{*}{$\alpha$}&\multirow{2}{*}{\texttt{opt}}&\multirow{2}{*}{\texttt{method}}&\multicolumn{3}{c|}{\texttt{strength}}&
		\multicolumn{2}{c}{\texttt{performance}}\\
		&&&&\texttt{val}&\texttt{gap(\%)}&\texttt{imp(\%)}&\texttt{time(s)}&\texttt{cuts}\\	
		\hline \hline
		\multirow{9.6}{*}{1}&\multirow{3}{*}{2 }&\multirow{3}{*}{100.0}&\Basic&92.5& 7.5& & $<$0.1 & - \rule{0mm}{\rowspace}\\
		&&&\Perspective&98.4& 1.6 && $<$0.1 & - \\
		&&&\Supermodular&100.0& 0.0&100.0 &0.1 & 1 \\ 
		&\multirow{3}{*}{10}&\multirow{3}{*}{100.0}&\Basic&82.9& 17.1 && $<$0.1 & -  \rule{0mm}{\rowspace} \\
		&&&\Perspective&90.9& 9.1 && $<$0.1 & - \\
		&&&\Supermodular&100.0& 0.0&100.0 &0.1 & 1 \\
		&\multirow{3}{*}{50}&\multirow{3}{*}{100.0}&\Basic&61.7& 38.3 && $<$0.1 & - \rule{0mm}{\rowspace} \\
		&&&\Perspective&65.4& 34.6 && $<$0.1 & - \\
		&&&\Supermodular&94.3& 5.7 &83.5&0.1 & 1 \\
		\hline
		
		\multirow{9.6}{*}{5}&\multirow{3}{*}{2}&\multirow{3}{*}{100.0}&\Basic&88.3& 11.7 && $<$0.1 & - \rule{0mm}{\rowspace} \\
		&&&\Perspective&96.5& 3.5 && $<$0.1 & - \\
		&&&\Supermodular&97.7& 2.3 &34.3&0.1 & 3 \\
		&\multirow{3}{*}{10}&\multirow{3}{*}{100.0}&\Basic&69.7& 30.3 && $<$0.1 & - \rule{0mm}{\rowspace}\\
		&&&\Perspective&80.5& 19.5 && $<$0.1 & - \\
		&&&\Supermodular&88.5& 11.5 &41.0&0.3 & 4 \\
		&\multirow{3}{*}{50}&\multirow{3}{*}{100.0}&\Basic&41.8& 58.2 && $<$0.1 & - \rule{0mm}{\rowspace}\\
		&&&\Perspective&46.6& 53.4 && $<$0.1 & - \\
		&&&\Supermodular&68.1& 31.9 &40.3&0.6 & 5 \\
		\hline
		
		\multirow{9.6}{*}{10}&\multirow{3}{*}{2}&\multirow{3}{*}{100.0}&\Basic&87.1& 12.9 && $<$0.1 & - \rule{0mm}{\rowspace}\\
		&&&\Perspective&95.6& 4.4 && $<$0.1 & - \\
		&&&\Supermodular&95.8& 4.2&4.5 &0.1 & 2 \\
		&\multirow{3}{*}{10}&\multirow{3}{*}{100.0}&\Basic&62.0& 38.0 && $<$0.1 & - \rule{0mm}{\rowspace} \\
		&&&\Perspective&72.9& 27.1 && $<$0.1 & - \\
		&&&\Supermodular&76.1& 23.9 &11.8&0.7 & 7 \\
		&\multirow{3}{*}{50}&\multirow{3}{*}{100.0}&\Basic&27.4& 72.6 && $<$0.1 & - \rule{0mm}{\rowspace} \\
		&&&\Perspective&30.9& 69.1 && $<$0.1 & - \\
		&&&\Supermodular&40.4& 59.6 &13.7&1.0 & 12 \\
		\hline \hline
	\end{tabular}
\end{table}

\begin{table}[!h]
	\setlength{\tabcolsep}{4pt}
		\small
	\caption{Computational results for \texttt{general} instances, $\rho=-0.5$.}
	\label{tab:portfolioGen005}	\small
	\begin{tabular}{c c c c|c c c| c c }
		\hline
		\multirow{2}{*}{$r$}&\multirow{2}{*}{$\alpha$}&\multirow{2}{*}{\texttt{opt}}&\multirow{2}{*}{\texttt{method}}&\multicolumn{3}{c|}{\texttt{strength}}&
		\multicolumn{2}{c}{\texttt{performance}}\\
		&&&&\texttt{val}&\texttt{gap(\%)}&\texttt{imp(\%)}&\texttt{time(s)}&\texttt{cuts}\\	
		\hline \hline
		\multirow{9.6}{*}{1}&\multirow{3}{*}{2}&\multirow{3}{*}{100.0}&\Basic&92.5& 7.5 && $<$0.1 & - \rule{0mm}{\rowspace} \\
		&&&\Perspective&98.3& 1.7 && $<$0.1 & - \\
		&&&\Supermodular&100.0& 0.0&100.0 &0.1 & 1 \\
		&\multirow{3}{*}{10}&\multirow{3}{*}{100.0}&\Basic&82.9& 17.1 && $<$0.1 & - \rule{0mm}{\rowspace}\\
		&&&\Perspective&91.0& 9.1 && $<$0.1 & - \\
		&&&\Supermodular&99.9& 0.1 &98.9&0.1 & 1 \\
		&\multirow{3}{*}{50}&\multirow{3}{*}{100.0}&\Basic&61.7& 38.3 && $<$0.1 & - \rule{0mm}{\rowspace} \\
		&&&\Perspective&65.3& 34.7 && $<$0.1 & - \\
		&&&\Supermodular&94.2& 5.8 &83.3&0.1 & 1 \\
		\hline
		
		\multirow{9.6}{*}{5}&\multirow{3}{*}{2}&\multirow{3}{*}{100.0}&\Basic&91.5& 8.5 && $<$0.1 & -  \rule{0mm}{\rowspace} \\
		&&&\Perspective&96.5& 3.5 && $<$0.1 & - \\
		&&&\Supermodular&98.1& 1.9&45.7 &0.3 & 4 \\
		&\multirow{3}{*}{10}&\multirow{3}{*}{100.0}&\Basic&76.4& 23.6 && $<$0.1 & -  \rule{0mm}{\rowspace} \\
		&&&\Perspective&83.1& 16.9 && 0.1 & - \\
		&&&\Supermodular&92.7& 7.3 &56.8&0.4 & 4 \\
		&\multirow{3}{*}{50}&\multirow{3}{*}{100.0}&\Basic&52.1& 47.9 && $<$0.1 & -  \rule{0mm}{\rowspace} \\
		&&&\Perspective&55.1& 44.9 && $<$0.1 & - \\
		&&&\Supermodular&79.0& 21.0&53.2 &0.4 & 4 \\
		\hline
		
		\multirow{9.6}{*}{10}&\multirow{3}{*}{2}&\multirow{3}{*}{100.0}&\Basic&89.7& 10.3 && $<$0.1 & -  \rule{0mm}{\rowspace}  \\
		&&&\Perspective&93.3& 6.7 &&0.1 & - \\
		&&&\Supermodular&94.7& 5.3 &20.9&1.6 & 10 \\
		&\multirow{3}{*}{10}&\multirow{3}{*}{100.0}&\Basic&69.5& 30.5 && $<$0.1 & -  \rule{0mm}{\rowspace} \\
		&&&\Perspective&73.2& 26.8 && $<$0.1 & - \\
		&&&\Supermodular&81.4& 18.6 &30.6&2.4 & 11 \\
		&\multirow{3}{*}{50}&\multirow{3}{*}{100.0}&\Basic&38.3& 61.7 && $<$0.1 & - \rule{0mm}{\rowspace}  \\
		&&&\Perspective&39.6& 60.4 && $<$0.1 & - \\
		&&&\Supermodular&54.5& 45.5 &24.7&2.2 & 16 \\
		\hline \hline
	\end{tabular}
\end{table}

\begin{table}[!h]
	\setlength{\tabcolsep}{4pt}
	\caption{Computational results for \texttt{general} instances, $\rho=-0.2$.}
	\label{tab:portfolioGen002}	\small
	\begin{tabular}{c c c c|c c c| c c }
		\hline
		\multirow{2}{*}{$r$}&\multirow{2}{*}{$\alpha$}&\multirow{2}{*}{\texttt{opt}}&\multirow{2}{*}{\texttt{method}}&\multicolumn{3}{c|}{\texttt{strength}}&
		\multicolumn{2}{c}{\texttt{performance}}\\
		&&&&\texttt{val}&\texttt{gap(\%)}&\texttt{imp(\%)}&\texttt{time(s)}&\texttt{cuts}\\	
		\hline \hline
		\multirow{9.6}{*}{1}&\multirow{3}{*}{2}&\multirow{3}{*}{100.0}&\Basic&92.5& 7.5 && $<$0.1 & -  \rule{0mm}{\rowspace}  \\
		&&&\Perspective&98.3& 1.7 && $<$0.1 & - \\
		&&&\Supermodular&100.0& 0.0&100.0 &0.1 & 1 \\
		&\multirow{3}{*}{10}&\multirow{3}{*}{100.0}&\Basic&82.9& 17.1 && $<$0.1 & -  \rule{0mm}{\rowspace}  \\
		&&&\Perspective&90.9& 9.1 && $<$0.1 & - \\
		&&&\Supermodular&99.9& 0.1&98.9 &0.1 & 1 \\
		&\multirow{3}{*}{50}&\multirow{3}{*}{100.0}&\Basic&61.7& 38.3 && $<$0.1 & -  \rule{0mm}{\rowspace}  \\
		&&&\Perspective&65.4& 34.6 && $<$0.1 & - \\
		&&&\Supermodular&94.2& 5.8 &83.2&0.1 & 1 \\
		\hline
		
		\multirow{9.6}{*}{5}&\multirow{3}{*}{2}&\multirow{3}{*}{100.0}&\Basic&93.8& 6.2 && $<$0.1 & -  \rule{0mm}{\rowspace}  \\
		&&&\Perspective&96.3& 3.7 && $<$0.1 & - \\
		&&&\Supermodular&98.9& 1.1 &70.3&0.6 & 5 \\
		&\multirow{3}{*}{10}&\multirow{3}{*}{100.0}&\Basic&79.6& 20.4 && $<$0.1 & -  \rule{0mm}{\rowspace}  \\
		&&&\Perspective&82.1& 17.9 && 0.1 & - \\
		&&&\Supermodular&93.8& 6.2 &65.4&0.8 & 6 \\
		&\multirow{3}{*}{50}&\multirow{3}{*}{100.0}&\Basic&57.6& 42.4 && $<$0.1 & -  \rule{0mm}{\rowspace}  \\
		&&&\Perspective&58.7& 41.3 && $<$0.1 & - \\
		&&&\Supermodular&84.3& 15.7 &62.0&0.6 & 5 \\
		\hline
		
		\multirow{9.6}{*}{10}&\multirow{3}{*}{2}&\multirow{3}{*}{100.0}&\Basic&93.1& 6.9 && $<$0.1 & -  \rule{0mm}{\rowspace}  \\
		&&&\Perspective&94.9& 5.1 &&$<$0.1 & - \\
		&&&\Supermodular&97.6& 2.4 &52.9&6.6 & 14 \\
		&\multirow{3}{*}{10}&\multirow{3}{*}{100.0}&\Basic&77.8& 22.2 && $<$0.1 & -  \rule{0mm}{\rowspace}  \\
		&&&\Perspective&79.3& 20.7 && $<$0.1 & - \\
		&&&\Supermodular&89.7& 10.3 &50.2&3.0 & 12 \\
		&\multirow{3}{*}{50}&\multirow{3}{*}{100.0}&\Basic&56.9& 43.1 && $<$0.1 & -  \rule{0mm}{\rowspace}  \\
		&&&\Perspective&57.5& 42.5 && $<$0.1 & - \\
		&&&\Supermodular&76.9& 23.1&45.6 &10.3 & 20 \\
		\hline \hline
	\end{tabular}
\end{table}

\begin{table}[!h]
	\setlength{\tabcolsep}{4pt}
	\caption{Computational results for \texttt{positive} instances, $\rho=0$.}
	\label{tab:portfolioPos01}	\small
	\begin{tabular}{c c c c|c c c| c c }
		\hline \hline
		\multirow{2}{*}{$r$}&\multirow{2}{*}{$\alpha$}&\multirow{2}{*}{\texttt{opt}}&\multirow{2}{*}{\texttt{method}}&\multicolumn{3}{c|}{\texttt{strength}}&
		\multicolumn{2}{c}{\texttt{performance}}\\
		&&&&\texttt{val}&\texttt{gap(\%)}&\texttt{imp(\%)}&\texttt{time(s)}&\texttt{cuts}\\	
		\hline
		\multirow{9.6}{*}{1}&\multirow{3}{*}{2}&\multirow{3}{*}{100.0}&\Basic&92.5& 7.5 && $<$0.1 & -  \rule{0mm}{\rowspace}   \\
		&&&\Perspective&98.3& 1.7 && $<$0.1 & - \\
		&&&\Supermodular&100.0& 0.0&100.0 &0.1 & 2 \\
		&\multirow{3}{*}{10}&\multirow{3}{*}{100.0}&\Basic&82.9& 17.1 && $<$0.1 & -  \rule{0mm}{\rowspace}   \\
		&&&\Perspective&91.0& 9.0 && $<$0.1 & - \\
		&&&\Supermodular&99.9& 0.1 &98.9&0.1 & 2 \\
		&\multirow{3}{*}{50}&\multirow{3}{*}{100.0}&\Basic&61.7& 38.3 && $<$0.1 & -  \rule{0mm}{\rowspace}   \\
		&&&\Perspective&65.3& 34.7 && $<$0.1 & - \\
		&&&\Supermodular&94.2& 5.8 &83.3&0.1 & 2 \\
		\hline
		
		\multirow{9.6}{*}{5}&\multirow{3}{*}{2}&\multirow{3}{*}{100.0}&\Basic&94.1& 5.9 && $<$0.1 & -  \rule{0mm}{\rowspace}   \\
		&&&\Perspective&96.2& 3.8 && $<$0.1 & - \\
		&&&\Supermodular&98.7& 1.3 &65.8&0.2 & 10 \\
		&\multirow{3}{*}{10}&\multirow{3}{*}{100.0}&\Basic&80.4& 19.6 && $<$0.1 & -  \rule{0mm}{\rowspace}   \\
		&&&\Perspective&82.4& 17.6 && $<$0.1 & - \\
		&&&\Supermodular&93.4& 6.6 &65.2&0.2 & 10 \\
		&\multirow{3}{*}{50}&\multirow{3}{*}{100.0}&\Basic&65.6& 34.4 && $<$0.1 & -  \rule{0mm}{\rowspace}   \\
		&&&\Perspective&66.7& 33.3 && $<$0.1 & - \\
		&&&\Supermodular&90.9& 9.1 &72.7&0.2 & 10 \\
		\hline
		
		\multirow{9.6}{*}{10}&\multirow{3}{*}{2}&\multirow{3}{*}{100.0}&\Basic&94.0& 6.0 && $<$0.1 & -  \rule{0mm}{\rowspace}   \\
		&&&\Perspective&95.5& 4.5 && $<$0.1 & - \\
		&&&\Supermodular&97.6& 2.4 &51.1&0.6 & 20 \\
		&\multirow{3}{*}{10}&\multirow{3}{*}{100.0}&\Basic&83.1& 16.9 && $<$0.1 & -  \rule{0mm}{\rowspace}   \\
		&&&\Perspective&84.4& 15.6 && $<$0.1 & - \\
		&&&\Supermodular&93.2& 6.8 &56.4&0.6 & 20 \\
		&\multirow{3}{*}{50}&\multirow{3}{*}{100.0}&\Basic&66.0& 34.0 && $<$0.1 & -  \rule{0mm}{\rowspace}   \\
		&&&\Perspective&66.7& 33.3 && $<$0.1 & - \\
		&&&\Supermodular&82.6& 17.4 &47.7&0.6 & 20 \\
		\hline \hline
	\end{tabular}
\end{table}

First note that \Perspective achieves only a very modest improvement over \Basic due to the low diagonal dominance parameter $\delta=0.01$. 
We also point out that instances with smaller positive weight $\rho$ have weaker natural convex relaxations, i.e., \Basic has larger gaps -- a similar phenomenon was observed in \cite{frangioni2018decompositions}. 

The relative performance of all methods in rank-one instances, $r=1$, is virtually identical regardless of the value of the positive weight parameter $\rho$. In particular \Supermodular substantially improves upon \Basic and \Perspective: it achieves $0\%$ gaps in instances with $\alpha\leq 10$, and reduces to gap from 35\% to 6\% in instances with $\alpha=50$.

In instances with $r\geq 5$, the relative performance of \Supermodular depends on the positive weight parameter $\rho$: for larger values of $\rho$, more cuts are added and \Supermodular results in higher quality formulations. For example, in instances with $r=5$, $\alpha=50$, the improvements achieved by \Supermodular are 40.3\% ($\rho=-1$), 53.2\% ($\rho=-0.5$), 62.0\% ($\rho=-0.2$) and 72.7\% ($\rho=0$). Similar behavior can be observed for other combinations of parameters with $r\geq 5$. 

Our interpretation of the dependence of $\rho$ in the strength of the formulation is as follows. For instances with small values of $\rho$, it is possible to reduce the systematic risk of the portfolio $y'(FF')y$ close to zero due to negative correlations, i.e., achieve ``perfect hedge" although it may be unrealistic in practice. In such instances, the idiosynctratic risk $\sum_{i=1}^n (d_iy_i)^2$ and constraints \eqref{eq:portfolio_budget}--\eqref{eq:portfolio_compl}, which limit diversification, are the most important components behind the portfolio variance. In contrast, as $\rho$ increases, it is increasingly difficult to reduce the systematic risk (and altogether impossible for $\rho\geq 0$). Thus, in such instances, the systematic risk $y'(FF')y$ accounts for the majority of the variance of the portfolio. Thus, the lifted supermodular inequalities, which exploit the structure induced by the systematic risk, are particularly effective in the later class of instances. 

Figure~\ref{fig:resultsGap} 
depicts the integrality gap of different formulations as a function of rank for instances with $\rho=0$. We see that \Supermodular achieves large ($> 70\%$) improvement over \Perspective especially in the challenging low-rank settings. The improvement is significant (44\%) also for high-rank settings with $r=35$.   

\begin{figure}[!h]		
	\includegraphics[width=0.85\textwidth,trim={10.8cm 5.8cm 10.8cm 5.8cm},clip]{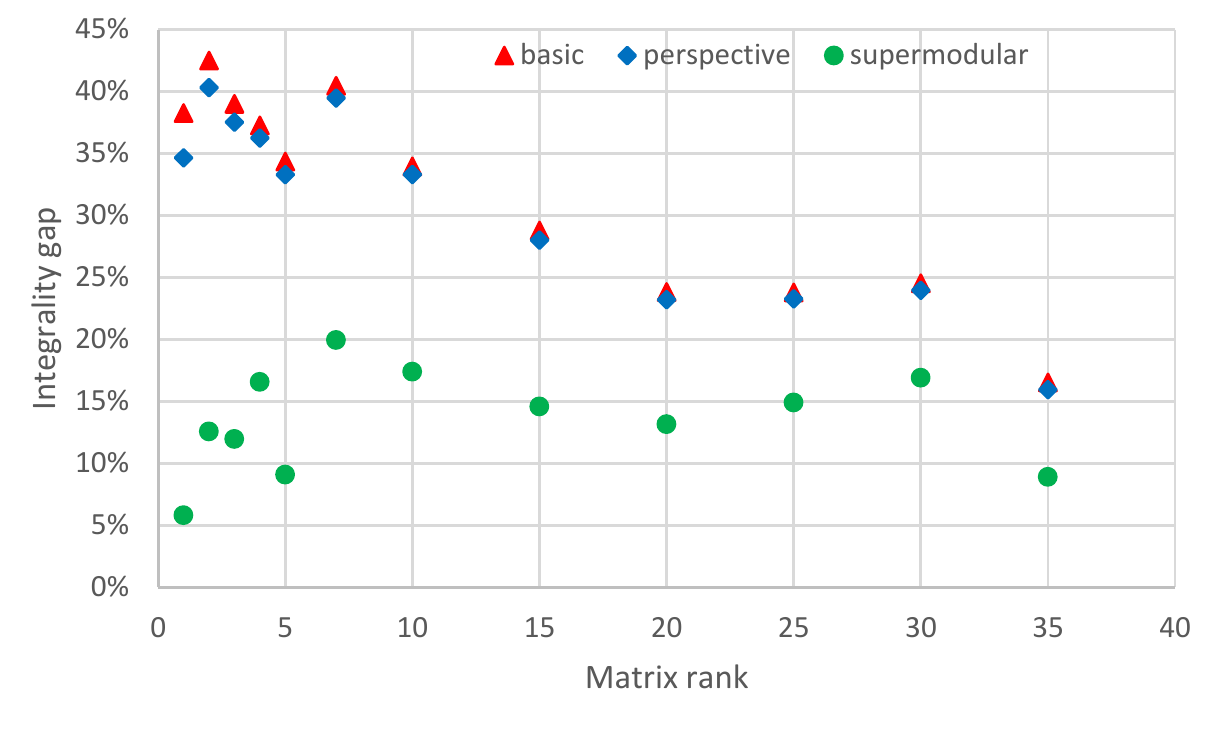}
	\caption{Integrality gap vs matrix rank ($\alpha=50$, $\delta=0.01$, $\rho=0$).}
	\label{fig:resultsGap}
\end{figure}

Finally, to evaluate the computational burden associated with the formulations, we plot in Figure~\ref{fig:resultsSize} the time in seconds (in a logarithmic scale) require the solve the convex relaxations of each method for different dimensions $n$. Each point in Figure~\ref{fig:resultsSize} corresponds to an average of 15 portfolio optimization instances generated with parameters $r=10$, $\delta=0.01$ and $\alpha\in\{2,10,50\}$ (5 instances for each value of $\alpha$).  The time for \Supermodular includes the total time used to generate cuts and solving the convex relaxations many times.

\begin{figure}[!h]		
		\subfloat[\texttt{General} instances with $\rho=-1$.]{\includegraphics[width=0.49\textwidth,trim={11cm 5.8cm 11cm 5.8cm},clip]{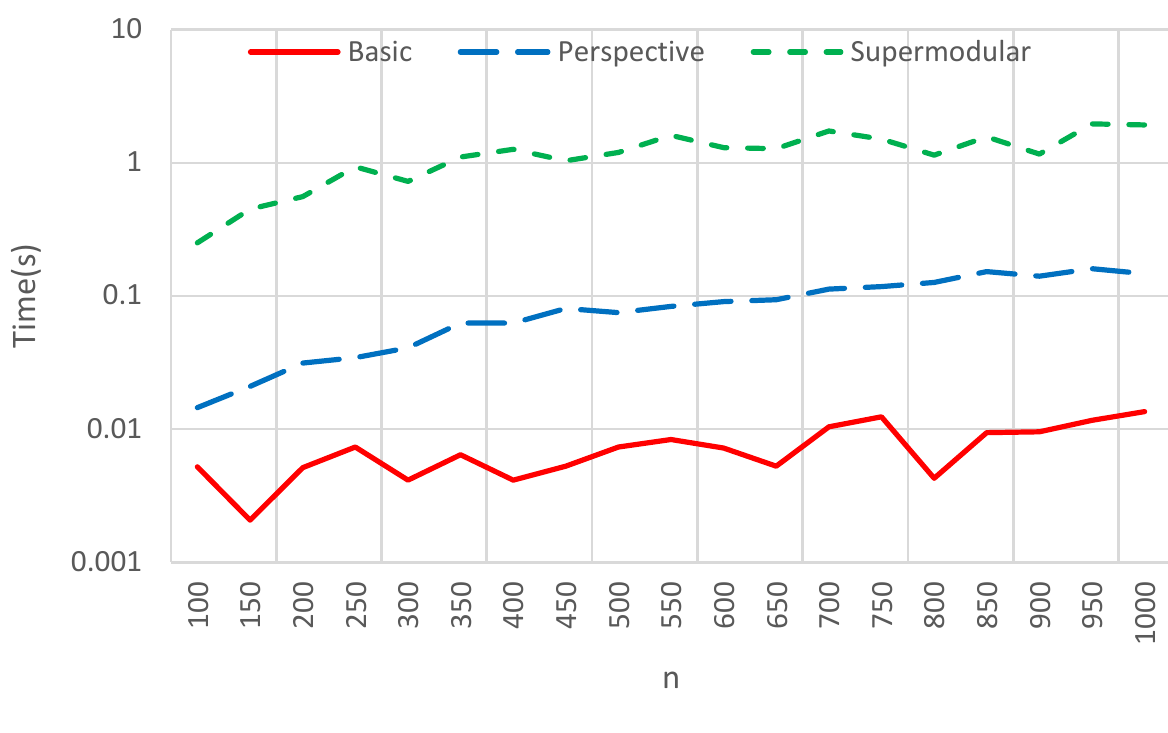}}\ \hfill
	\subfloat[\texttt{Positive} instances with $\rho=0$.]{\includegraphics[width=0.49\textwidth,trim={11cm 5.8cm 11cm 5.8cm},clip]{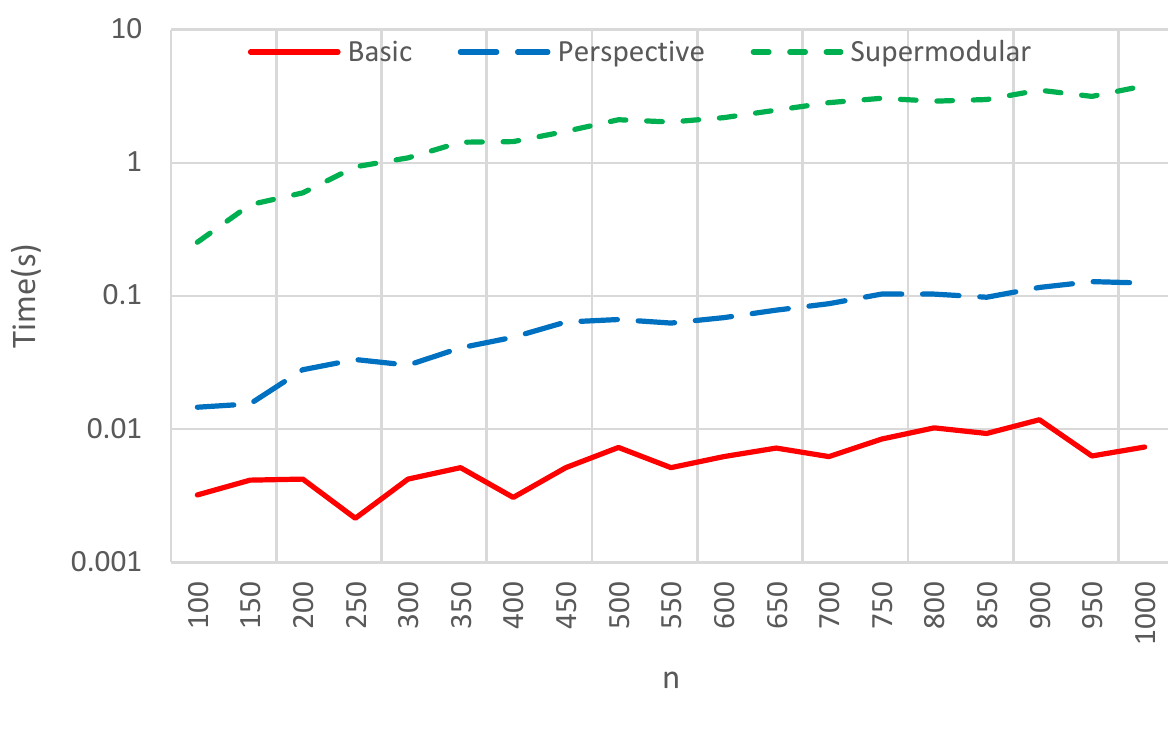}}
		\caption{Solution time  vs problem dimension ($r=10$, $\delta=0.01$).}
		\label{fig:resultsSize}
\end{figure}

We see that, in general, formulation \Basic is an order-of-magnitude faster than \texttt{Perspective}, which in turn is an order-of-magnitude faster than \texttt{Supermodular}. Nonetheless, the computation times for \Supermodular are adequate for many applications, solving instances with $n=1,000$ on average under four seconds. 

Contrary to expectations, \Supermodular is faster for \texttt{general} instances than for \texttt{positive} instances, despite the larger and more complex inequalities \eqref{eq:extendedGen} used for the \texttt{general} case; for $n=1,000$, \Supermodular runs in 1.9 seconds in \texttt{general} instances versus 3.8 seconds in \texttt{positive} instances.	This counter-intuitive behavior is explained by the number of cuts added, as several more violated cuts are found in instances with large values of $\rho$, leading to larger convex formulations and the need to resolve them more times; for $n=1,000$, 20 cuts are added in each instance with $\rho=0$, whereas on average only 3.7 cuts are added in instances with $\rho=-1$.  

The computation times are especially promising for tackling large-scale quadratic optimization problems with indicators, where alternatives to constructing strong convex relaxations (often based on decomposition of matrix $FF'+D$ into lower-dimensional terms) may not scale. For example, \citet{frangioni2018decompositions} solve convex relaxations of instances up to $n=50$, \citet{hga:2x2} solve relaxations for instances up to $n=150$, and \citet{atamturk2018strong} report that solving the convex relaxation of quadratic instances with $n=200$ requires up to 1,000 seconds. All of these methods require adding $O(n^2)$ variables and constraints to the formulations to achieve strengthening. In contrast,  the supermodular inequalities \eqref{eq:extendedGen} and \eqref{eq:extendedPos} yield formulations with $O(nr)$ additional variables and constraints, which can be solved efficiently even if $n$ is large provided that the rank $r$ is sufficiently small: in our computations, instances with $r=10$ and $n=1,000$ are solved in under four seconds. Nonetheless, as discussed in the next section, even if the convex relaxations can be solved easily, incorporating the proposed convexification in branch-and-bound methods may required tailored implementations, not supported by current off-the-shelf branch-and-bound solvers.

\subsection{On the performance with off-the-shelf branch-and-bound solvers}\label{sec:branchBound}
We also experimented with solving the formulations \Supermodular obtained after adding cuts with CPLEX branch-and-bound algorithm. However, note that inequalities \eqref{eq:extendedGen} and, to a lesser degree, inequalities \eqref{eq:extendedPos}, involve several ratios that can result in division by $0$ -- from the proof of Proposition~\ref{prop:correct}, we see that this in fact the case in many scenarios. Therefore, while we did not observe any particular numerical difficulties when solving the convex relaxations (via interior point methods), in a small subset of the instances we observed that the branch-and-bound method (based on linear outer approximations) resulted in numerical issues leading to incorrect solutions. 

Table~\ref{tab:pathological} reports the results on the two instances that exhibiting such pathological behavior. It shows, for each instance and method and different CPLEX settings, the bounds on the optimal solution obtained reported by CPLEX when solving the convex relaxation via interior point methods (\texttt{barrier}, corresponding to a lower bound), and lower and upper bounds reported by running the branch-and-bound algorithm for one hour. We do not scale the solutions obtained in Table~\ref{tab:pathological}. The tested settings are default CPLEX (\texttt{def}), default CPLEX with numerical emphasis enabled (\texttt{+num}), and CPLEX with numerical emphasis enabled and presolve and CPLEX cuts disabled (\texttt{+num-pc}).

\begin{table}[!h]
	\setlength{\tabcolsep}{2pt}
	\caption{Examples of pathological behavior in branch-and-bound.}
	\label{tab:pathological}	
	\begin{tabular}{l | c c | c c c   }
		\hline
		\multirow{2}{*}{\texttt{instance}} & \multirow{2}{*}{\texttt{method}}& \multirow{2}{*}{\texttt{setting}}&\multicolumn{3}{c}{\texttt{bounds}}\\
		&&&\texttt{barrier}&\texttt{lb\_bb}&\texttt{ub\_bb}\\
		\hline
		&\Perspective&def&0.0202&0.0942&0.0942\\
		&&&&&\\
		200-10-1.0-0.01&\multirow{3}{*}{\Supermodular}&def&0.0243&0.1249&0.1249\\
		-50.0-1-1-103$^\dagger$&&+num&0.0243&0.1078&0.1078\\
		&&+num-pc&0.0243&0.0942&0.0942\\
		&&&&&\\
		\hline

			&\Perspective&def&0.1950&0.4849&0.4849\\
		&&&&&\\
		200-10-1.0-0&\multirow{3}{*}{\Supermodular}&def&0.2471&0.4849&0.4849\\
		-50.0-1-1-104$^{\dagger\dagger}$&&+num&0.2471&0.4849&0.4849\\
		&&+num-pc&0.2471&0.5209&0.6629\\
		&&&&&\\
\hline
\multicolumn{5}{l}{\footnotesize $\dagger$ \texttt{General} portfolio instance with $\rho=-1$, $n=200$, $r=10$, $\delta=0.01$, $\alpha=50$}\\
\multicolumn{5}{l}{\footnotesize $\dagger\dagger$ \texttt{General} portfolio instance with $\rho=-1$, $n=200$, $r=10$, $\delta=0$, $\alpha=50$}
	\end{tabular}
\end{table}

In the first instance shown in Table~\ref{tab:pathological}, when using \Supermodular with the default CPLEX settings, the solution reported is worse than the optimal solution by 30\%. By enabling the numerical emphasis option, the solution improves but is still 10\% worse than the solution reported by \texttt{Perspective}. Nonetheless, if presolve and CPLEX cuts are disabled, then both solutions coincide. The second instance shown in Table~\ref{tab:pathological} exhibits the opposite behavior: when used with the default settings, independently of the numerical emphasis, the solutions obtained by \Perspective and \Supermodular coincide; however, if presolve and CPLEX cuts are disabled, then the lower bound obtained after one hour of branch-and-bound with the \Supermodular method already precludes finding the correct solution. We point out that pathological behavior of conic quadratic branch-and-bound solvers 
have been observed in the past for other nonlinear mixed-integer problems with a large number of variables, see for example \cite{atamturk2018strong,atamturk2018signal,frangioni2018decompositions,gomez2018strong}. 


\ignore{

\subsection{Future work}\label{sec:future} We now discuss three pointers for future research. \todo{These are good points but I am afraid they may encourage referees to ask for them for this paper.}

First, in general, there are infinitely many decompositions of a function $y'Qy$ as a sum of rank-one functions. In fact, in \cite{atamturk2019rank} the authors use weaker inequalities that do not account for negativity constraints, and that if used directly, as done in our computations, would likely have no impact at all. However, by finding an \emph{optimal} decomposition, the authors achieve substantial improvement over perspective reformulation approaches in real best subset selection instances. We find it highly encouraging that the inequalities reported in this paper achieve good improvements  (especially in \texttt{positive} instances) without any effort in finding decompositions. Thus, using a good decomposition may result in considerable improvements over simple perspective relaxations. The problem of decomposing matrix $Q$ is likely to involve semi-definite programming approaches or heuristics. We refer the reader to \cite{atamturk2019rank,atamturk2018signal,dong2015regularization,frangioni2018decompositions} for further information on this topic. 

Second, when tackling large quadratic optimization problems, it is critical to exploit the low-rank structure of the matrix $Q$ (as done in our computations) to speed-up solutions times. However, most decompositions techniques are unable to preserve or exploit this low-rank structure, and result in formulations with $O(n^2)$ additional variables and constraints that solvers are unable to solve efficiently for large $n$. For example, \citet{frangioni2018decompositions} use instances with a most $n=50$, and \citet{atamturk2018strong} report that just solving the convex relaxation of quadratic instances with $n=200$ requires up to 1,000 seconds. In contrast, when implemented as done in this paper (without decompositions), the supermodular inequalities \eqref{eq:extendedGen} and \eqref{eq:extendedPos} yield formulations with $O(nr)$ additional variables and constraints, which can be solved efficiently even if $n$ is large provide that the rank $r$ is sufficiently small: in our computations, problems with $r=10$ and $n=1,000$ are solved in under four seconds. Thus, the inequalities discussed here may be critical to tackling large-scale quadratic optimization problems with indicators that cannot be handled via existing decomposition approaches. \todo{The second point is more of an observation than future direction.}

Third, while solving the convex relaxations resulting from strong formulations for nonseparable quadratic and, more generally, nonlinear optimization problems with indicators, can be accomplished rather efficiently via interior point methods, the integration of such formulations with branch-and-bound methods remains a difficult yet important challenge, see \cite{atamturk2018signal,frangioni2018decompositions,gomez2018strong} for other examples. 
}

\section{Conclusions}\label{sec:conclusions}

In this paper we describe the convex hull of the epigraph of a rank-one quadratic functions with indicator variables. In order to do so, we first describe this convex hull of a underlying supermodular set function in a lower-dimensional space, and then maximally lift the resulting facets into nonlinear inequalities in the original space of variable. The approach is broadly applicable, as most of the existing results concerning convexifications of convex quadratic functions with indicator variables can be obtained in this way, as well as several well-known classes of facet-defining inequalities for mixed-integer linear problems.

\section*{Acknowledgments}
Alper Atamt\"urk is supported, in part, by NSF grant 1807260 and ONR grant 12951270.
A G\'omez is supported, in part, by NSF grants 1818700 and 1930582.

\bibliographystyle{spbasic}      
\bibliography{Bibliography}

\begin{appendix}
\section{}\label{sec:appendix}

\begin{proof}[Proof of Proposition~\ref{prop:validPos}]
	In order to solve problem \eqref{eq:liftingPos} we introduce an  auxiliary variable $\gamma\in \R_+$ such that  $\gamma=\max_\alpha(S^+)$. Then, inequality \eqref{eq:liftingPos} reduces to
	\begin{subequations}\label{eq:liftingOriginalPos}
		\begin{align}
		t\geq \max_{S\subseteq N}\max_{\alpha,\gamma}& -\frac{\gamma^2}{4}-\sum\limits_{i\in N\setminus S}\frac{\left(\alpha_i^2-\gamma^2\right)_+}{4}x_i+\alpha'y \label{eq:liftingOriginalPos_obj}\\
		\text{s.t.}\;& \alpha_i\leq \gamma, \ \  i\in S\label{eq:liftingOriginalPos_gamma}\\
		&\alpha\in \R_+^N,\;\gamma\in \R_+,\label{eq:liftingOriginalPos_bounds}
		\end{align}
	\end{subequations}
	where 
	constraint \eqref{eq:liftingOriginalPos_gamma} enforces the definition of $\gamma$. 
	
	Note that there exists an optimal solution for \eqref{eq:liftingOriginalPos} were $\gamma\leq \alpha_i$ for all $i\in N$: if $\alpha_i<\gamma$ for some $i\in N$, then setting $\alpha_i=\gamma$  yields a feasible solution with improved objective value. Therefore,  $S$ is completely determined by $\gamma$ since $S=\left\{i\in N: \alpha_i\leq\gamma\right\}$.
	
	Now, let $L=\left\{i\in N:\alpha_i=\gamma \right\}$ in a solution of \eqref{eq:liftingOriginalPos}. From the discussion above, we find that \eqref{eq:liftingOriginalPos} reduces to 
	\begin{subequations}\label{eq:liftingOriginalOptSetsPos}
		\begin{align}
		t\geq \max_{\alpha,\gamma}\;& \gamma \cdot y(L)-\frac{\gamma^2}{4}\big(1-x(N\setminus L)\big)+\sum\limits_{i\in N\setminus L}\left(\alpha_iy_i-\frac{\alpha_i^2}{4}x_i\right)\\
		\text{s.t.}\;& \gamma<\alpha_i &\hspace{-3cm}\forall i\in N\setminus L\label{eq:liftingOriginalOptSetsPos_constraint}\\
		&\alpha\in \R^N,\;\gamma\in \R_+.
		\end{align}
	\end{subequations}
	Observe that for $(L,\gamma)$ to  correspond to an optimal solution, we require that $1-x(N\setminus L)\geq 0$ (otherwise $\gamma$ can be increased and set to an upper bound while improving the objective value).
	When this condition is satisfied, we find by taking derivatives of the objective and setting to 0, that $\alpha_i=2y_i/x_i$ for $i\in  N\setminus L$ and $\gamma= 2 y(L)/\big(1-x(N\setminus L)\big)$, and \eqref{eq:liftingOriginalOptSetsPos} simplifies to \eqref{eq:originalSpacePos}.
	Note however that, in general, $(\alpha,\gamma)$ may not satisfy constraints \eqref{eq:liftingOriginalOptSetsPos_constraint} for any choice of sets $L\subseteq N$. The constraints are satisfied if and only if $\gamma<\alpha_i$ for all $i\in N\setminus L$, i.e., if and only if conditions  \eqref{eq:conditionsPos_validL} are satisfied.
		
	In order for $L$ to be optimal we require condition \eqref{eq:conditionsPos_optL}, i.e., 
	$$\frac{y(L)}{1-x(N\setminus L)}\geq \frac{y_i}{x_i}. \quad \forall i\in  L.$$
	Indeed, if this condition is not satisfied for some $j\in L$, then increasing $\alpha_j$ from $\gamma=2\frac{y(L)}{1-x(N\setminus L)}\geq 2\frac{y_j}{x_j}$ to $2y_j/x_j$ (or setting it to $\beta$ if $\beta<2y_j/x_j$) results in a better objective value. 
\end{proof}

	\end{appendix}

\end{document}